\RequirePackage[final]{graphicx}
\RequirePackage[usenames,dvipsnames]{xcolor}
\documentclass[11pt, oneside, usenames, dvipsnames, svgnames, table, final]{amsart}

\usepackage{fourier}
\usepackage{longtable}
\usepackage{overpic}

\usepackage{tbjwHeader}
\usepackage{colortbl}
\usepackage{afterpage}
\usepackage{pdflscape}
\usepackage{upquote}
\usepackage{fancyvrb}

\usepackage{caption}
\newcommand{\sqbullet}{\blacksquare}

\definecolor{highlight}{rgb}{1,1,1}

\usepackage[backref=true,style=alphabetic,citestyle=alphabetic,url=false]{biblatex}
\newcommand{\benw}[2][]{\ifdraft{\todo[linecolor=Green,backgroundcolor=Green!25,bordercolor=Green,#1]{#2---Ben W.}}{}}

\newcounter{todocounter}

\newcommand{\half}{\mathrm{sign}}
\renewcommand{\inv}{\mathrm{inv}}
\newcommand{\Fix}{\operatorname{Fix}}
\newcommand{\Homeo}{\operatorname{Homeo}}
\newcommand{\Diff}{\operatorname{Diff}}
\newcommand{\sign}{\mathrm{sign}}
\newcommand{\Rep}{\operatorname{Rep}}
\newcommand{\ORep}{\operatorname{ORep}}
\newcommand{\FORep}{\operatorname{FORep}}

\newcommand{\link}{\operatorname{lk}}

\newcommand{\rtn}{\operatorname{rtn}}
\renewcommand{\Isom}{\operatorname{Isom}}

\author{Keegan Boyle}
\address{Keegan Boyle \\ Department of Mathematical Sciences\\
1290 Frenger Mall \\
MSC 3MB / Science Hall 236 \\
Las Cruces, New Mexico 88003-8001\\
United States of America.}
\email{kboyle@nmsu.edu}

\author{Nicholas Rouse}

\author{Ben Williams}
\address{Ben Williams\\ Department of Mathematics \\
University of British Columbia,
Vancouver BC \, V6T 1Z2\\
Canada}
\email{tbjw@math.ubc.ca}


\thanks {This work has been
supported by 
Discovery Grants from the Natural Sciences and Engineering Research Council of Canada, and by the Collaborative Research
Grant ``Novel Techniques in Low Dimension'' from the Pacific Institute for the Mathematical Sciences.}

\begin{document}

\title{Types of symmetries of knots}
\begin{abstract}We classify all finite group actions on knots in the 3-sphere. By geometrization, all such actions are conjugate to actions by isometries, and so we may use orthogonal representation theory to describe three cyclic and seven dihedral families of symmetries. By constructing examples, we prove that all of the cyclic and four of the dihedral families arise as symmetries of prime knots. The remaining three dihedral families apply only to composite knots. We also explain how to distinguish different types of symmetries of knots using diagrammatic or topological data.

In passing, we establish two technical results: one concerning finite cyclic or dihedral subgroups of isometries of the 3-sphere and another concerning linking numbers of symmetric knots.
\end{abstract}

\maketitle
\tableofcontents
\listoffigures
\listoftables

\ifdraft{
  \todo[inline]{Make sure we say $\rho$ ``acts'' not ``is'' a rotation, etc.}
  \todo[inline]{Add a reference to Gr\"unbaum and Shephard}
  \todo[inline]{Add a section on diagrams}
  \todo[inline]{Make less redundant}
  \todo[inline]{Fix up appendix}
\todo[inline]{Figure out if we have obstructed equivariant unknottings.}
\todo[inline]{For each $G$-representation, is there a knot that is invariant under the corresponding action, and for which $G$ is the entire symmetry group of the knot? (Many small-crossing knots seem to have extraneous symmetries making the construction of such knots not entirely trivial). Probably the Alexander module and Jones polynomial are sufficient to check this? Add this as a question / conjecture?}
\todo[inline]{It might be nice to associate the appropriate names from crystallography to the actions. I think this only applies to ones with a global fixed point. See: \url{https://en.wikipedia.org/wiki/Space_group}.}
\todo[inline]{Add global fixed-point data to the tables.}}{}

\section{Introduction}
In light of many recent papers concerning specific classes of symmetric knots (see, among others,
\cite{Grove2021,Hirasawa2022,boyle_musyt_2022,MillerA,DiPrisa2023,DiPrisa1}), it seems to us that a classification of
the symmetries of knots is overdue. Such a classification will add additional context to these recent results and we
hope will also stimulate further investigation of symmetric knots by indicating commonalities and distinctions between
different types of symmetry.

\medskip

We consider a \emph{knot} $K$ to be a smooth embedding $K : S^1 \to S^3$, frequently identified with its (oriented) image, $|K|$, and for convenience we will consider only faithful group actions. It is well known that a finite group $G$ can act faithfully on a nontrivial knot in $S^3$ if and only if $G$ is cyclic or
dihedral.
\begin{definition} \label{def:GsymmKnot}
    Let $G$ be a group. A \emph{$G$-symmetric knot} is a knot $K$ along with a faithful action $\alpha$ of $G$ by diffeomorphisms on $S^3$ for which $|K|$ is $G$-invariant.
\end{definition}
(We compare this with other notions of symmetric knot in Section \ref{sec:definition-type}).

\medskip

We classify $G$-symmetric knots into types, and show that the types we enumerate actually occur. The ``type'' of a $G$-symmetric knot consists of the information inherent in the representations $G \to \Diff(S^3)$ and $G \to \Diff(S^1)$, up to orientation-preserving reparameterization.

As a preliminary to our main theorems, we reduce actions by diffeomorphisms on knots to actions by
isometries. Geometrization, as manifested in \cite[Theorem E]{Dinkelbach2009}, and the equality of linear and
topological equivalence of representations, \cite{Cappell1999}, does much of the work. Since we view knots as oriented
and distinguish a knot from its mirror, we have to go beyond what is in the literature: the appropriate notions
 of equivalence of actions are conjugation by orientation-preserving diffeomorphism or orientation-preserving
 isometry. Adapting the literature to our needs takes up most of Sections
 \ref{sec:transv-link-numb}--\ref{sec:reduction}. We recast our definition of ``type'' in terms of linear representations in Section
 \ref{sec:simpl-defin-type}.

Our first main theorem is a complete classification of types of symmetric knots. For prime knots we distinguish three
infinite families of cyclic symmetries and four infinite families of dihedral symmetries. There are three further families that apply only to composite knots. Actions by cyclic groups of
order $2$ arise both as special cases of actions by cyclic groups and as degenerate cases of actions by dihedral groups,
so we list them separately.

\begin{theorem} \label{th:mainTheoremA}
  Suppose $G$ is a finite group and $K$ is a nontrivial $G$-symmetric knot. Then the type of the $G$-symmetric knot is
  listed in Table \ref{tab:C2Types}, \ref{tab:CnTypes} or \ref{tab:D2nTypes}. If $K$ is a prime knot, then the type is
  not \hyperlink{2R}{2R} from Table \ref{tab:C2Types} or \hyperlink{DihB}{DihB}, \hyperlink{DihD}{DihD} or \hyperlink{DihF}{DihF} from Table \ref{tab:D2nTypes}.
\end{theorem}

The proof of this theorem is the content of Sections \ref{sec:Cyc2}, \ref{sec:typesCn} and
\ref{sec:typesDn}. It follows from a study of the representations of cyclic and dihedral groups in $\Og(4)$, up to
 $\SO(4)$-conjugacy. In Sections \ref{sec:detectingDifferentCnTypes} and \ref{sec:an-enumeration-types}, we
give procedures for determining the type of the symmetry.

We also construct examples to show that all asserted types can actually arise.
\begin{theorem} \label{thm:knot_existence}
  For each entry in Tables \ref{tab:C2Types}, \ref{tab:CnTypes} and \ref{tab:D2nTypes} and every applicable value of the
  parameters $a,b$, there exists a $G$-symmetric nontrivial knot $K$ of this type. Excepting $2R$ from
  Table \ref{tab:C2Types} and DihB, DihD and DihF from Table \ref{tab:D2nTypes}, there exists a $G$-symmetric prime knot
  of the given type.
\end{theorem}
The parameters $a$ or $(a,b)$ that appear in some types may be viewed as second-order information, less significant
than the name of the type. They are defined in representation-theoretic terms, but have interpretations as linking
numbers of the knot with invariant unknotted circle.

The proof of Theorem \ref{thm:knot_existence} when $G$ is cyclic of order $2$ is by exhibiting knots of the
required type, which is done later in this introduction. For larger $G$, the proof is by restricting the symmetry of
certain families of knots to subgroups. The families of knots that we use are the torus knots, see Section
\ref{sec:torus}, and a family of amphichiral, alternating and hyperbolic knots that is constructed in Section
\ref{sec:exAmphichiral}.

One might hope that all types arise as symmetries of hyperbolic knots. We go further and make the following conjecture.
\begin{conjecture} \label{conj:allsymshyperbolic}
Let $G$ be a finite cyclic or dihedral group. For each $G$-symmetry type of a prime knot, there is a $G$-symmetric hyperbolic knot $K$ with
that type and for which $G$ is the full symmetry group of $K$.  
\end{conjecture}
The condition that $G$ is the ``full symmetry group of $K$'' can be stated in several ways. For instance, it means that an induced map $G \to \Out(\pi_1(S^3 \sm K))$
(see Section \ref{sec:relNonRigid}) is an isomorphism, or alternatively it means that any symmetric structure on a knot
equivalent to $K$, say by a group $H$, can be conjugated by means of an element $d \in \Diff^+(S^3)$ to consist of
actions by a subgroup of $G$.

The most difficult types for this conjecture seem to be \hyperlink{GFPer}{GFPer} and \hyperlink{SIFP}{SIFP}, which is to say the types containing a freely periodic symmetry. It is
possible to produce infinite families of knots having these symmetry types for all applicable values of the parameters $(a,b)$, but it is difficult to prove that the
knots in these families are hyperbolic and do not possess extraneous symmetries.

\subsection{Relation to prior work}

In 1992, \cite{Luo1992} presented a partial classification of finite group actions on knots, with particular
emphasis on the case of cyclic groups, but this classification is incomplete, essentially because
the geometrization theorem was not known at the time. We complete and extend Luo's work in a variety of
ways: by describing the free symmetries precisely, by enumerating the dihedral symmetries, by considering the group action not just
on the oriented $3$-sphere but also on the oriented knot, by proving invariance of the symmetry types under equivalence of
symmetric knots, and by showing by explicit construction that each type of symmetry we assert actually arises as a
symmetry of a prime knot.

We should also mention a paper of Gr\"{u}nbaum and Shephard, \cite{Grunbaum1985}, in which the authors determine which subgroups of $\Og(3)$ can act by isometries on a knot $K$ in $\RR^3$. Their emphasis is on the $3$-dimensional point groups, rather than the knots, so that (for instance) symmetries that fix a point on the knot are not distinguished from those that do not, and primality of knots is not considered. Nonetheless, one can see in \cite[Figure 4]{Grunbaum1985} examples of the types of symmetry we call \hyperlink{DihB}{DihB}, \hyperlink{Per}{Per}--$(1)$, \hyperlink{SIP}{SIP}--$(1)$, \hyperlink{DihD}{DihD} and \hyperlink{RRef}{RRef}--$(1)$.

Our classification directly implies the following result:
\begin{corollary}\cite[Main Theorem]{Sakumka}\label{cor:Sakuma}
A $G$-symmetric knot cannot have both a freely-periodic symmetry and an amphichiral symmetry. In particular, no hyperbolic knot is both freely periodic and amphichiral. 
\end{corollary}
Since our classification relies heavily on \cite[Theorem E]{Dinkelbach2009}, the argument
we outline was not accessible at the time Corollary \ref{cor:Sakuma} was originally proved.
The fact that we can recover this result is an advertisement for laying out a table of types of symmetry.

\begin{remark}
  Paoluzzi and Sakuma \cite[Theorem 1.1]{Paoluzzi-Sakuma} have constructed infinitely many prime knots that
  are both freely periodic and amphichiral. These knots are necessarily satellite knots. By Corollary \ref{cor:Sakuma}, these freely periodic and amphichiral
  symmetries generate an infinite subgroup of $\Diff(S^3)$ and cannot be simultaneously realized as isometries.
\end{remark}

\subsection{Symmetric diagrams}

\begin{remark} \label{rem:onDiagrams} Assume that $G$ is finite but nontrivial and that $K$ is a nontrivial $G$-symmetric knot. A knot diagram in this paper is an attempt to represent the compact $G$-invariant subset $|K|$ of $S^3$ by means of a projection to $\RR^2$. 

The most favourable case is when $G$ acts on $S^3$ by isometries leaving $2$ points $p,q \in S^3$ fixed. The $2$-dimensional sphere $X$ consisting of points equidistant from $p$ and $q$ is $G$-invariant, and since $G$
  is dihedral or cyclic, it follows that there is an invariant subset of $X$ consisting of a pair of antipodal
  points, which we call $\{0,\infty\}$. As we classify all possible actions of $G$ up to conjugacy, we will discover
  that $\{0, \infty\}$ may always be chosen to be disjoint from $K$. 

  Let $S^2$ denote the $G$-invariant sphere equidistant from $0$ and $\infty$, which contains the two fixed points $p$, $q$. 
  We may project $S^3 \sm \{0,\infty\}$ equivariantly onto $S^2$ (we disregard here the problem of triple- or
  higher-order points in the knot diagram), giving an immersion of $K$ in $S^2$ along with crossing data. Then
  projecting away from one of the fixed points, say $p$, produces an equivariant stereographic projection of $S^2 \sm \{p\}$ onto $\RR^2$. The result is a diagram of the knot in $\RR^2$ where every element of $G$ acts by a isometry of $\RR^2$ (possibly mirroring the diagram, exchanging under- and over-crossings).

We will say that a $G$-symmetric knot $K$ has a \emph{good} diagram if the conditions above hold. Examples of good diagrams include Figures \ref{fig:12a_427other}, \ref{fig:10_67}, \ref{fig:8_17}, and \ref{fig:5_2}, among others.

    For symmetric knots without good diagrams, different problems arise in depicting the knot symmetrically. In some cases some elements of $G$ act by
    double rotations without fixed points on $S^3$. If $G$ is cyclic and the generator is of this kind, then the action
    is a double rotations of $S^3 \subset \RR^4$ leaving a torus invariant, and that torus may be used to produce a knot diagram in which the symmetry is visually apparent: see Figure \ref{fig:free2period} for an example. 
    A proof that all freely-periodic knots have diagrams of this form may be found in  \cite[Theorem 1]{MR1664974}.

    One may also give a knot diagram in which one is obliged to visualize some symmetry as ``rotation around a circle", as in Figures \ref{fig:10_155} or \ref{fig:DihF}.

    In the specific case of a freely periodic symmetry of order $2$, where we may suppose that the action on $S^3$ is given by the antipodal map, there is an invariant $2$-sphere. Projection onto this sphere allows one to produce a knot diagram where the symmetry is given by the transformation $v \mapsto \frac{-v}{\|v\|}$ on $\RR^2$, composed with a mirroring of the diagram exchanging under- an over-crossings. The diagram in Figure \ref{fig:free2period} may be adjusted to be of this form by closing the braid and distributing the partial twists around the resulting diagram.\benw{Is this ok?}
    
    In other cases, no element of $G$ acts without a pair of fixed points, but there is no globally fixed pair of points. Here, one can find diagrams in which some symmetries are immediately apparent, but others do not correspond to isometries of $\RR^2$. Typical examples are given in Figures \ref{fig:12a_868} and \ref{fig:10_99}.
\end{remark}

We now give some examples of different types of symmetric prime knots. These examples are intended to explain in broad terms
what the different types are. In each case, the knot given is hyperbolic, and the symmetries described constitute the full symmetry group of the knot. 

The precise definitions of the types are given in Definition \ref{def:types} and reinterpreted
in terms of orthogonal representation theory of finite groups in Section \ref{sec:simpl-defin-type}. 

We begin with $C_2$-symmetry since we refer to this classification when describing the other types.
\subsection{Examples of \texorpdfstring{$C_2$}{C₂}-symmetric knots}
\label{sec:exampl-texorpdfstr-s}

Let $C_2 = \langle \tau \mid \tau^2 = e \rangle$ be a cyclic group of order $2$. There are $6$ types of nontrivial $C_2$-symmetric
knot. These are distinguished by the dimensions of the fixed sets of the action in $S^3$ and on the knot. 

\begin{enumerate}
\item \textit{Free 2-periodicity}, abbreviated F2P, where the action on $S^3$ (and therefore also on the knot) is free. The
  action preserves orientations of both $S^3$ and the knot. A hyperbolic $C_2$-symmetric knot of this type is depicted in Figure
  \ref{fig:free2period}. In this example, the action of $\tau$ on $S^3$ is the antipodal action.
\begin{figure}
    \centering
    \includegraphics{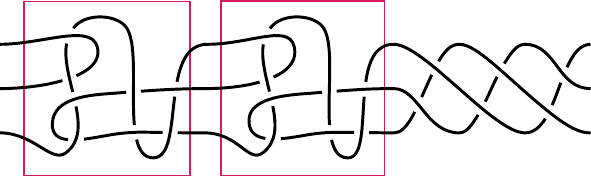}
    \caption[A freely 2-periodic knot.]{The closure of the shown tangle is a freely 2-periodic knot. The symmetry
      exchanges the tangles in the boxes. The leftover full twist on the right is an artifact of the projection onto the
      plane.  }
    \label{fig:free2period}
\end{figure}
\item \textit{Strong positive amphichirality}, abbreviated SPAc. Here the action on $S^3$ is orientation-reversing, fixing two
  points disjoint from the knot. The action on the knot is therefore free and orientation-preserving. A hyperbolic symmetric knot of this type is depicted in Figure
  \ref{fig:12a_427other}. We remark in passing that the symmetric diagram is not a minimal-crossing diagram\benw{more?}.
\begin{figure}
    \centering
    \begin{minipage}[t]{0.45\textwidth}
    \centering
    \includegraphics[width=0.9\textwidth]{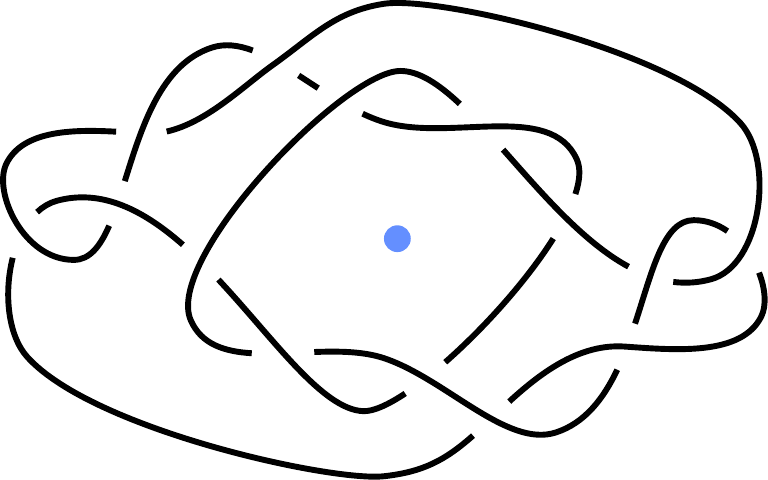}
    \caption[The strongly positive amphichiral knot $12a_{427}$.]{The strongly positive amphichiral knot
      $12a_{427}$. The symmetry is point reflection across the blue point in the centre of the diagram}
    \label{fig:12a_427other}
  \end{minipage} \hfil
    \begin{minipage}[t]{0.45\textwidth}
    \centering
    \includegraphics[width=0.9\textwidth]{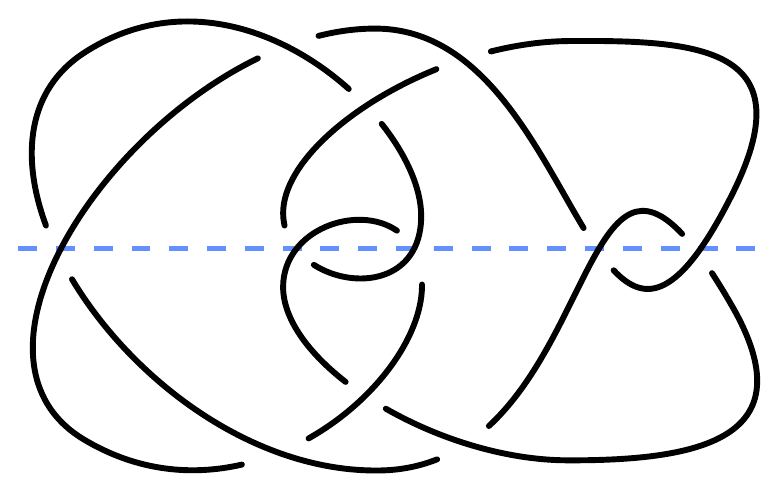}
    \caption[The 2-periodic knot $11a_{200}$.]{The 2-periodic knot $11a_{200}$. The symmetry is a $\pi$-rotation around the indicated horizontal axis, which
      is disjoint from the knot. }
    \label{fig:10_67}
    \end{minipage} 
 \end{figure}
    \smallskip
    \begin{figure}
    \centering
    \begin{minipage}[t]{0.45\textwidth}
    \centering    
  \includegraphics[width=0.9\textwidth]{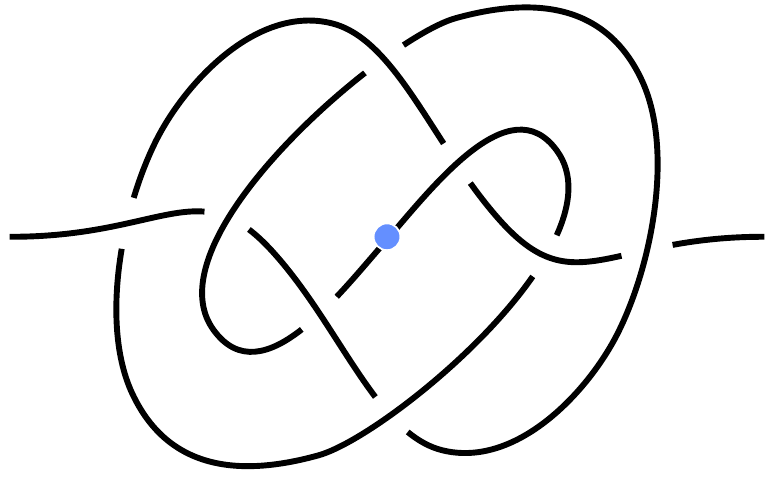}
    \caption[The strongly negative amphichiral knot $8_{17}$.]{The strongly negative amphichiral knot $8_{17}$. The
      symmetry is point reflection across the blue point in the centre of the diagram.}
    \label{fig:8_17}
\end{minipage} \hfil
    \begin{minipage}[t]{0.45\textwidth}
    \centering
   \includegraphics[width=0.9\textwidth]{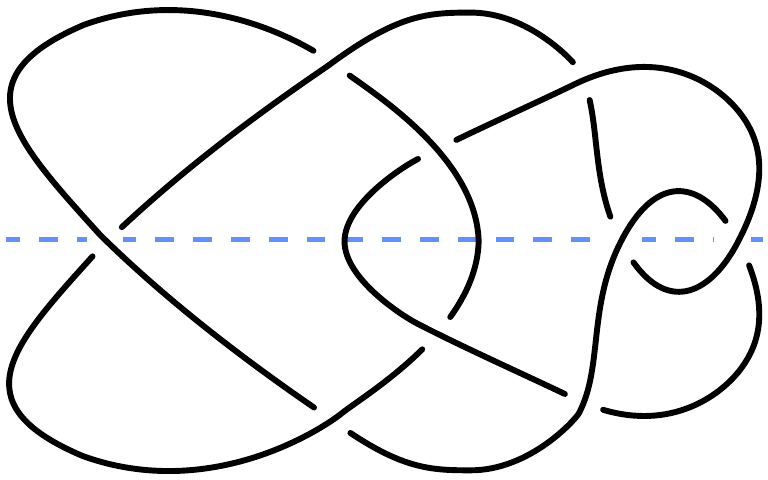}
    \caption[The strongly invertible knot $8_{10}$.]{The strongly invertible knot $8_{10}$. The symmetry is a
      $\pi$-rotation around the indicated horizontal axis,
      which meets the knot at two points.}
    \label{fig:8_10}
    \end{minipage} 
  \end{figure}

 \item \textit{2-periodicity}, abbreviated 2P. Here the action on $S^3$ is orientation-preserving, fixing a $1$-dimensional
   unknotted circle disjoint from the knot. The action on the knot is therefore free and orientation-preserving. A hyperbolic symmetric knot of this type is depicted in Figure \ref{fig:10_67}.
  \item \textit{Strong negative amphichirality}, abbreviated SNAc. Here the action of $\tau$ on $S^3$ is orientation-reversing, fixing $2$
    points that lie on the knot. The action on the knot is therefore also orientation-reversing. A hyperbolic symmetric knot of this type is depicted in Figure
  \ref{fig:8_17}.
  \item \textit{Strong invertibility}, abbreviated SI. Here the action on $S^3$ is orientation preserving, fixing an unknotted
    circle that meets the knot in two points. The action on $K$ is orientation-reversing. A hyperbolic symmetric knot of this type is depicted in Figure \ref{fig:8_10}.
  \item \textit{$S^2$-reflection}, abbreviated 2R. Here the action on $S^3$ is orientation-reversing by reflection across an
    embedded $2$-sphere that meets the knot in $2$ points. The action on $K$ is orientation-reversing. This type of
    symmetry does not arise for prime knots. The square knot of \cite[3.D~Example 10]{Rolfsen2003} admits a $C_2$-action
    of this type, exchanging the trefoil summands.
  \end{enumerate}
  
\begin{remark}
  The examples in Figures \ref{fig:free2period}--\ref{fig:8_10} are hyperbolic knots and have been chosen so that the indicated
  $C_2$-symmetry is the only possible type of nontrivial symmetry on any knot $K$ that is (non-equivariantly) equivalent to the
  given one. This can be deduced using the results of Section \ref{sec:relNonRigid} and a computation of the full symmetry group in SnapPy \cite{SnapPy}.
\end{remark}
  
\subsection{Examples of symmetric knots of higher order}
\label{sec:exampl-symm-knots}

Let $C_n = \langle \rho \mid \rho^n = e \rangle$ be a cyclic group of order $n$ where $n > 2$. Such a group cannot act
on a knot by a symmetry that reverses the orientation of the knot. We distinguish three families of types of
$C_n$-symmetric knots.

\begin{enumerate}
\item \textit{Generalized Free periodicy}, abbreviated GFPer.  This is the analogue for $n > 2$ of free $2$-periodicity. Here the action of the
  group on $S^3$ is orientation-preserving and without global fixed points. Examples of this type are given by torus
  knots: see Section \ref{sec:torus}.

  Under this heading, we give both the symmetries that are truly freely periodic, where the group acts
  freely on $S^3$, and \emph{semi-periodic} symmetries in the sense of \cite[Definition 1]{Paoluzzi2018}, where $\rho$
  itself acts without fixed points, but non-identity powers of $\rho$ have fixed points. The semi-periodic symmetries
  contain the \emph{biperiodic} symmetries of \cite{Guilbault2021} as a special case. These are discussed in Remark \ref{rem:subdivisionsOfFPer}.
  \item \textit{Periodicity}, abbreviated Per. This is the analogue for $n > 2$ of $2$-periodicity. Here the action of the group $S^3$ is orientation-preserving, but fixing an
    unknotted circle which is, up to conjugacy by $\Diff^+(S^3)$, an axis of rotational symmetry.  Examples of this type are also given by torus
  knots: see Section \ref{sec:torus}.
  \item \textit{Rotoreflectional symmetry}, abbreviated RRef. This is the analogue for $n> 2$ of strongly positive amphichiral symmetry. Here
    $\rho$ acts by an orientation-reversing symmetry, fixing $2$ points. The integer $n$ is necessarily even in this
    case. The action of the subgroup generated by $\rho^2$ is by a periodicity. Examples of this type are constructed in
    Section \ref{sec:exAmphichiral}.
  \end{enumerate}

Let $D_n = \langle \rho, \sigma \mid \rho^n = \sigma^2 = (\rho \sigma)^2 = e \rangle$ be a dihedral group of order
$2n$ where $n \ge 2$. When $n=2$, the group theory alone is not sufficient to distinguish $\rho$ from $\sigma$ and $\rho
\sigma$. We adopt the convention that whenever $D_n$ acts on a knot, the element $\rho$, even if it is of order $2$,
preserves the orientation of the knot, while $\sigma$ reverses it.

\begin{enumerate}
\item \textit{Strongly Invertible Generalized Freely Periodic}, abbreviated SIFP. Here the group generated by $\rho$ acts on $S^3$ by a generalized free periodicity. The elements
  $\rho^i \sigma$ all act by strong inversions. This type of symmetry may be found on torus knots, see Section
  \ref{sec:torus}.  An example of the case of $n=2$ is given in Figure \ref{fig:10_155}.
\item \textit{Strongly Invertible Periodic}, abbreviated SIP. Here the group generated by $\rho$ acts on $S^3$ by
  periodicity. The elements $\rho^i \sigma$ all act by strong inversions. This type of symmetry can be found among the
  torus knots (see Section \ref{sec:torus}) or by restricting the $D_{2n}$-symmetric knot in Proposition \ref{pr:snasi}
  to the dihedral subgroup generated by $\rho^2, \sigma$. An example in the case of $n=2$ is given in Figure \ref{fig:5_2}.
  \begin{figure}[h]
    \centering
    \begin{minipage}[b]{0.45\textwidth}
      \centering
      \includegraphics[width=0.9\textwidth]{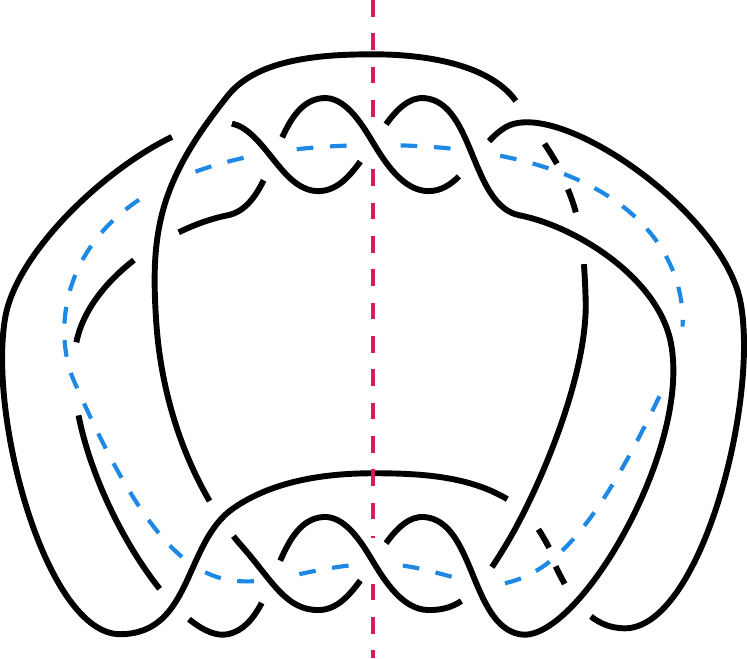}
      \caption[The knot $10_{155}$, which has a $D_2$-symmetry of type \protect\hyperlink{SIFP}{SIFP}-(1,1).]{The knot $10_{155}$, which has a $D_2$-symmetry of type \protect\hyperlink{SIFP}{SIFP}-(1,1). Here
        $\rho$, the freely periodic symmetry, acts by a $\pi$ rotation around the axis perpendicular to the page
        composed with a $\pi$ rotation around the blue circle, and $\sigma$, a strong inversion, acts by rotation around the vertical violet axis.}
      \label{fig:10_155}
    \end{minipage}\hfil
    \begin{minipage}[b]{0.45\textwidth}
      \centering
      \includegraphics[width=0.9\textwidth]{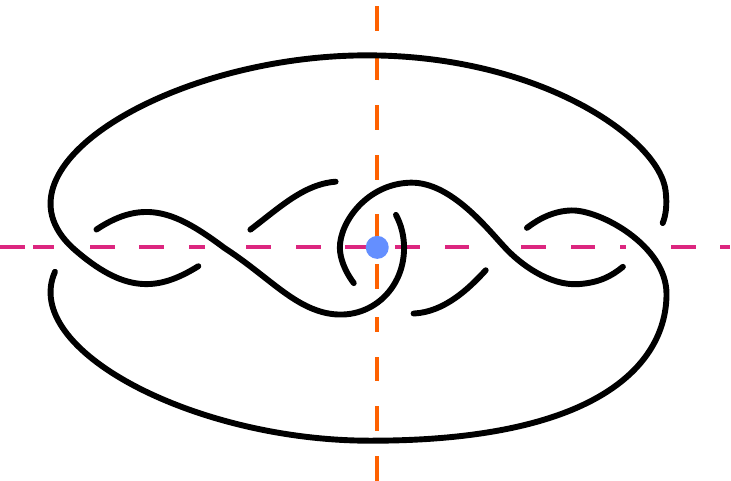}
      \caption[The knot $5_2$, which has a $D_2$-symmetry of type \protect\hyperlink{SIP}{SIP}-(1).]{The knot $5_2$, which has a $D_2$-symmetry of type \protect\hyperlink{SIP}{SIP}-(1,0). Here $\rho$, the
        periodic symmetry, acts by rotation around the blue axis perpendicular to the page, and $\sigma$, a strong inversion,
        acts by rotation around the horizontal violet axis. Their composition $\rho\sigma$ acts by rotation around the
        vertical orange axis.}
      \label{fig:5_2}
    \end{minipage}
  \end{figure}
    
  \item \textit{Strongly Negative Amphichiral Periodic}, abbreviated SNAP. Here the subgroup generated by $\rho$ acts on $S^3$ by
    periodicity. The elements $\rho^i \sigma$ are all strongly negative amphichiral. Examples of this kind of symmetry can be found by
  restricting the $D_{2n}$-symmetric knot in Proposition \ref{pr:snasi} to the dihedral subgroup generated by $\rho^2,
  \rho \sigma$. See also Figure \ref{fig:12a_868} for an example when $n=2$.
  \item \textit{Strongly Negative Amphichiral, Strongly Invertible and Periodic}, abbreviated SNASI. This type can arise only when $n$
    is even. The action of the subgroup generated by $\rho$ on $S^3$ is by rotoreflections. The order-$2$ elements $\rho^i \sigma$ belong to two conjugacy classes, depending on whether $i$ is even or odd. In one class the elements act by strong negative amphichiralities, and in
    the other they act by strong inversions. 

    Examples of this kind of symmetry are constructed in Proposition
    \ref{pr:snasi}. See Figures \ref{fig:Dih5-1}, \ref{fig:Dih5-3} for constructions indicative of the general
    cases. See also Figure \ref{fig:10_99} for an example when $n=2$.
    \begin{figure}[h]
      \centering
   
\begin{minipage}[t]{0.45\textwidth}
    \centering
    \includegraphics[width=0.9\textwidth]{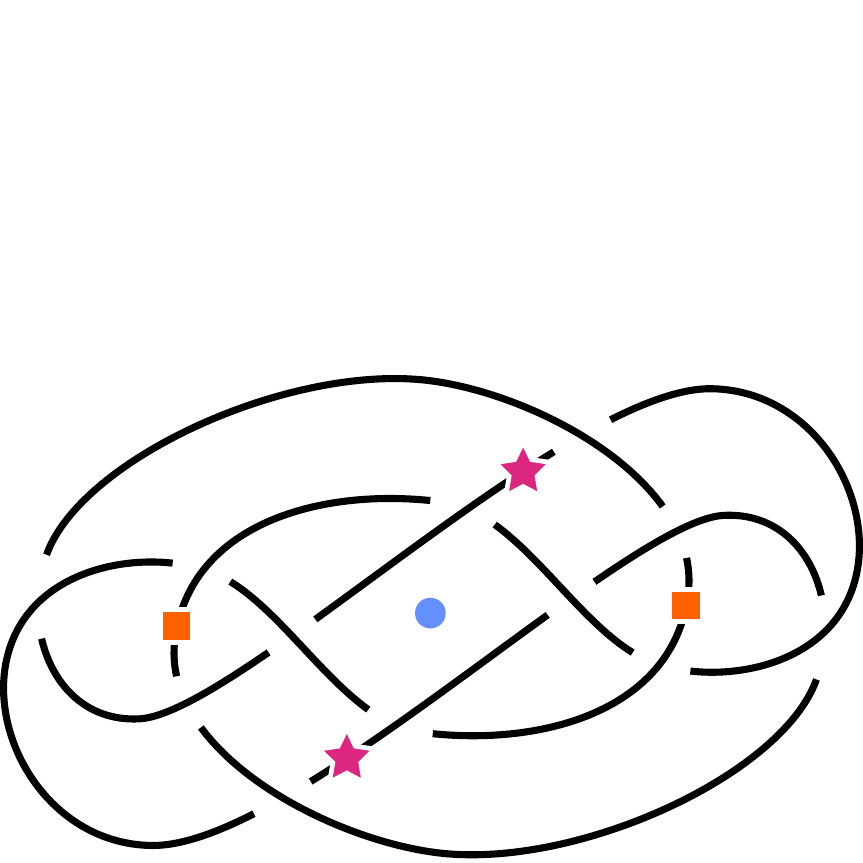}
    \caption[The knot $12a_{868}$, which has a $D_2$-symmetry of type \protect\hyperlink{SNAP}{SNAP}-(1).]{The knot $12a_{868}$, which has a $D_2$-symmetry of type \protect\hyperlink{SNAP}{SNAP}-(1). Here $\rho$,
      a periodic symmetry, acts by rotation around an axis perpendicular to the page: $\sigma$, a strongly negative
      amphichiral symmetry, fixes only the two violet points, and $\rho\sigma$ (also strongly negative amphichiral) fixes the two orange points.}
    \label{fig:12a_868}
\end{minipage}
\hfill
\begin{minipage}[t]{0.45\textwidth}
    \centering
    \includegraphics[width=0.9\textwidth]{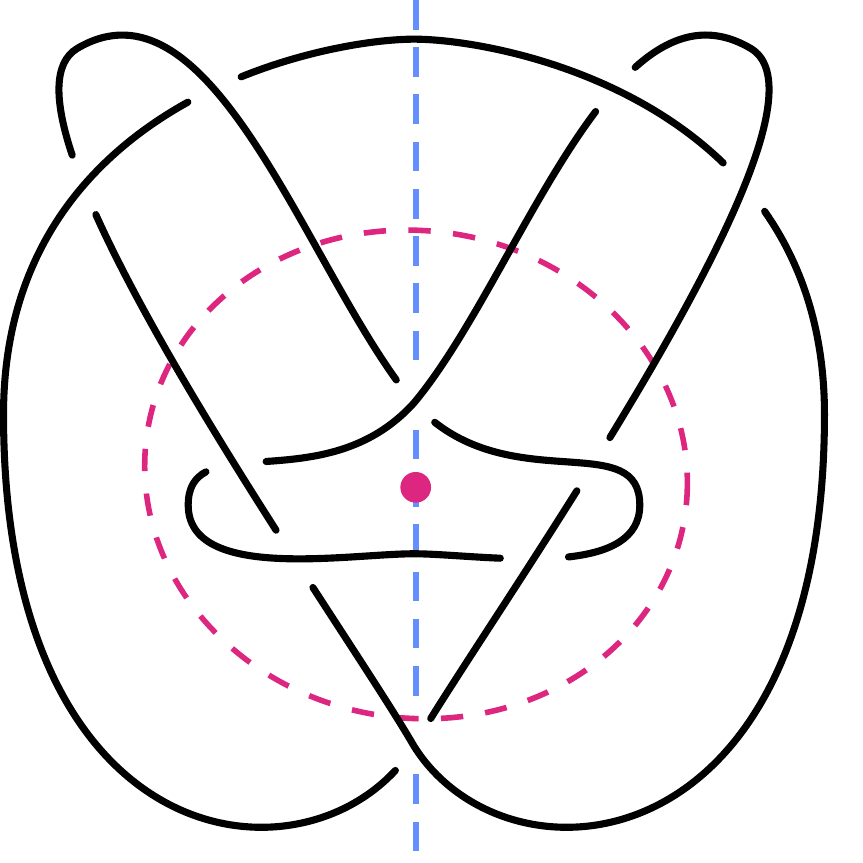}
    \caption[The knot $10_{99}$, which has a $D_2$-symmetry of type \protect\hyperlink{SNASI}{SNASI}-(1).]{The knot $10_{99}$, which has a $D_2$-symmetry of type \protect\hyperlink{SNASI}{SNASI}-(1). Here $\rho$, a
      strongly positive amphichiral symmetry, acts by a composition of a $\pi$-rotation around the axis indicated by
      the central point followed by reflection across a $2$-sphere meeting the diagram in the dotted curve, both in violet,
      and $\sigma$, a strong inversion, acts by $\pi$ rotation around the indicated vertical blue axis.}
    \label{fig:10_99}
\end{minipage}
\end{figure}
\end{enumerate}

  \begin{remark}
    The hyperbolic examples in Figures \ref{fig:10_155}--\ref{fig:10_99} have been chosen so that the indicated
    $D_2$-symmetry is the unique maximal symmetry on knots of this class. That is, any symmetry of on any knot $K$ that
    is (non-equivariantly) equivalent to the given knot is obtained from the given symmetry by conjugation and
    restriction of symmetry group. This can be deduced using the results of Section \ref{sec:relNonRigid} and the
    information about each of the knots used in the KnotInfo database \cite{knotinfo}.
\end{remark}

There are three further families of dihedral symmetry types, but these apply only to composite knots. They are discussed further in Subsection \ref{sec:an-enumeration-types}

\subsection{Detection by computational hyperbolic geometry}

  For a hyperbolic knot $K$, the full symmetry group  of the knot $S=\pi_0(\Diff(S^3, K))$ is (up to isomorphism) the unique maximal finite group
  that acts faithfully on the knot (see Section \ref{sec:relNonRigid}), and the type of the action depends only on the smooth isotopy class of $K$ (see
  Propositions \ref{pr:fullSymmetryAllEquivalent} and \ref{pr:typeDependsOnlyOnEquivClass}). It is possible to determine which of these symmetry types a particular hyperbolic knot has using the software
  SnapPy \cite{SnapPy}, with the exceptions that we cannot distinguish between periodic, semi-periodic and freely periodic symmetries of the same order, e.g., between \protect\hyperlink{Per}{Per} and
  \protect\hyperlink{GFPer}{GFPer}, or between \protect\hyperlink{SIP}{SIP} and \protect\hyperlink{SIFP}{SIFP}. We also cannot determine the secondary invariants $a$ or $(a,b)$ that appear in Tables \ref{tab:CnTypes} and
\ref{tab:D2nTypes}. The steps
  to perform this determination in SnapPy \cite{SnapPy} as of version 3.1 are as follows.
\begin{enumerate}
\item Input a hyperbolic knot complement $M$, for example the figure-eight knot:\, \texttt{M = Manifold(\textquotesingle 4\_1\textquotesingle)}.
\item Compute the full symmetry group $S$ of $M$:\, \texttt{S = M.symmetry\_group(); S}.

If $S$ is trivial, then there is nothing further to be done. If $S$ is cyclic of order $2$, then there is an ambiguity since $C_2$ fits in both the cyclic and dihedral families. We assume for the time being that $|S|>2$, and will return to the case $|S|=2$ in point \ref{order2SnapPy} below.
\item Determine whether $S$ contains an inversion:\, \texttt{S.is\_invertible\_knot()}.
\item Determine whether $S$ contains an amphichiral element:\, \texttt{S.is\_amphicheiral()}.
\end{enumerate}
From this information, we deduce the following when $S$ has order greater than 2:
\begin{itemize}
\item If $S$ is cyclic but with no amphichiral element, then the type is \protect\hyperlink{Per}{Per} or \protect\hyperlink{GFPer}{GFPer}. 
\item If $S$ is cyclic with an amphichiral element, then the type is \protect\hyperlink{RRef}{RRef}.
\item If $S$ is dihedral with an inversion but no amphichiral element, then the type is \protect\hyperlink{SIP}{SIP} or \protect\hyperlink{SIFP}{SIFP}.
\item If $S$ is dihedral with an amphichiral element but no inversion, then the type is \protect\hyperlink{SNAP}{SNAP}.
\item If $S$ is dihedral with an amphichiral element and an inversion, then the type is \protect\hyperlink{SNASI}{SNASI}.
\end{itemize}

If $|S|=2$, we instead continue as follows.
\begin{enumerate}[label=(\arabic*')] \setcounter{enumi}{2}
    \item \label{order2SnapPy} Explicitly calculate the action of the elements of $S$ on the cusp of the knot complement by the command: \texttt{S.isometries()}. The nontrivial element of $S$ has matrix
\[ \begin{bmatrix} x & 0 \\ 0 & y \end{bmatrix},\quad x,y \in \{-1,1\}.  \] 
Here $x$ denotes the action of the symmetry on a knot meridian and $y$ the action on a longitude.
\end{enumerate}
\begin{itemize}
   \item If $x=1$ and $y=1$, then the nontrivial element preserves the orientations of both meridian and longitude, and the symmetry is either 2-periodic (\protect\hyperlink{2P}{2P}) or freely 2-periodic (\protect\hyperlink{F2P}{F2P}).
   \item If $x=-1$ and $y=1$, then the symmetry is strongly positive amphichiral (\protect\hyperlink{SPAc}{SPAc}).
   \item If $x=1$ and $y=-1$, then the symmetry is strongly negative amphichiral (\protect\hyperlink{SNAc}{SNAc}).
   \item If $x=-1$ and $y=-1$, then the symmetry is a strong inversion (\protect\hyperlink{SI}{SI}).
\end{itemize}

\subsection{Other results}
\label{sec:other-results}

Along the way, we prove the following two propositions, which may be of independent interest and seem not to have appeared in the literature. The first proposition improves on work in \cite{Cappell1999} by removing the requirement that the conjugating element be orientation-preserving.

\begin{proposition*}[Proposition \ref{pr:reduceDiff+toSO}]
If $G$ is a finite cyclic or dihedral group and $\beta,\beta' \colon G \to \Og(4)$ are injective homomorphisms that are conjugate by an orientation-preserving diffeomorphism $d \in \Diff^+(S^3)$, then $\beta$ and $\beta'$ are conjugate by an element of $\SO(4)$.
\end{proposition*}
We remark that the following proposition, though stated in terms of a $C_n$-symmetric knot, applies to ones with dihedral symmetry by restricting to a cyclic subgroup.
\begin{proposition*}[Proposition \ref{cor:linkSimilarActions}]
 Fix a faithful action $\alpha \colon C_n \to \Diff^+(S^3)$. Let $K, L_1$ and $L_2$ be $C_n$-symmetric knots for this action,
 where $\alpha|_K$ and $\alpha|_{L_i}$ preserve the orientations and where $|K|$ is disjoint from $|L_1| \cup |L_2|$. 
    Let $T$ be a smooth $C_n$-isotopy from $L_1$ to $L_2$ and suppose that the action of $C_n$ on $K$ is free. Then \[
      \link(|K|, |L_1|) \equiv \link(|K|, |L_2|) \pmod{n}. \] 
\end{proposition*}


\subsection{Acknowledgments}

We thank Liam Watson for helpful conversations and Shmuel Weinberger for remarks about smooth actions of finite groups
on $S^1$. We thank Jim Davis for drawing our attention to the work of \cite{Guilbault2021} on biperiodic knots.

\section{Notation and other preliminaries}

Throughout the paper, both $S^1$ and $S^3$ have a preferred orientation, but unless otherwise stated, maps $S^1 \to S^1$ and $S^3 \to S^3$ may be orientation-reversing.

\begin{notation}
    If $\phi\colon G \to H$ is a homomorphism of groups and $h \in H$, then the \emph{$h$-conjugate $\phi^h$} denotes the homomorphism $\phi^h(g) = h^{-1} \phi(g) h$.
  \end{notation}

\begin{notation}
  If $M$ is a smooth manifold, then $\Diff(M)$ denotes the group of diffeomorphisms of $M$. If $M$ is oriented, then $\Diff^+(M)$ denotes the subgroup of orientation-preserving
  diffeomorphisms. The notation $\Diff_0(M)$ denotes the subgroup of diffeomorphisms that are isotopic to the identity.

  If $N$ is a closed submanifold of $M$, then $\Diff(M,N)$ denotes the subgroup of $\Diff(M)$ consisting of maps leaving
  $N$ invariant.
\end{notation}

\begin{notation} \label{not:FORep}
  Suppose $G$ and $D$ are groups, and $D^+ \subset D$ is a subgroup. If $\phi, \phi'$ are two homomorphisms $G \to D$, then we say that they are \emph{$D^+$-conjugate} if there exists some $d \in D^+$ such that $\phi^d = \phi'$.

  We will apply this terminology when the group $D$ is of the form $\Diff(S^n)$, diffeomorphisms of the $n$-sphere, or
  $\Og(n+1)$, isometries of the $n$-sphere. The subgroup $D^+$ is invariably $\Diff^+(S^n)$ or $\SO(n+1)$ in these
  cases. 
  
  Then the set of \textit{oriented-equivalence classes of $D$-representations of $G$}, denoted $\ORep(G; D)$ (or
  $\ORep(G; D, D^+)$ in the rare case where it is necessary to specify $D^+$),
  consists of equivalence classes of homomorphisms $\phi\colon G \to D$ under the relation of $D^+$-conjugacy.
  
  We will write $\FORep(G; D)$ for the subset of classes of $\ORep(G; D)$ consisting of faithful representations.
\end{notation}


We remark that $\Aut(G)$ acts on $\ORep(G; D)$ by precomposition.

An unknot is not considered to be a prime knot in this paper.

\begin{notation}
  If $K$ is a knot, then $\nu(K)$ denotes some open tubular neighbourhood of $|K|$ in $S^3$. If $K$ is a $G$-symmetric
  knot for some finite $G$, then $\nu(K)$ will be assumed to be $G$-invariant.
\end{notation}

\begin{notation}
  The notation $\langle X \mid Y \rangle$ will be used to denote the  group or ideal generated by $X$ subject to
  relations $Y$, and the notation $\langle X \rangle$, without relations, will be used when $X$ is a subset of some
  specified group or ring to denote the subgroup or ideal generated by $X$.
\end{notation}

\begin{notation}
  If $R$ is a ring, then $R^\times$ denotes the group of multiplicative units in $R$.
\end{notation}

\begin{notation} \label{not:Fn}
  The notation $F(n)$ will be used for the set of equivalence classes of elements of $\ZZ/(n)$ under the multiplication action by $\{ \pm 1\}$. As a matter of convention, we may treat an element $a \in F(n)$ as an element of $\ZZ/(n)$ if it does not matter in context which of $a$ or $-a$ is meant. The notation $F(n)^\times$ will be used for the classes of elements $a$ that lie in $\ZZ/(n)^\times$.
\end{notation}
  
\begin{notation} \label{not:Tn}
  The notation $T(n)$ will be used for the set of equivalence classes of the set
  \[  \{ (a,b) \in \ZZ/(n) \times \ZZ/(n) \mid a \neq 0, \, b \neq 0, \, \{ a, b\} \text{ generates the unit ideal in
      $\ZZ/(n)$} \} \]
  by the equivalence relation
  \[ (a,b) \sim (-a,-b) \sim (b,a) \sim (-b, -a). \]
\end{notation}

\begin{notation}
  If $n \ge 3$, the notation $C_n$ denotes a cyclic group of order $n$, generated by a named element, $\rho$. That is
  \[ C_n = \langle \rho \mid \rho^n = e \rangle. \]
  If $n \ge 2$, the notation $D_n$ denotes a dihedral group of order $2n$, generated by named elements $\rho$ and
  $\sigma$. Specifically
  \[ D_n = \langle \rho \mid \rho^n = \sigma^2 = (\rho \sigma)^2 = e \rangle.\]
  Cyclic groups of order $2$ have a dual nature in this work, being both isomorphic to  $C_2$ and to the (degenerate)
  $D_1$.  We will write $C_2$ for a cyclic group of order $2$, and the nontrivial element may be named $\rho$, when we
  view the group as part of the $C_n$ family, or $\sigma$, when it is part of the $D_n$ family, or $\tau$ when we do not
  wish to choose.
\end{notation}

A \emph{round circle} in $S^3$ is a $1$-dimensional intersection of $S^3$ with a $2$-dimensional plane $P$ in $\RR^4$.

\begin{notation}
  We endow the space $\RR^4$ with named coordinates $(x,y,z,w)$, and write $S^1_{xy}$ and $S^1_{zw}$ etc.~for the round
  circles obtained by intersecting the coordinate planes with $S^3$. Similarly, $S^2_{xyz}$ etc.~denote the
  intersection of $S^3$ with the $xyz$-subspace of $\RR^4$ etc.\end{notation}

\begin{notation}
  For knots of crossing-number $10$ or less, notation of the form $10_{67}$ refers to the table of \cite{Rolfsen2003}.
  For knots of crossing-number $11$ or $12$,  notation of the form $12a_{427}$ refers to the table of KnotInfo \cite{knotinfo}. 
\end{notation}

\section{Symmetric knots} \label{sec:symmetric_knots}

Recall from Definition \ref{def:GsymmKnot} that a $G$-symmetric knot consists of a knot $K$ along with a faithful action $\alpha$ of $G$ by diffeomorphisms on $S^3$ for which $|K|$ is $G$-invariant. 
When we wish to emphasize the action map $\alpha$, we will write $G$-symmetric knots as ``$(K, \alpha)$''. At other times, we will write such a knot as ``$K$'', the $G$-action being understood. The groups $G$ we consider are usually finite discrete groups, but in the case of torus knots, the Lie groups $\Og(2)$ and $\SO(2)$ may arise. In these cases, the action homomorphism $\alpha \colon G \to \Diff(S^3)$ is required to be continuous.

\begin{notation}
    If $K$ is a $G$-symmetric knot, then $G$ acts smoothly on the closed submanifold $|K|$, and via the diffeomorphism $K^{-1}$, the group $G$ acts smoothly on $S^1$. We will denote the action on $S^1$ by $\alpha|_K$.
  \end{notation}
  Loosely, we will refer to either of the actions of $G$, on $S^1$ or on $|K|$ as ``the action of $G$ on the knot''.

\begin{remark}
    A slightly finer notion of $G$-symmetric knot may be useful in applications. Specifically, in addition to requiring
    $S^3$ and $S^1$ to be oriented, one might also specify orientations on all the fixed loci $(S^3)^H$ and $|K|^H$ as
    $H$ runs over (closed) subgroups of $G$ and require these orientations to be part of the data. This is used, for
    example, when defining an equivariant concordance group for strongly invertible knots in \cite{Sakuma1986}.
\end{remark}

\subsection{Actions on nontrivial knots}
\label{sec:nontrivial}

The resolution of the Smith conjecture, \cite{Morgan-Bass}, implies that the fixed locus of finite-order diffeomorphism
$\phi \colon S^3 \to S^3$ is either empty or an unknotted sphere of some dimension $d$. If $\phi$ reverses the orientation of $S^3$, then $d$
is even, whereas if $\phi$ preserves the orientation of $S^3$, then $d$ is odd---including the possibility that the fixed set is the empty set, which is assigned a dimension of $-1$.

As a consequence, we arrive at the following well-known proposition.
\begin{proposition} \label{lem:groupActingOnK}
If $K$ is a nontrivial $G$-symmetric knot for some group $G$, then the group $G$ acts faithfully on $K$. If $G$ is finite, then $G$ is either cyclic or dihedral. 
\end{proposition}

\benw[inline]{Longer versions of the above have been hidden in comments.}

\subsection{Relation to the full symmetry group}
\label{sec:relNonRigid}

Our definition of ``$G$-symmetric knot'' is an instance of ``rigid'' symmetries of knots in the terminology of
\cite{Sakuma1986}. That is, $G$ is embedded as a finite subgroup of $\Diff(S^3, K)$.  The ``(full) symmetry
group'' of a knot is defined (\cite{Sakuma1986}, \cite{Boileau1989}) to be $\pi_0(\Diff(S^3, K))$. From the point of
view of finite group actions, the two notions of symmetry are remarkably similar. In this subsection, we collect a
number of results from the literature that make this similarity precise.

The following is a modification of \cite[Thm.~2.1]{Boileau1989}. If $K$ is a $G$-symmetric knot, then $G$ may be viewed as a subgroup of $\Diff(S^3, K)$\benw{move preamble?}. 
\begin{theorem} \label{th:BZ}
    Let $K$ be a $G$-symmetric knot that is not a torus knot. Write $G'$ for the image of $G$ in $\pi_0(\Diff(S^3, K))$. 
    \begin{enumerate}
        \item If $G$ is finite, then $G \to G'$ is an isomorphism.
        \item If $G'$ is finite, then there exists a finite subgroup of $\Diff(S^3, K)$ whose image under $\pi_0$ is $G'$.
    \end{enumerate}
\end{theorem}

The proof of this theorem involves another observation that is useful, especially when treating the case of hyperbolic
knots. Fix a knot $K$ and let $M = S^3 \sm \nu(K)$ denote the knot exterior. Then there exists a canonical injection
$\pi_0(\Diff(S^3, K)) \to \pi_0(\Diff(M))$, which is actually an isomorphism by \cite{GordonLuecke1989}. By
\cite[Satz~0.1]{Zimmermann1982} (adjusted in the addendum), if $G$ is a finite subgroup of $\pi_0(\Diff(M))$, then $G$
admits a lift to $\Diff(M)$ and from there to $\Diff(S^3, K)$. In conclusion, if
$\pi_0(\Diff(M)) \isom \pi_0(\Diff(S^3, K))$ is finite, then there exists a section of the homomorphism taking
$\Diff(M)$ to $\pi_0(\Diff(M))$.

If $M$ is a connected manifold and $m \in M$ a basepoint, then there is a well-known\benw{reference} homomorphism
\[ \tilde \Phi \colon \Homeo(M) \to \Out(\pi_1(M,m)).\] If $f \in \Homeo(M)$ is isotopic to the identity, then $\tilde \Phi(f)$
is trivial. Therefore we arrive at a canonical map $\Phi \colon \pi_0(\Homeo(M)) \to \Out(\pi_1(M, m))$. If $M = S^3 \sm \nu(K)$ is the
exterior of a nontrivial knot, then the map $\Phi$ is injective by \cite[Theorem 7.1]{Waldhuasen1968}.


When $M$ is also hyperbolic, more can be said.
\begin{prop} \label{pr:generalPropAboutHyperbolic3Manifolds} Let $M$ be the exterior of a hyperbolic knot. The group $\Isom(M)$ of isometries is a discrete finite group and the evident homomorphisms
	\begin{equation} \label{eq:chain}
	    \mathrm{Isom}(M) \to \pi_0(\Diff(M)) \rightarrow \pi_0(\Homeo(M)) \rightarrow \mathrm{Out}(\pi_1(M))
	\end{equation}
	are isomorphisms.
\end{prop}
\begin{proof}
    The isometry group a compact hyperbolic manifold is well-known to be finite.
	The first two maps in \eqref{eq:chain} are the image under $\pi_0$ of the inclusions
	\[
	\mathrm{Isom}(M) \hookrightarrow \mathrm{Diff}(M) \hookrightarrow \mathrm{Homeo}(M)
	\]
	together with the identification of $\mathrm{Isom}(M)$ with
        $\pi_0(\mathrm{Isom}(M))$. \par
	The Smale conjecture (\cite{Hatcher1976}) implies that the inclusion of $\mathrm{Isom}(M)$
        into $\Diff(M)$ is a homotopy equivalence, so $\Isom(M) \to \pi_0(\Diff(M))$ is an isomorphism.
	   Mostow rigidty implies that the map $\pi_0(\Homeo(M)) \rightarrow \mathrm{Out}(\pi_1(M))$ is surjective, and \cite[Theorem 7.1]{Waldhuasen1968} implies that it is injective. See also \cite[Corollary 5.3]{Gabai1997}.
	Finally, the composition $\mathrm{Isom}(M) \rightarrow \mathrm{Out}(\pi_1(M))$ is an isomorphism by Mostow        rigidity, so the map $\pi_0(\mathrm{Diff}(M)) \rightarrow \pi_0(\mathrm{Homeo}(M))$ must be an isomorphism as
        well.  
\end{proof}

If $G \subset \pi_0(\Diff(S^3, K))$ is a finite subgroup, then Theorem \ref{th:BZ} assures us that $G$ can be lifted to
a subgroup of $\Diff(S^3, K)$. This subgroup is known to be unique up to $\Diff_0(S^3, K)$-conjugacy in many 
cases.\benw{Check it's not known in all cases.}

\begin{proposition} \label{pr:uniquenessOfLifting}
  Suppose $K$ is a knot that is not a torus knot. Let $G_1, G_2 \subset \Diff(S^3, K)$ be two finite subgroups so that
  the isomorphic images under $\pi_0 \colon \Diff(S^3, K) \to
  \pi_0(\Diff(S^3, K))$ agree. If one of the following holds,
  \begin{enumerate}
  \item \label{i:un1} the actions of $G_1$ and $G_2$ are by orientation-preserving symmetries of $S^3$;
  \item \label{i:un3} the actions of $G_1$ and $G_2$ on $S^3 \sm \nu(K)$ are geometric;
  \item \label{i:un4} the knot $K$ is hyperbolic;
  \end{enumerate}
then $G_1$ and $G_2$ are conjugate by an element of $\Diff_0(S^3,K)$.
\end{proposition}
\begin{proof}
  We give references for each of these statements: case \ref{i:un1} is part of
  \cite[Th.~2.1(c)]{Boileau1989}. Case \ref{i:un3} is \cite[Thm.~0.1]{Zimmermann1986} along with the injectivity of
  $\pi_0(\Diff(S^3, K)) \to \Out(\pi_1(S^3 \sm \nu(K)))$ established above. Case \ref{i:un4} is an instance of Case
  \ref{i:un3} and Proposition \ref{pr:generalPropAboutHyperbolic3Manifolds}, which says the $G$ action on $S^3 \sm \nu (K)$ is
  geometric.
\end{proof}
\benw{change what's below}

\begin{remark} \label{rem:fullSymmHyperbolic}
 In the case of a hyperbolic knot $K$, combining Theorem \ref{th:BZ}, Proposition \ref{pr:generalPropAboutHyperbolic3Manifolds} and Theorem
  \ref{pr:uniquenessOfLifting} assures us that, up to conjugacy by $\Diff_0(S^3, K)$, there is a unique maximal finite subgroup $G \subset
  \Diff(S^3, K)$, and this finite subgroup maps isomorphically to $\pi_0(\Diff(S^3, K))$, the full symmetry group of the knot.
\end{remark}

  In general, $\pi_0(\Diff(S^3, K))$ is infinite: \cite[Proposition 2.7]{Sakuma2} gives an example of a knot $K$ for
  which $\pi_0(\Diff(S^3, K))$ acts on $(S^3, K)$ by diffeomorphisms but not by isometries, and therefore the group
  cannot be finite by (e.g.) Theorem \ref{pr:allGsymmetricKnotsEquivIsometry} below. Nonetheless, \cite{Flapan1986}
  proves that, except for torus knots, there are only finitely many conjugacy classes of finite subgroups of $\Diff(S^3,
  K)$, and \textit{a fortiori} there are only finitely many conjugacy classes of finite subgroups of $\pi_0(\Diff(S^3,
  K))$.

\subsection{Relation to non-rigid symmetries}
\label{sec:relation-non-rigid}

``Non-rigid'' symmetries are defined by viewing knots only up to equivalence. For a given representative embedding $K \subset S^3$, one may consider diffeomorphisms $f :S^3 \to S^3$ with the property that $f(K)$ and $K$ are equivalent,
disregarding orientation. To eliminate the dependence on the representative $K$, the diffeomorphism $f$ should be considered
only up to postcomposition by diffeomorphisms that preserve the orientations of both $S^3$ and $f(K)$. Notably, if $f$
preserves the orientation of $S^3$ and also $f(K)$ is oriented-equivalent to $K$, then $f$ is equivalent to the identity.

The equivalence class of $K$ is said to be: \emph{invertible}, \emph{positively amphichiral} or \emph{negatvely
  amphichiral} according to whether there exists an $f$ for which $f(K)$ is equivalent to the reverse of $K$, the mirror
of $K$ or the reverse of the mirror of $K$. A knot that has none of these properties is \emph{chiral}, and a knot that
has at least two (and therefore all three) is \emph{fully amphichiral}. This list of five exhausts all the kinds of
non-rigid symmetry that a knot may possess.

In the case of a hyperbolic knot, the non-rigid symmetries admit strictification. Specifically, if
$f\colon(S^3, K) \to (S^3, f(K))$ is an instance of a non-rigid symmetry, then there exists a diffeomorphism
$d\colon (S^3, f(K)) \to (S^3, K)$ preserving the orientations of both $S^3$ and $f(K)$, so that $d \circ f$ yields an
element of $\Homeo(M)$. By Proposition \ref{pr:generalPropAboutHyperbolic3Manifolds}, $d \circ f$ represents a class of
finite order in $\pi_0(\Diff(S^3,K))$, which admits a lift to an element $h \in \Diff(S^3, K)$ of finite order, thus
endowing $K$ with the structure of a $C_n$-symmetric knot in the terminology of this paper. Following Proposition
\ref{pr:clearStatementOfReduction}, we may assume $h$ acts by isometries on $S^3$.

For an invertible or negatively amphichiral symmetry, the element $h$ we have produced reverses the orientation of $K$,
fixing two points on the knot, and by reference to Table \ref{tab:C2Types}, must be a strong inversion or a strongly
negative amphichiral symmetry as appropriate. In the case of a positively amphichiral symmetry, the element $h$ may be a
rotoreflection of order $2n$, and so the knot admits a $C_{2n}$-symmetry of type RRef.
    
Additionally, a hyperbolic knot is fully amphichiral if and only if its full symmetry group contains two elements that
lift to a strong inversion and to a strongly negative amphichiral symmetry, which by consideration of Table
\ref{tab:D2nTypes} is equivalent to admitting the structure of a symmetry of \hyperlink{SNASI}{SNASI} type. The example
of the Figure-8 knot, whose full symmetry group is $D_4$ acting by symmetries of \hyperlink{SNASI}{SNASI} type, shows
that a positively amphichiral hyperbolic knot need not admit a strongly positive amphichiral symmetry.

\smallskip

In the case of a non-hyperbolic knot, it may not be possible to rigidify inversions or amphichiral symmetries. Knots
that are invertible but not strongly invertible have been constructed in \cite{Hartley1980} and in
\cite{Whitten1981}. The paper \cite{Hartley1980} also constructs negatively amphichiral knots that are not strongly negatively
amphichiral.

\section{Equivalences of \texorpdfstring{$G$}{G}-symmetric knots}
\label{sec:definition-type}

\subsection{Equivariant Isotopy}
\label{sec:equivIsotopy}

Suppose $L_0, L_1 \colon S^1 \to S^3$ are two knots in $S^3$, and suppose $S^3$ carries a $G$-action $\alpha$ for
which $L_0$ and $L_1$ are both invariant. There are two ways in which one might define ``equivariant isotopy'' from
$L_0$ to $L_1$. The first is the more restrictive definition.
\begin{definition} \label{def:smoothGisotopy}
  Endow $S^1$ with the $G$-action $\alpha|_{L_0}$ induced by $L_0$. A \emph{smooth $G$-isotopy} from $L_0$ to $L_1$ is a smooth map $L \colon S^1 \times [0,1] \to S^3$ such that for each $t \in [0,1]$, the map
   \[ L_t \colon S^1 \to S^3, \quad  x \mapsto L(x,t) \]
   is a smooth $G$-equivariant embedding.
\end{definition}

The second definition is this.
\begin{definition} \label{def:isotopyThroughInv}
  A smooth isotopy $L \colon S^1 \times [0,1] \to S^3$ from $L_0$ to $L_1$ is said to be an \emph{isotopy through invariant maps}
  if the embeddings
  \[ L_t \colon S^1 \to S^3, \quad x \mapsto L(x,t)\]
  have $G$-invariant image for all $t \in [0,1]$.
\end{definition}
The latter is the definition used in \cite{Freedman1995}. Note that a smooth $G$-isotopy is necessarily an isotopy through
invariant maps, but in a smooth $G$-isotopy, one has the additional condition that the action of $G$ on $S^1$ is
independent of $t$. It is not difficult to write down examples with $G=\Z$ where the two notions disagree.

Fortunately, in the case of finite groups acting on $S^1$, the two notions of equivariant isotopy essentially coincide. 
\begin{proposition} \label{pr:coincidenceOfGisotopy}
  Let $G$ be a finite cyclic or dihedral group acting smoothly on $S^3$. Suppose $L \colon S^1 \times [0,1]\to S^3$ is a smooth isotopy through
  invariant maps. Then there exists an orientation-preserving diffeomorphism $f \colon S^1 \to S^1$ and a smooth $G$-isotopy
  $L' \colon S^1 \times [0,1] \to S^3$ so that $L'_0 = L_0$ and $L'_1 = L_1 \circ f$.
\end{proposition}
\begin{proof}
  For each $t \in [0,1]$, the isotopy $L$ induces an action $\beta_t \colon G \to \Diff(S^1)$. For all $g \in G$, we obtain in this way a continuous path
  $\beta_t(g) \colon [0,1] \to \Diff(S^1)$.

  We claim that the subgroup $\ker(\beta_t)$ is independent of $t$. Suppose $\beta_{t'}(g) = e$ for some $g \in G$ and some
  $t' \in [0,1]$. In particular, $\beta_t(g)$ must be an orientation-preserving diffeomorphism of $S^1$ for all $t$. The rotation
  number $\rtn(\beta_t(g)) \in \ZZ/(n)$ (see Definition \ref{def:rtn}) is a constant function of $t$, which can be deduced from
  \cite[II~Prop.~2.7]{Herman1979}. The identity is the unique periodic orientation-preserving diffeomorphism of
  $S^1$ with rotation number $0$ (this can be deduced from \cite[II, IV]{Herman1979}), so that $\beta_t(g) = e$ for all
  $t \in [0,1]$. This proves our claim.

  We assume without loss of generality that $\beta_t \colon G \to \Diff(S^1)$ is injective for all $t$.
\medskip

  If $h\colon S^1 \times [0,1] \to S^1$ is a smooth isotopy, then
\[ 
  L \circ h \colon S^1 \times [0,1] \to S^3, \quad (s,t) \mapsto L(h(s,t), t) 
\]
  is also a smooth isotopy, and the path of $G$-actions induced by $L \circ h$ is
\[ 
  \beta^{h^{-1}} \colon G \times [0,1] \to \Diff(S^1), \quad \beta^{h^{-1}}_t(g)(s) = h_t^{-1} \circ \beta(g) \circ
  h_t(s), \quad \forall s \in S^1. 
\] 
The isotopy $L \circ h$ has the same image as $L$, and if $h_0 = \id$, then $(L \circ h)_0 = L_0$. At the
  other end, $(L \circ h)_1 = L_1 \circ h_1$. To prove the proposition, we produce a smooth isotopy
  $h\colon S^1 \times [0,1] \to S^1$ with the property that $\beta^{h^{-1}}_t$ is independent of $t$.

  Let $G^+$ denote the subgroup of those $g \in G$ for which $\beta_t(g) \in \Diff^+(S^1)$, i.e., for which the degree
  is $1$. This is independent of $t$. The group $G^+$ is cyclic. Let $n$ denote its order. There exists a
  unique $\rho \in G^+$ for which $\rtn(\beta_t(\rho)) = 1/n \in \RR/\ZZ$ for one, and therefore for all,
  $t \in [0,1]$. By considering rotation numbers, we see that $\rho$ generates $G^+$.

  If $G=\langle \rho \rangle$, we use Proposition \ref{pr:conjDiffCyclic} to produce a path $h : [0,1] \to \Diff^+(S^1)$
  so that $h_t \beta_t(\rho) h_t^{-1}$ is rotation by a $1/n$-turn. Since the rotation does not depend on $t$, the result
  is proved in this case.
  
  If instead $G \neq G^+$, choose some element $\sigma \in G \sm
  G^+$ so that $G$ is generated by $\rho, \sigma$. We now use Proposition \ref{pr:conjDiffDihedral} to produce a path
  $h: [0,1] \to \Diff^+(S^1)$ so that $h_t \beta_t(\rho) h_t^{-1}$ is rotation by a $1/n$-turn and $h_t \beta_t(\sigma)
  h_t^{-1}$ is a fixed reflection. This establishes the result.
  \end{proof}

\begin{notation}
  Since we generally care about $K \colon S^1 \to S^3$ only up to orientation-preserving reparameterization of $S^1$, we will
  say that two knots $K_0$ and $K_1$ are \emph{equivariantly isotopic} if there is an isotopy through invariant maps
  between them. In light of Proposition \ref{pr:coincidenceOfGisotopy}, such an isotopy can be elevated to a smooth
  $G$-isotopy if need be.
\end{notation}

\subsection{Equivalence of symmetric knots}
\label{sec:isoGsymmetric}

There are several senses in which $G$-symmetric knots $(K_0, \alpha_0)$ and $(K_1, \alpha_1)$ may be considered to be
equivalent. Our aim in this section is to give, in Definition \ref{def:Gequivalence}, a reasonable definition that is at
least as coarse as alternative definitions, but for which the ``type'' we define is an invariant. Our version of
$G$-equivalence is the equivariant generalization of the ordinary definition that two knots $K$, $K'$ are equivalent if
there is an orientation-preserving diffeomorphism of $S^3$ carrying $|K|$ to $|K'|$. One might instead prefer to
generalize one of the other well-known notions of knot equivalence: two knots being equivalent if they differ by a
smooth isotopy or by an ambient isotopy. In Propositions~\ref{pr:smoothImpliesAmbientGisotopy} and
\ref{pr:AmbientGisotopyImpliesGisotopy} we show that equivariantly smoothly- or ambiently-isotopic knots are $G$-equivalent in our sense.

\begin{definition} \label{def:isomorphismGsymmetric}
  If $(K_0,\alpha_0)$ and $(K_1,\alpha_1)$ are $G$-symmetric knots then a \emph{$G$-diffeomorphism} between them consists of
  an orientation-preserving diffeomorphism $d \colon S^3 \to S^3$ so that
  \begin{itemize}
  \item $\alpha_0 = \alpha_1^d$;
  \item $d (|K_0|) = |K_1|$;
  \item The map $f=K_1^{-1} \circ d \circ K_0$ is an orientation-preserving diffeomorphism of $S^1$.
  \end{itemize}
\end{definition}

We have chosen to view the group $G$ of symmetries of the knot as neither a subgroup of $\Diff(S^3)$ nor as a subgroup of
$\Diff(S^1)$, but rather as a group equipped with injections to each of these groups. This means that our notion of
equivalence should also allow for automorphisms of $G$. With this in mind, we complete our definition of $G$-equivalence.

\begin{definition} \label{def:Gequivalence} Let $(K_0, \alpha_0)$ and $(K_1, \alpha_1)$ be two $G$-symmetric knots. We
  say $K_0$ and $K_1$ are \emph{$G$-equivalent} (or simply \emph{equivalent} when $G$ is clear from the context) if
  there exists an automorphism $\phi \colon G \to G$ such that $(K_0, \alpha_0)$ is $G$-diffeomorphic to
  $(K_1, \alpha_1 \circ \phi)$.
\end{definition}

That is, we allow the following in our notion of $G$-equivalence: orientation-preserving reparameterizations of $S^3$
and of $S^1$, and relabelling of group elements of $G$. It is easy to verify that this is an equivalence relation.

\medskip
We now show that equivariantly isotopic $G$-symmetric knots are $G$-equivalent, at least when
$G$ is finite, and that ambiently $G$-isotopic knots are $G$-equivalent.

\begin{definition} \label{def:ambientGIsotopy}
    If $(K_0,\alpha)$ and $(K_1, \alpha)$ are two $G$-symmetric knots, then an \emph{ambient $G$-isotopy} from $K_0$ to
    $K_1$ is a smooth $G$-equivariant map $A \colon S^3 \times [0,1] \to S^3$, where $G$ acts trivially on $[0,1]$, and an
    orientation-preserving diffeomorphism $f \colon S^1 \to S^1$ such that:
    \begin{itemize}
        \item each function
        \[ A_t \colon S^3 \to S^3, \quad s \mapsto A(s,t) \]
        is a diffeomorphism;
        \item $A_0 = \id_{S^3}$;
        \item $A_1 \circ K_0 = K_1 \circ f$.
    \end{itemize}
  \end{definition}
As in $G$-diffeomorphisms, the map $f$ is determined by the rest of the data.

\begin{proposition} \label{pr:smoothImpliesAmbientGisotopy}
  Let $G$ be a finite group. Fix an action $\alpha \colon G \to \Diff(S^3)$. Suppose $K_0$ and $K_1$ are two $G$-symmetric knots. Then the following are equivalent:
  \begin{enumerate}
  \item \label{i:iso1} There exists an equivariant isotopy from $K_0$ to $K_1$;
  \item \label{i:iso3} There exists an ambient $G$-isotopy from $K_0$ to $K_1$.
  \end{enumerate}
\end{proposition}
\begin{proof}
  Suppose $K_0$ and $K_1$ are equivariantly isotopic. Proposition \ref{pr:coincidenceOfGisotopy} says that, at the cost of replacing $K_1$ by $K_1'= K_1 \circ h$ for some $h \in
  \Diff^+(S^1)$, we may suppose that there exists a $G$-isotopy from $K_0$ to $K_1'$.
  By the isotopy extension theorem \cite[Thm.~8.1.3]{Hirsch1973}, there exists an extension of $T$ to a
  possibly non-equivariant smooth isotopy $A' \colon S^3 \times [0,1] \to S^3$ that agrees with $T$ on $|K_0| \times [0,1]$, where
  it is $G$-equivariant. Then by \cite[Thm.~VI.3.1]{Bredon1972}, we may replace $A'$ by a $G$-equivariant isotopy $A$
  that agrees with $A'$ on the subspace of equivariance, in particular on $|K_0| \times [0,1]$. Then $A_1 \circ K_0 =
  T_1 = K_1' \circ f = K_1 \circ h^{-1} \circ f$, as required.

  Conversely, suppose \ref{i:iso3} holds. Consider $A \circ K_0 \colon S^1 \times [0,1] \to S^3$. This gives a smooth isotopy
  from $K_0$ to $K_1 \circ f$ such that for all $t \in [0,1]$, the image is $G$-invariant.
\end{proof}

\begin{proposition} \label{pr:AmbientGisotopyImpliesGisotopy}
    If $(K_0, \alpha)$ and $(K_1, \alpha)$ are ambiently $G$-isotopic symmetric knots, then they are $G$-equivalent.
\end{proposition}
\begin{proof}
    Take an ambient $G$-isotopy $A \colon S^3 \times [0,1] \to S^3$, along with its attendant $f \colon S^1 \to S^1$, and define $A_1 = d$. Since $d \colon S^3 \to S^3$ is
    $G$-equivariant, $\alpha^d = \alpha$. From the definition of ambient $G$-isotopy, we see that $d\circ K_0= K_1 \circ
    f$.    
  \end{proof}

In the case of a hyperbolic knot $K$, there is a maximal finite group $G$, the full symmetry group of $K$, for which $K$ has a $G$-symmetric
structure (see Remark \ref{rem:fullSymmHyperbolic}). We remark that any two actions of this group on knots equivalent to
$K$ are $G$-equivalent in our sense.

\begin{proposition} \label{pr:fullSymmetryAllEquivalent}
  Suppose $K_0$ and $K_1$ are equivalent hyperbolic knots and $G$ is isomorphic to the full symmetry groups of
  $K_0$ and $K_1$. Choose actions $\alpha_i : G \to \Diff(S^3, K_i)$ for $i \in \{0,1\}$ making $K_0$, $K_1$ into $G$-symmetric
  knots. Then $(K_0, \alpha_0)$ is $G$-equivalent to $(K_1, \alpha_1)$.
\end{proposition}
\begin{proof}
  Let $d$ be an orientation-preserving diffeomorphism of $S^3$ taking $K_0$ to $K_1$. Then $(K_1, \alpha_1)$ is
  $G$-equivalent to $(K_0, \alpha_1^d)$, since they are $G$-diffeomorphic.

  From Theorem \ref{pr:uniquenessOfLifting}, we know that the two subgroups $\alpha_0(G), \alpha_1^d(G) \subset \Diff(S^3, K)$ are conjugate by some $e$ in
  the identity component of $\Diff(S^3, K_0)$: say $\alpha^e_0(G) = \alpha_1^d(G)$. Then for some automorphism $\phi$ of $G$,
   $(K_0,\alpha_0)$ is $G$-diffeomorphic to $(K_1, \alpha_1^d \circ \phi)$, completing the proof.  
\end{proof}

The following fact, that any symmetry is conjugate to an isometry, is not original to us. In practice, it is usually
clear from a diagrammatic presentation of a symmetry how to view it as an isometry. We also refer the reader to
\cite[Chapters 15 and 16]{BonSie} for the case when the acting group is not required to be finite, in which case the
action need not be conjugate to an action by isometries.

  \begin{theorem} \label{pr:allGsymmetricKnotsEquivIsometry}
    Suppose $(K,\alpha)$ is a $G$-symmetric knot where $G$ is a finite group. Then $K$ is $G$-equivalent to a $G$-symmetric knot $(K', \alpha')$ where the action $\alpha' \colon G \to \Diff(S^3)$ is by isometries
   of $S^3$.
  \end{theorem}
  \begin{proof}
    This follows immediately from \cite[Theorem
    E]{Dinkelbach2009}, which furnishes an element $d \in \Diff^+(S^3)$ so that $\alpha^d$ has image in $\Og(4)$. Set $K' =
    d^{-1}\circ K$.
  \end{proof}

\section{The definition of \texorpdfstring{``type of $G$-symmetry''}{“type of G-symmetry”}}

\begin{definition} \label{def:types}
  Suppose $(K, \alpha)$ is a $G$-symmetric knot. The \emph{type of $G$-symmetry} of a $G$-symmetric knot $(K, \alpha)$ is the equivalence class of $(\alpha, \alpha|_K)$ as an element of 
  \[ \left(\ORep(G; \Diff(S^3)) \times \ORep(G; \Diff(S^1))\right)\Big/\Aut(G).\]
\end{definition}

That is, the type of $G$-symmetry consists of the information of the two actions, $G$ on $S^3$ and $G$ on the knot, taken up to orientation-preserving smooth reparameterizations of $S^3$ and $S^1$ and relabelling of elements of $G$.

\begin{proposition} \label{pr:typeDependsOnlyOnEquivClass}
  Suppose $(K,\alpha)$ and $(K',\alpha')$ are equivalent $G$-symmetric knots, where $K$ is not trivial. Then the types of $G$-symmetry of $(K,\alpha)$ and $(K', \alpha')$ agree.
\end{proposition}
\begin{proof}
    If $\phi$ is an automorphism of $G$, then the equivalent $G$-symmetric knots $(K, \alpha)$ and $(K, \alpha \circ \phi)$ have the same type, essentially by definition of types.

    Now suppose $(K_0, \alpha_0)$ and $(K_1, \alpha_1)$ are $G$-diffeomorphic symmetric knots, and let $d \in \Diff^+(S^3)$ be the diffeomorphism. Then $\alpha_0$ and $\alpha_1$ are conjugate by $d$. Also, $\alpha_0|_{K_0}$ and $\alpha_1|_{K_1}$ are conjugate by $K_1^{-1} \circ d \circ K_0$, which is required to be orientation-preserving. Therefore the types of the two knots are the same.
  \end{proof}
  
\benw[inline]{I'm not sure we care about the next observation}

The following statement is a trivial consequence of the definitions.
\begin{proposition}
  Suppose $(K, \alpha)$ is a $G$-symmetric knot, and suppose that the class of $(\beta, \gamma)$ agrees with that of $(\alpha, \alpha|_K)$ in
  \[  \left( \ORep(G; \Diff(S^3)) \times \ORep(G; \Diff(S^1))\right)\Big/\Aut(G).\]
  Then there exists a $G$-symmetric knot $(K', \alpha')$ such that $K$ is equivalent to $K'$, and such that $\alpha' = \beta$ and $\alpha'|_{K'} = \gamma$.
\end{proposition}

\section{Transversality and linking numbers}
\label{sec:transv-link-numb}

\subsection{Transversality}
\label{sec:transversality}

This is a technical result that will be needed to show that the linking number of a generalized freely $n$-periodic knot with the core of one of the axes of rotation is well-defined modulo $n$. 

\begin{proposition} \label{pr:transversality}
  Suppose $\alpha \colon G \to \Diff(S^3)$ is a faithful smooth action of a finite
  cyclic group on $S^3$. Suppose $M$ is a $G$-manifold, $F \colon M \to S^3$ is a smooth $G$-map, $K \colon S^1 \to S^3$ is a $G$-symmetric knot on
  which $G$ acts freely, and $U \supset |K|$ an open $G$-invariant neighborhood. Then there exists an equivariant isotopy
  from $K$ to some knot $K'$ so that $|K'| \subset U$ and $F$ is transverse to $|K'|$.
\end{proposition}
\begin{proof}

  
  Let $X \subset S^3$ denote the closed subset where the action of $G$ on $S^3$ is not free. Since $K$ is disjoint
  from $X$, it suffices to prove the analogous result for the ambient space $S^3 \sm X$, the smooth map $F \colon M \sm
  F^{-1}(X) \to S^3 \sm X$ and the open neighborhood $U \sm X$. The advantage this confers is that $G$ acts freely on
  $S^3 \sm X$. Let us write $q \colon S^3 \sm X \to (S^3 \sm X)/G =: Y$.

  \benw[inline]{What follows is a terse version, followed by a verbose version, of the proof. First the terse version.}
  There is a quotient knot $\bar K \colon S^1 \to Y$ and a composite $\bar F = q \circ F$. Standard transversality results
  (e.g., \cite[IV Cor.~2.4]{Kosinski1993}) allow us to deform $\bar K$ by a smooth isotopy so that its image is
  transverse to $\bar F$. By uniqueness of lifting in the covering space $q \colon S^3 \sm X \to Y$, the isotopy in $Y$ lifts
  to a $G$-isotopy from $K$ to some other knot $K'$ in $S^3 \sm X$ that is transverse to $F$, as required.
  
\benw[inline]{The verbose version is commented out here.}
\end{proof}

\subsection{Linking numbers of \texorpdfstring{$G$}{G}-equivariant knots.}
\label{sec:linking}

Let $G$ be a finite discrete group and consider two $G$-symmetric knots $K$ and $L$. An ambient isotopy of $K$ may
change the linking number of $K$ and $L$ arbitrarily by passing $K$ through $L$, but if the isotopy is required to be
$G$-symmetric as well, then the change to the linking number is more constrained: the symmetry of $K$ and $L$ may
prevent $K$ from passing through $L$,  which we record as Proposition \ref{cor:linkDifferentActions}, or the symmetry may otherwise
constrain how $K$ passes through $L$, which we record in Proposition \ref{cor:linkSimilarActions}.

\begin{proposition} \label{cor:linkDifferentActions}
  Fix a faithful action $\alpha : C_n \to \Diff^+(S^3)$. Let $K, L_1$ and $L_2$ be $C_n$-symmetric knots for this action, where $\alpha|_K$ and $\alpha|_{L_i}$ preserve the orientations. 
    Let $T$ be an equivariant isotopy from $L_1$ to $L_2$ and suppose that $\ker(\alpha|_K) \neq
    \ker(\alpha|_{L_1})$. Then \[ \link(|K|, |L_1|) = \link(|K|, |L_2|). \] 
\end{proposition}
\begin{proof}
    The kernels of $\alpha|_K$ and $\alpha|_{L_1}$ are the stabilizers of points on $|K|$ and on $|T|$,
    respectively. Since these stabilizers do not agree, $|K| \cap |T| = \emptyset$, so that the isotopy takes place in
    $S^3 \sm |K|$. It is elementary that $\link(|K|,|L_1|) - \link(|K|, |L_2|) = 0$.
\end{proof}

In the case where the actions on $K$ and $L_1,L_2$ are free, the linking number can change, but only by a multiple of
$n$. 

\begin{lemma} \label{pr:technicalIntersectionLabel}
    Fix a faithful action $\alpha \colon C_n \to \Diff^+(S^3)$. Let $K, L_1$ and $L_2$ be $C_n$-symmetric knots for this action, where $\alpha|_K$ and $\alpha|_{L_i}$ preserve the orientations. Suppose that $|K|$ is disjoint from $|L_1| \cup |L_2|$.
    Let $T$ be a smooth $C_n$-isotopy from $L_1$ to $L_2$ so that $T$ is transverse to $|K|$. Then the difference 
    \[ \link(K, L_1) - \link(K, L_2) \in \ZZ \] 
    lies in the subgroup generated by the cardinalities of the orbits in the finite $C_n$-set $|K| \cap |T|$.
\end{lemma}

\begin{proof}
  Choose a Seifert surface $\Sigma$ for $K$ that is transverse to $T$.\benw{Why can this be done?}
  The argument consists of counting signed intersection points of various kinds. We give it in homological terms, using
  the intersection product $\bullet$ of homology classes as set out in \cite[VIII \S 13]{Dold1995}.

   Begin with the classes $\lambda_i \in \Hoh_1(|L_i|)$ for $i \in \{1,2\}$, and $\sigma \in \Hoh_2(\Sigma, |K|)$ which
   are determined by the orientations of $L_i$ and $\Sigma$ (note that the orientations of $K$ and $S^3$ determine the orientation of $\Sigma$). Then the linking numbers $\link(K, L_i)$ are identified with the product $\sigma \bullet \lambda_i \in \Hoh_0(S^3) = \ZZ$. Next we need the class 
\[ 
    \tau \in \Hoh_2(|T|, |L_1| \cup |L_2|)
\]
    corresponding to an orientation on $[0,1]$, and note that $\bd(\tau) = \lambda_1 - \lambda_2 \in \Hoh_0(|L_1| \cup |L_2|)$. We can then compute the intersection product in $S^3$ 
    \[ \sigma \bullet \tau \in \Hoh_1(\Sigma \cap |T|, [|T| \cap |K|] \cup [\Sigma \cap (|L_1| \cup |L_2|)]),  \]
    which satisfies the identity
    \begin{equation} \label{eq:intersectionRelation}
         \bd(\sigma \bullet \tau) = i_* [(\bd \sigma) \bullet \tau] -  j_*[\sigma \bullet (\bd \tau)] \in \Hoh_0(\Sigma \cap (|L_1| \cup |L_2|) \cup (|T| \cap |K|)),
    \end{equation}
    where $i$ and $j$ are the inclusions of $|T| \cap |K|$ and $\Sigma \cap [ |L_1| \cup |L_2|]$ in their union, respectively. Then the composite
    \[ \Hoh_1(\Sigma \cap |T|, [|T| \cap |K|] \cup [\Sigma \cap (|L_1| \cup |L_2|)]) \overset{\bd}{\to} \Hoh_0( [|T| \cap |K|] \cup [\Sigma \cap (|L_1| \cup |L_2|)]) \to \Hoh_0(\Sigma \cap |T|) \]
    is already $0$, so that the image of $\bd(\sigma \bullet \tau)$ in $\Hoh_0(S^3)$ is $0$. Hence relation \eqref{eq:intersectionRelation} implies that the images of 
    \[  (\bd \sigma) \bullet \tau \text{\quad and \quad} \sigma \bullet (\bd \tau) = \sigma \bullet \lambda_1 - \sigma \bullet \lambda_2. \]
    in $\Hoh_0(S^3)$ agree.

    The homology class $(\bd \sigma) \bullet \tau$ lies in $\Hoh_0(|K| \cap |T|)$ by construction, and since $K$ and $T$
    are $C_n$-equivariant by orientation-preserving actions, the orientation classes $\bd \sigma$ and $\tau$ are
    $C_n$-invariant. In particular, the image of $\sigma \bullet \lambda_1 - \sigma \bullet \lambda_2$ in $\Hoh_0(S^3)$
    is a $C_n$-invariant class supported on $|K|\cap |T|$. We use the identification $\Hoh_0(S^3) = \ZZ$ to deduce that
    \[ \link(K, L_1) - \link(K, L_2) \]
    is a $\ZZ$-linear combination of homology classes of $C_n$-orbits in $|K| \cap |T|$, which establishes the result.
\end{proof}



\begin{proposition} \label{cor:linkSimilarActions}
 Fix a faithful action $\alpha \colon C_n \to \Diff^+(S^3)$. Let $K, L_1$ and $L_2$ be $C_n$-symmetric knots for this action,
 where $\alpha|_K$ and $\alpha|_{L_i}$ preserve the orientations and where $|K|$ is disjoint from $|L_1| \cup |L_2|$. 
    Let $T$ be a smooth $C_n$-isotopy from $L_1$ to $L_2$ and suppose that the action of $C_n$ on $K$ is free. Then \[
      \link(|K|, |L_1|) \equiv \link(|K|, |L_2|) \pmod{n}. \] 
\end{proposition}
\begin{proof}
  If the action of $C_n$ on $L_1$ is not free, then Proposition \ref{cor:linkDifferentActions} applies, and the result
  holds \textit{a fortiori}. We assume that the action of $C_n$ on $L_1$ is free.

  Using Proposition \ref{pr:transversality},  we can assume that $T$ is transverse to $|K|$. By hypothesis, the action
  of $C_n$ on $|K| \cap |T|$ is free. The result then follows from Lemma \ref{pr:technicalIntersectionLabel}.
\end{proof}

\section{Orthogonal representation theory of cyclic and dihedral groups}
\label{sec:orthRepTheory}

\subsection{Orthogonal representations of cyclic and dihedral groups}
\label{sec:orth-repr-cycl}
This subsection is included for reference and to fix notation. For the theory of orthogonal representations of finite groups, see \cite[\S 13.2]{Serre1977}.

We adopt the view that an orthogonal representation is a homomorphism of groups $f \colon G \to \Og(m)$. Two orthogonal
representations are \emph{isomorphic} if they are conjugate in $\Og(m)$. If $f_1, f_2$ are orthogonal representations,
then we write $f_1 \oplus f_2$ to denote their block sum.

Let $n \ge 2$.  We will view $C_n$ as a subgroup of $D_{n}$ where appropriate.

For any $i \in \ZZ/(n)$, we define
\[ R^i_n = \begin{bmatrix} \cos(2 i \pi/n) & -\sin(2 i \pi /n) \\ \sin(2 i \pi/n) & \cos(2 i \pi/n) \end{bmatrix} \]
and
\[ S = \begin{bmatrix} -1 & 0 \\ 0 & 1 \end{bmatrix}. \]

\begin{notation} \label{not:orthRepsOfCn}
  We name a succession of orthogonal representations $C_n \to \Og(m)$. The meaning of these names depend on the specific choice of generator, $\rho$. Therefore, whenever we use these names, we must have a cyclic group along with a specified generator.
  \begin{itemize}
  \item $1$, defined by $1(\rho) = 1$;
  \item $w_\sign$, defined by $w_\sign(\rho) = -1$;
  \item $w_a$ for $a \in \ZZ/(n)$, defined by $w_a(\rho) = R_n^a$.
  \end{itemize}
\end{notation}

The proofs of the following facts are elementary.
\begin{enumerate}
\item There are isomorphisms of representations $w_a \iso w_{-a}$, but otherwise the representations of Notation
  \ref{not:orthRepsOfCn} are pairwise non-isomorphic.
\item The representations $w_a$ are irreducible unless $a=0$ or (when $n$ is even) $a = n/2$. In these cases, there are
  the decompositions
  \[ w_0 = 1 \oplus 1, \quad w_{n/2} = w_\sign \oplus w_\sign.\]
\item The following constitutes a complete set of irreducible orthogonal representations of $C_n$:
  \[ \{ 1, w_\sign\} \cup \{ w_a \}_{a \in F(n) \sm \{0, n/2\}}. \]
\item The action of $\Aut(C_n) = \ZZ/(n)^\times$ on these representations by precomposition fixes $1$ and $w_{\sign}$
  and satisfies $i \cdot w_a = w_{ia}$.
\end{enumerate}

\begin{notation} \label{not:orthRepsOfD2n}
  We name a succession of orthogonal representations $D_{n} \to \Og(m)$. As in the cyclic case, these names depend on our choice of $\rho$, $\sigma$.
  \begin{itemize}
  \item $1$, defined by $1(\rho) = 1$ and $1(\sigma) = 1$;
  \item $v_\half$, defined by $v_\half(\rho) = -1$ and $v_\half(\sigma) = 1$;
  \item $v_\sigma$, defined by $v_\sigma(\rho) = 1$ and $v_\sigma(\sigma) = -1$;
  \item $v_\inv$, defined by $v_\inv(\rho) = -1$ and $v_\inv(\sigma) = -1$;
   \item $v_a$ for $a \in \ZZ/(n)$, defined by $v_a(\rho) = R_n^a$, $v_a(\sigma) = S$.
  \end{itemize}
\end{notation}

The proofs of the following facts are elementary.
\begin{enumerate}
\item There are isomorphisms of representations $v_a \iso v_{-a}$, but otherwise the representations of Notation
  \ref{not:orthRepsOfD2n} are pairwise non-isomorphic.
\item The representations $v_a$ are irreducible unless $a=0$ or (when $n$ is even) $a = n/2$. In these cases, there are
  the decompositions
  \[ v_0 = v_\sigma \oplus 1, \quad v_{n/2} = v_\inv \oplus v_\half.\]
\item The following constitutes a complete set of irreducible orthogonal representations of $D_n$:
  \[ \{ 1, v_\half, v_\sigma, v_\inv \} \cup \{ v_a \}_{a \in F(n) \sm \{0, n/2\}}. \]
\end{enumerate}

\subsection{Chirality of orthogonal representations}
\label{sec:chir-orth-repr}

Let $G$ be a finite group.
\begin{definition}
  Suppose $\phi\colon G \to \Og(m)$ is a representation. We say $\phi$ is \emph{amphichiral} if there is some $d$ of
  determinant $-1$ such that $\phi^d = \phi$. We say that $\phi$ is \emph{chiral} otherwise.
\end{definition}

There is an obvious homomorphism $r\colon \ORep(G; \Og(m)) \to \Rep(G; \Og(m))$ .

Fix some $d \in \Og(4)$ of determinant $-1$. Conjugation by $d$ induces an action of $C_2$ on $\ORep(G;
\Og(m))$. Suppose $\phi, \psi\colon G \to \Og(m)$ are two homomorphisms that are $\Og(m)$-conjugate. Then $\phi$ is
$\SO(m)$-conjugate to at least one of the two homomorphisms $\psi, \psi^d$. That is, $C_2$ acts transitively on the
inverse image $r^{-1}([\phi])$, which therefore consists of $1$ or $2$ classes.

\begin{proposition} \label{pr:chiralRepClassifier}
  Suppose $\phi\colon G \to \Og(m)$ is a homomorphism. The following are equivalent:
  \begin{enumerate}
  \item \label{pia1} $\phi$ is amphichiral;
  \item \label{pia2} $r^{-1}([\phi])$ is a singleton, i.e., the $\Og(m)$- and $\SO(m)$-conjugacy classes of $\phi$ coincide;
  \item \label{pia4} the representation $\phi$ admits an odd-dimensional subrepresentation.
  \end{enumerate}
\end{proposition}
\begin{proof}
  The equivalence of \ref{pia1} and \ref{pia2} is elementary.

  We show \ref{pia1} is equivalent to \ref{pia4}.   Suppose $\phi$ is amphichiral and let $d \in \Og(m)$ satisfy $\phi^d
  = \phi$ and $\det(\phi) = -1$. Since $d$ is an orthogonal matrix of negative determinant, it must have the eigenvalue
  $-1$ with odd multiplicity. Let $W$ denote the $-1$-eigenspace of $d$. Then $\phi|_W$ is an odd-dimensional
  subrepresentation of $\phi$. Conversely, if $\phi$ admits an odd-dimensional suprepresentation, then we may write
  $\phi = \phi|_W \oplus \phi|_{W^\perp}$ for some odd-dimensional $W \subset \RR^m$. The operation of multiplication by
  $-1$ on $W$ and by $1$ on $W^\perp$ is an orthogonal transformation of determinant $-1$ that commutes with $\phi$.
  \benw[inline]{Slightly longer proof in comments.}%
%
 \end{proof}

\subsection{Invariant round circles}
\label{sec:invariantRound}

\begin{proposition} \label{pr:whatRoundCircles}
    Suppose $C_n$ acts on $S^3 \subset \RR^4$ by means of the representation $w_a \oplus w_b$ and suppose $a \not \in \{0, n/2\}$. Then:
    \begin{enumerate}
    \item \label{rcZ} If $b=0$, then there are two kinds of $C_n$-invariant round circles: the unique circle $S^1_{zw}$ on which $C_n$
      acts trivially, and infinitely many round circles on which $C_n$ acts by an $a/n$-turn, all of which are smoothly
      $C_n$-isotopic to $S^1_{xy}$;
    \item \label{rc1} If $b \not \in \{a,-a,0\}$, then there are exactly two $C_n$-invariant round circles in $S^3$, consisting of $S^1_{xy}$ and $S^1_{zw}$;
    \item If $a=b$ or $a=-b$, then there are infinitely many $C_n$-invariant round circles in $S^3$, all of which are smoothly $C_n$-isotopic to each other.
    \end{enumerate}
\end{proposition}
\begin{proof}
  By elementary arguments, the round circles consist precisely of the proper nonempty intersection
  $S^3 \cap (\vec v + W)$ where $\vec v$ is a $C_n$-fixed vector and $W$ is a $C_n$-invariant subspace of $\RR^4$.
  Cases \ref{rcZ} and \ref{rc1} follow directly
  from these observations, since there are only two $C_n$-invariant $2$-dimensional subspaces in each case.

  If $a = b$ or $a=-b$, there are no nonzero fixed vectors and the action of $\rho$ on $\RR^4$ is by an \emph{isoclinic} or \emph{equiangular} rotation. There are two families of these, according to whether $a=b$ or $a=-b$, and they appear in \cite[pp.~36-38]{DuVal1964} where they are termed left- and right-screws. They are most easily understood using quaternions. We identify $\HH = \RR^4$ by means of a standard basis $\{1, i, j , k\}$. There exists a unit quaternion $q$ so that the action of $\rho$ on $\HH$ is either $x \mapsto q  x $ or $x \mapsto x \bar q$. The proofs are much the same in either case, so we suppose $x \mapsto qx$. 
  
  The possibilities $q = \pm 1$ are ruled out by the condition $a \not \in \{0,n/2\}$. Therefore we know $q \in \HH \sm \RR$. We determine the invariant $2$-dimensional subspaces of $\HH$, which correspond to the invariant round circles. The set $\{1,q\} \subset \HH$ generates a commutative sub-$\RR$-algebra of $\HH$, which is necessarily isomorphic to $\CC$, and so may be written $A(q) = \RR \oplus \RR q$. Suppose $X \subset \HH$ is a $2$-dimensional subspace invariant under the operation $x \mapsto qx$. Let $t \in X$ be a nonzero element, and consider $Xt^{-1}$, which is $2$-dimensional, invariant and contains both $1$ and $q\cdot 1 = q$. It follows that $Xt^{-1} = A(q)$, and we deduce that the $2$-dimensional invariant subspaces of $\HH$ are precisely the subspaces of the form $A(q)t$ as $t$ ranges over $\HH^\times$. There are infinitely many different subspaces in this family, and therefore infinitely many invariant round circles. Since $\HH^\times$ is path-connected, we may find a smooth path from any given $t\in \HH^\times$ to $1$, and we deduce that all such circles are equivariantly isotopic to the circle $A(q) \cap S^3$.
\end{proof}

\begin{proposition} \label{pr:whatRoundCircles2}
    Suppose $n$ is even $C_n$ acts on $S^3$ by means of the representation $w_a \oplus w_\half \oplus 1$ and suppose $a \neq
    0$. Then there are two kinds of $C_n$-invariant round circles: the unique circle $S^1_{zw}$ on which $C_n$
      acts by reflection, and infinitely many round circles on which $C_n$ acts by an $a/n$-turn, all of which are smoothly
      $C_n$-isotopic to $S^1_{xy}$.
\end{proposition}
The proof is very similar to that of case \label{rc0} of Proposition \ref{pr:whatRoundCircles} and we omit it.

\subsection{Chirality of \texorpdfstring{$\Og(4)$}{O(4)}-representations of cyclic and dihedral groups}
\label{sec:chiralityO(4)reps}

It is not hard to enumerate all isomorphism classes of $\Og(4)$-representations of the cyclic and dihedral groups $C_n$ and $D_{n}$. Among them, only the following classes are chiral, following Proposition \ref{pr:chiralRepClassifier}:
\begin{itemize}
\item The $C_n$-representations in the classes $w_a \oplus w_b$, where $\{a,b\} \cap \{0, n/2\} =\emptyset$;
\item The $D_n$-representations in the classes $v_a \oplus v_b$, where $\{a,b\} \cap \{0, n/2\}=\emptyset$.
\end{itemize}

\begin{proposition} \label{pr:chiralCyclicDihedralO4Reps}
    The representations $w_a \oplus w_b$ and $w_{a'} \oplus w_{b'}$ are $\SO(4)$-conjugate if and only if $\{a, b\} =
    \{a',b'\}$ or $\{a,b\} = \{-a', -b'\}$. Similarly, the representations $v_a \oplus v_b$ and $v_{a'} \oplus v_{b'}$
    are $\SO(4)$-conjugate if and only if $\{a, b\} = \{a',b'\}$ or $\{-a', -b'\}$. 
\end{proposition}
\begin{proof}
    We will prove only the case of $C_n$. The other case is very similar. Standard representation theory says that $w_a
    \oplus w_b$ is isomorphic to $w_{a'} \oplus w_{b'}$ if and only if $a=\pm a'$ and $b=\pm b'$, up to a reordering of
    summands. It remains to distinguish which of the isomorphic representations $w_{\pm a} \oplus w_{\pm b}$ are
    $\SO(4)$-equivalent to $w_a \oplus w_b$.

    Let $x$ be the diagonal matrix with entries $(-1,1,1,1)$, of determinant $-1$, and $y$ the diagonal matrix with
    entries $(-1,1,-1,1)$, of determinant $1$. Since $w_a \oplus w_b$ is chiral representation, it is not $\SO(4)$-conjugate to
    $(w_a \oplus w_b)^x = w_{-a} \oplus w_b$ or to $(w_a \oplus w_b)^{xy} = w_a \oplus w_{-b}$, but it is
    $\SO(4)$-conjugate to $(w_a \oplus w_b)^y = w_{-a} \oplus w_{-b}$.
\end{proof}

\begin{proposition} \label{pr:chiralityIsTopological}
    Fix a positive integer $n$ and suppose $a,b$ are two elements of $\ZZ/(n)$ that generate the unit ideal and such
    that $w_a \oplus w_b$ is chiral. Then $w_a \oplus w_b$ is not $\Diff^+(S^3)$-conjugate to $w_a \oplus
    w_{-b}$.\benw{Is this ok?}
\end{proposition}
\begin{proof}
    Observe that the condition that $\{a,b\}$ is disjoint from $\{0,n/2\}$ implies that $n \ge 3$. Write $C_n = \langle \rho \mid \rho^n = e \rangle$. 

    For brevity, write $\beta = w_a \oplus w_b$ and $\beta' = w_a \oplus w_{-b}$. Assume for the sake of contradiction that there exists some $d \in \Diff^+(S^3)$ such that $\beta^d = \beta'$.

    The intersections of the $xy$- and $zw$-coordinate planes with $S^3 \subset \RR^4$ form a Hopf link $S^1_{xy} \cup S^1_{zw}$. These circles are oriented in such a way that:
    \begin{itemize}
        \item the linking number of $S^1_{xy}$ and $S^1_{zw}$ is $1$;
        \item the isometries $\beta(\rho)$ and $\beta'(\rho)$ restrict to an $a/n$-turn in the positive direction of $S^1_{xy}$;
        \item the isometry $\beta(\rho)$ restricts to a $b/n$-turn in the positive direction of $S^1_{zw}$, and $\beta'(\rho)$ restricts to a $-b/n$-turn.
    \end{itemize}
    
    We may apply $d$ to $S^1_{xy}$ and $S^1_{zw}$ to obtain two oriented unknots $K_{xy}$ and $K_{zy}$ that have linking number $1$ in $S^3$, and on which $\beta(\rho)$ acts by an $a/n$-turn and a $-b/n$-turn, respectively. We work only with the $\beta$-action of $C_n$ on $S^3$ from now on.

    According to the main theorem of \cite{Freedman1995}, each of $K_{xy}$ and $K_{zy}$ is equivariantly isotopic to a round circle in $S^3$. Any round circle on which $\rho$ acts by an $a/n$-turn is equivariantly isotopic to $S^1_{xy}$, and any round circle on which $\rho$ acts by a $-b/n$-turn is equivariantly isotopic to $-S^1_{zw}$, by Proposition \ref{pr:whatRoundCircles}. 
    
    The equivariant isotopies can change the linking numbers of the unknots in restricted ways: if $a$ and $b$ generate
    different ideals in $\ZZ/(n)$, then Proposition \ref{cor:linkDifferentActions} applies to say the linking number is
    unchanged, whereas if $a$ and $b$ generate the same ideal in $\ZZ/(n)$, then it must be the unit ideal and
    Proposition \ref{cor:linkSimilarActions} applies. In either case
    \[ -1 = \link(S^1_{xy}, -S^1_{zw}) \equiv \link(K_{xy}, K_{zw}) \pmod{n} \]
    but $\link(K_{xy}, K_{zw}) = 1$, so that this is a contradiction.
\end{proof}

\begin{corollary} \label{cor:chiralityIsTopologicalDih}
    Fix a positive integer $n$ and suppose $a,b$ are two elements of $\ZZ/(n)$ that generate the unit ideal and such
    that $v_a \oplus v_b$ is chiral. Then $v_a \oplus v_b$ is not $\Diff^+(S^3)$-conjugate to $v_a \oplus v_{-b}$
\end{corollary}
\begin{proof}
    If $n \le 2$, there is nothing to prove, so we may assume $n\ge 3$. Observe that $v_a \oplus v_b$ may be restricted to the unique order-$n$ cyclic subgroup $C_n \subset D_n$, to give $w_a \oplus w_b$. If $v_a \oplus v_b$ is $\SO(4)$-conjugate to $v_a \oplus v_{-b}$, then $w_a \oplus w_b$ is conjugate to $w_a \oplus w_{-b}$, contradicting Proposition \ref{pr:chiralityIsTopological}.
\end{proof}

\section{Reduction to isometry groups}
\label{sec:reduction}

\begin{proposition} \label{pr:reduceDiff+toSO} If $G$ is a finite cyclic or dihedral group and
  $\beta,\beta'\colon G \to \Og(4)$ are injective homomorphisms that are conjugate by an orientation-preserving
  diffeomorphism $d \in \Diff^+(S^3)$, then $\beta$ and $\beta'$ are conjugate by an element of $\SO(4)$.
\end{proposition}
\begin{proof}
    Using \cite{Cappell1999}, we can find some $x \in \Og(4)$ such that $\beta^x = \beta'$. 
     
    If $\beta$ is amphichiral, then this is enough because even if $x$ has determinant $-1$, we can compose $x$ with
    some $d$ of determinant $-1$ for which $\beta^d = \beta$, so that $dx \in \SO(4)$ satisfies $\beta^{dx} = \beta'$.

    If $\beta$ is chiral and $G$ is cyclic, then we know $\beta$ is in the $\ORep(G; \Og(4))$-class of $w_a \oplus w_b$
    for some pair $a,b \in \ZZ/(n)$ where $\{a,b\} \cap \{0, n/2\} = \emptyset$. We know that $\beta'$ is either in the
    $\ORep(G; \Og(4))$-class of $w_a \oplus w_b$ or in the class of $w_a \oplus w_{-b}$, by Proposition
    \ref{pr:chiralCyclicDihedralO4Reps}, but the latter possibility is ruled out by Proposition
    \ref{pr:chiralityIsTopological}. Therefore $\beta$ and $\beta'$ are conjugate by an element of $\SO(4)$.

    If $G$ is dihedral, then we argue in the same way, but use Corollary \ref{cor:chiralityIsTopologicalDih} in place of
    Proposition \ref{pr:chiralityIsTopological}.
\end{proof}

\benw[inline]{Candidate for removal, or at least moving to the appendix}

\begin{proposition} \label{pr:diffSOS1} If $G$ is a finite cyclic or dihedral group group and
  $\beta,\beta'\colon G \to \Og(2)$ are homomorphisms that are conjugate by a diffeomorphism $d \in \Diff^+(S^1)$, then
  $\beta$ and $\beta'$ are conjugate by an element of $\SO(2)$.
  \end{proposition}
\begin{proof}
    If $g \in G$, then $d\Fix(\beta(g); S^1) = \Fix(\beta'(g); S^1)$. Since an element of $\Og(2)$ lies in $\SO(2)$ if
    and only if its fixed set is either empty or all of $S^1$, we know that $\beta(g) \in \SO(2)$ if and only if
    $\beta'(g) \in \SO(2)$.

    First, let us consider the case where the image of $\beta$ lies in $\SO(2)$. Necessarily, $G$ is cyclic. Write
    $\rho$ for a generator of $G$, so that $\beta(\rho)$ is the matrix of an $a/n$-turn of $S^1$, for some
    $a \in \ZZ/(n)^\times$. The rotation number of $\beta(\rho)$, i.e., $a/n$, is an invariant of $\beta(\rho)$ up to
    conjugacy by $\Diff^+(S^1)$, \cite[II Prop.~2.10]{Herman1979}. It follows immediately that
    $\beta(\rho) = \beta'(\rho)$.

    Second, consider the general case where $\im(\beta) \not \subset \SO(2)$. Write $G^+$ for the subgroup
    $\beta^{-1}(\SO(2))$ of $G$, which must be a cyclic subgroup, although it may be trivial. Let $\sigma$ be an element
    such that $\beta(\sigma) \not \in \SO(2)$, i.e., $\beta(\sigma)$ is a reflection, as is $\beta'(\sigma)$. It is easy
    to write down a rotation matrix $d$ such that $d^{-1} \beta(\sigma) d = \beta'(\sigma)$. Note also that $d$, being a
    rotation matrix, commutes with all other rotations. In particular $d^{-1} \beta(g) d =\beta(g) = \beta'(g)$ for all
    $g \in G^+$. The group $G^+$ is of index $2$ in $G$, so that $G$ is generated by $G^+$ and $\sigma$. It follows that  $d^{-1} \beta(g) d = \beta'(g)$ for all $g \in G$, as required.    
    \end{proof}

Taken together, the above results allow us to replace the larger groups $\Diff(S^3)$ and $\Diff(S^1)$ with the smaller groups of isometries $\Og(4)$ and $\Og(2)$.

\begin{theorem} \label{pr:typeByOrth}
  If $G$ is a finite cyclic or dihedral group, then the inclusions $\Og(2) \subset \Homeo(S^1)$ and $\Og(4) \subset S^3$ induce bijections
  \begin{align} 
    \Psi_1\colon \ORep(G; \Og(2)) & \overset{=}{\to} \ORep(G; \Diff(S^1)) \\
    \Psi_3\colon \ORep(G; \Og(4)) & \overset{=}{\to} \ORep(G; \Diff(S^3)), \label{eq:comparisonOfOReps}
  \end{align}
  and similarly for the subsets of faithful oriented representations.
\end{theorem}
\begin{proof}
    We first give the arguments for $\Psi_3$.

    Since $\SO(4) = \Og(4) \cap \Diff^+(S^3)$, the function $\Psi_3$ is well defined. 
    
    We next show that it is surjective. If $\beta\colon G \to \Diff(S^3)$ is a homomorphism, then according to \cite[Theorem E]{Dinkelbach2009}, $\beta$ is conjugate to a homomorphism $\beta' \colon G \to \Og(4)$. The conjugating element may be taken to lie in $\Diff^+(S^3)$: if it does not, post-conjugate by an orientation-reversing element of $\Og(4)$. Then $\Psi_3([\beta']) = [\beta]$ by construction, so $\Psi_3$ is surjective.

    Next we show $\Psi_3$ is injective. If $\beta, \beta'\colon G \to \Og(4)$ are $\Diff^+(S^3)$-conjugate  then they are $\SO(4)$-conjugate according to Proposition \ref{pr:reduceDiff+toSO}, and in particular, they represent the same class in $\ORep(G; \Og(4)$. This concludes the case of $\Psi_3$.

    The arguments for $\Psi_1$ are similar. It is well defined for the same reason as for $\Psi_1$. By Proposition \ref{pr:orientedConjugateDihedral}, if $\beta\colon G \to \Diff(S^1)$ is a homomorphism, then $\beta$ is $\Diff^+(S^1)$-conjugate to $\beta' \colon G \to \Og(2)$. This shows that $\Psi_1$ is surjective. For injectivity, we argue was we did for $\Psi_3$, using Proposition \ref{pr:diffSOS1}.

    A class $[f]$ in $\ORep(G; D, D^+)$ is faithful if and only if a representative $f\colon G \to D$ is. Since $\Og(2) \to \Diff(S^1)$ and $\Og(4) \to  \Diff(S^3)$ are injective, $\Psi_1$ and $\Psi_3$ send faithful classes to faithful classes.
  \end{proof}

  \section{Automorphisms of dihedral groups}
  \label{sec:autDih}

  Let $n\ge 2$ be an integer. In this section, we make some observations about $\Aut(D_n)$, since this group appears in
  the definition of types of dihedral symmetry. This section consists of elementary results, and is included chiefly to
  fix notation.
  
\begin{notation} \label{not:Kn}
  We define a particular group $J_n \subset \Aut(D_{n})$ to consist of the
  automorphisms fixing $\rho$. The group $J_n$ consists of the automorphisms $f_a(\sigma) = \rho^a \sigma$, and is cyclic
  of order $n$ since $f_a \circ f_b = f_{a+b}$.
\end{notation}

We observe that conjugation by powers of $\rho$ defines a subgroup of $J_n$, which is of index $2$ if $n$ is even and is
all of $J_n$ if $n$ is odd.

\begin{proposition} \label{pr:AutD2n}
  Precomposition by $\phi \in \Aut(D_{n})$ defines a transitive action on $\FORep(D_{n}; \Og(2))$ and the
  stabilizer of the class of the representation $v_1$ is $J_n$.
\end{proposition}
If $n \ge 3$, then $J_n$ is normal, so the stabilizers of the classes of each $v_a$ is $J_n$.
\begin{proof}
  We divide this into two cases: $n=2$ and $n>2$.

  Suppose $n=2$. There are $3$ classes in $\FORep(D_2; \Og(2))$, distinguished by which of
  the $3$ elements $\rho$, $\sigma$ and $\rho \sigma$ has image $-I_2$---the other $2$ elements map to reflections in
  each case. The
  automorphism group of $D_4$ is the symmetric group acting on $\{\rho, \sigma, \rho\sigma\}$, so that $\Aut(D_4)$
  certainly acts transitively on $\FORep(D_4; \Og(2), \SO(2))$.

  Suppose $n>2$, then $\FORep(D_{n}; \Og(2))$ consists of the inverse images under
  $r\colon \ORep(D_{n}; \Og(2)) \to \Rep(D_{n}; \Og(2))$ of the faithful classes $v_a$ where $a \in
  F(n)^\times$. These are the classes of homomorphisms $v_a \colon D_{n} \to \Og(2)$ for all $a \in \ZZ/(n)^\times$.

  To see that $\Aut(D_{n})$ acts transitively on this set, consider the automorphisms of $D_{n}$ given by
  \[\rho \mapsto \rho^i,\quad \sigma \mapsto \sigma \]
  as $i$ ranges over $\ZZ/(n)$. Precomposition by these automorphisms acts transitively on $\FORep(D_{n}; \Og(2))$.

  For all $n \ge 2$, the stabilizer of $w_1$ consists of all the automorphisms $f$ of $D_{n}$ with the property that
  $f(\rho) =\rho$, which is $J_n$.
\end{proof}

\begin{proposition} \label{pr:KnActOnDihedralReps}
  The action of the cyclic group $J_n$ by precomposition
  on the set of representations in Notation \ref{not:orthRepsOfD2n}
  is as follows:
  \begin{itemize}
  \item $f_1 \cdot 1 = 1$;
  \item $f_1 \cdot v_\half = v_\half$;
  \item $f_1 \cdot v_\sigma = v_\inv$;
  \item $f_1 \cdot v_\inv = v_\sigma$;
  \item $f_1 \cdot v_a = v_a$.
  \end{itemize}
\end{proposition}
The proof is by direct calculation.
  
  \section{A simplification of the definition of type}
\label{sec:simpl-defin-type}

We can now begin to  classify all types of $G$-symmetry of nontrivial knots (where $G$ is finite cyclic or dihedral)
using the orthogonal representation theory of the groups $G$.

First, we state the way in which the problem has been reduced to representation theory.
\begin{proposition}\label{pr:clearStatementOfReduction}
  Suppose $(K',\alpha')$ is a nontrivial $G$-symmetric knot, where $G$ is a finite group. Then $(K', \alpha')$ is
  $G$-equivalent to a knot $(K,\alpha)$ where $\alpha$ and $\alpha|_K$ are actions by isometries. They type of $(K,
  \alpha)$, and therefore of $(K', \alpha')$ is identified by under the obvious map with the equivalence class of
  $(\alpha, \alpha|_K)$ in
  \begin{equation}
  \left. \Big[\FORep(G; \Og(4)) \times \FORep(G; \Og(2))\Big]\middle/ \Aut(G)\right..\label{eq:3}
\end{equation}
\end{proposition}
\begin{proof}
  To produce $(K, \alpha)$, use \cite{Dinkelbach2009} to replace $\alpha'$ by an action of $G$ on $S^3$ by isometries,
  and use Proposition \ref{pr:conjDiffCyclic} or Proposition \ref{pr:conjDiffDihedral} (with constant paths $g_t$ and
  $s_t$) to find a reparametrization of $S^1$ so that $\alpha|_K$ becomes an action by isometries.

  Given $(K, \alpha)$, Theorem \ref{pr:typeByOrth} says that the type of $(K, \alpha)$ is the
  equivalence class in \eqref{eq:3}.
\end{proof}

The set \eqref{eq:3} may seem complicated, but it simplifies considerably for every $G$ we consider:
\begin{enumerate}
\item $G=C_2$. If $G$ is cyclic of order $2$, then $\Aut(G)$ is trivial, and the type of $(K, \alpha)$ is the equivalence
  class of $(\alpha, \alpha|_K)$ in
  \[ \FORep(G; \Og(4)) \times \FORep(G; \Og(2)). \]
\item \label{i:simp2} $G=C_n$ where $n \ge 3$. In this case,
  \[ \FORep(C_n; \Og(2)) = \{ w_a \mid a \in F(n)^\times \},\]
and this is a principal homogeneous set for the group $\Aut(C_n) \iso
\ZZ/(n)^\times$. If we choose $w_1$ as a distinguished point, then we identify $\Aut(C_n)$ with $\FORep(C_n; \Og(2))$.
The effect of this identification is that we relabel the elements of $C_n$ so that the distinguished generator $\rho$
acts on the knot $K$ by $w_1$, i.e., a $1/n$-turn in the positive direction. Once this is done, the type of $(K, \alpha)$ is precisely the class of $\alpha$ in
  \[ \FORep(G; \Og(4)).\]

\item \label{i:simp3} $G=D_n$, where $n \ge 2$. Write $D_n = \langle \rho, \sigma \mid \rho^n = \sigma^2 = (\rho \sigma)^2 = e \rangle$.
  In this case, In this case, the $\SO(2)$-conjugacy classes of faithful representations
  $\alpha|_K \colon C_n \to \Og(2)$ consist of the classes
  \[ \FORep(D_n; \Og(2))) =  \{[v_a] \mid a \in \ZZ/(n)^\times \}.\]
  The action of $\Aut(D_n)$ on this set is transitive but not free. Choose $v_1$ as a distinguished element of this set,
  and let $J_n \subset \Aut(D_n)$ denote the stabilizer of $v_1$. By the transitivity of the action of $\Aut(D_n)$ on $\FORep(D_n;
  \Og(2))$, these choices yield an identification
  \[  \left. \left[ \FORep(G; \Og(4)) \times \FORep(G; \Og(2))\right] \middle/ \Aut(D_n) \right. \leftrightarrow
    \FORep(G; \Og(4))\big/J_n.\]
  That is, if we relabel the elements of $D_n$ so that $\rho$ acts as a $1/n$-turn in the positive direction on the knot, then the type of $(K, \alpha)$ is
  the equivalence class of $\alpha$ in
  \[ \FORep(D_n; \Og(4))/J_n.\]
\end{enumerate}

\begin{example} \label{ex:rhoAction}
  We made identifications in \eqref{i:simp2} and \eqref{i:simp3} above. We explain the geometric implications of these
  here, by means of an example.
  \begin{figure}[h]
    \centering
    \includegraphics[width=0.33\textwidth]{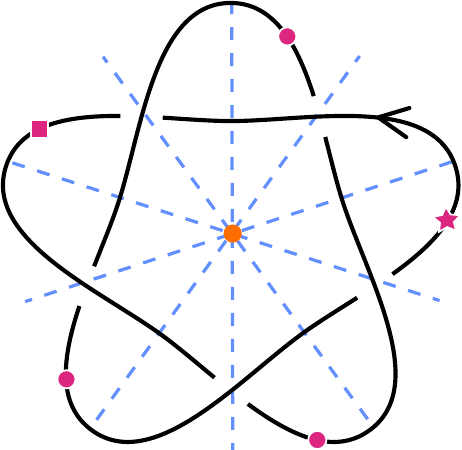}
    \caption{The torus knot $T_{2,5}$, an orientation, some marked axes of rotation and $5$ marked points on the knot.}
    \label{fig:T25}
  \end{figure}
  In Figure \ref{fig:T25}, we depict an oriented torus knot $T_{2,5}$. Also indicated are some axes of rotation: an
  orange dot in the centre of the diagram indicates an axis of $5$-fold rotational symmetry of the knot (one should
  imagine the axis emanating from the page); each of the five blue dashed line-segments indicate axes of order-$2$
  rotations.

  The indicated order-$5$ rotations of $S^3$ act freely on the knot $T_{2,5}$, whereas the order-$2$ rotations of $S^2$
  act by reflection on $T_{2,5}$. What is depicted is an action by the dihedral group $D_5$, where the subgroup
  $C_5$ acts by order-$5$ rotations of both $S^3$ and $T_{2,5}$, and the other nontrivial elements act by strong inversions.

  The group $C_5$ is generated by a distinguished element $\rho$. It is a matter of convention which of the order-$5$ rotations is
  designated to be the action by $\rho$. Mark a point (it does not matter which) on the knot: in Figure
  \ref{fig:T25} the marked point is indicated by a violet star-shape $\star$. The marked point has an orbit under the
  $C_5$-action on $T_{2,5}$, the other points in the orbit are marked by violet dots $\bullet$ or a violet square
  $\sqbullet$. Follow the knot in the direction of the orientation from the marked point, $\star$, to the next point in
  the orbit, $\sqbullet$.  Our convention in either \ref{i:simp2} and \ref{i:simp3} above is that the group element
  named $\rho$ acts by the rotation that takes $\star$ to $\sqbullet$.

  There is no intrinsic way to distinguish between the strong inversions, so we have no convention to declare any of
  them to be action by ``$\sigma$''. The type of $D_5$-symmetry indicated in Figure \ref{fig:T25} will be named ``Strongly
  Invertible Periodic $(2)$'', or ``\hyperlink{SIP}{SIP}--$(2)$'', in Section \ref{sec:potent-types-nontr}. The
  quantity $2$ in ``\hyperlink{SIP}{SIP}--$(2)$'' is defined only in $F(5)^\times$, i.e., up to a sign modulo $5$. It
  indicates that $\rho$ acts by a $2/5$-turn on $S^3$ for some orientation of the axis or alternatively that the linking number of the knot $T_{2,5}$
  with the axis of $5$-fold rotation is $2$. That axis does not have a canonical orientation, and reversing the
  orientation changes the sense of the rotation from positive to negative (i.e., to a $-2/5 = 3/5$-turn) and changes the
  linking number from $2$ to $-2$.
\end{example}

\section{Further restrictions on the types of symmetry}
\label{sec:furtherRestriction}

There are a number of further restrictions on the symmetries of knots. We state the following simple condition as a proposition for later reference.
\begin{proposition} \label{pr:fixedLociNested}
  Suppose $(K, \alpha)$ is a $G$-symmetric knot. For all $g \in G$, there is an inclusion $\Fix(g; K) \subset \Fix(g;
  S^3)$, and in particular
  \[  \dim \Fix(g; K) \le \dim \Fix(g; S^3). \]
\end{proposition}
Note that $\dim \Fix(g; K)$ and $\dim \Fix(g; S^3)$ depend only on the conjugacy classes of $\alpha|_K$ and $\alpha$ respectively.
  
\begin{proposition} \label{pr:noS2}
Suppose $(K,\alpha)$ is a $G$-symmetric knot and there exists some $g \in G$ for which $ \Fix(g; S^3)\approx S^2$. Then
\begin{enumerate}
    \item $X \cap |K|$ consists of $2$ points. In particular, $g$ does not act without fixed points on $K$, and
    \item $K$ is not a prime knot.
\end{enumerate}
\end{proposition}
\begin{proof}
Note that $|K| \cap X$ consists of an even number of points since $|K|$ is a closed loop. Consider a pair of these points that are connected by an arc of $|K|$ in one component of $S^3 \sm X$. Then by applying $\alpha(g)$, we see that these two points are also connected by an arc in the other component of $S^3\sm X$. Since $|K|$ has only one connected component, the union of the $2$ points and the $2$ arcs is the entirety of $|K|$.

 We conclude that $X$ separates $|K|$ into a mirrored pair of components and hence $K$ is not prime.
\end{proof}

\section{Types of \texorpdfstring{$C_2$}{C₂}-symmetric knots}
\label{sec:Cyc2}

The group $C_2 = \langle \tau \mid \tau^2 = e\rangle$ may be viewed as either a special case of $C_n$ (when $n=2$) and a
special case of $D_n$ (when $n=1$).

Suppose $(\alpha, \alpha|_K)$ is the type of $C_2$-symmetry of a nontrivial knot. We may assume that $\alpha$ and $\alpha|_K$ are actions by isometries. The set of possible types of nontrivial $C_2$-symmetric knots is a
subset of
\[\FORep(C_2; \Og(4)) \times \FORep(C_2 ; \Og(2)). \]

The set $\FORep(C_2; \Og(2))$ consists of $2$ classes: a class where $\tau$ acts as a reflection on $S^1$, reversing the
orientation and fixing $2$ points, and a class where the action of $\tau$ on $S^1$ is the antipodal action, preserving
the orientation and acting freely. Similarly, $\FORep(C_2; \Og(4))$ consists of $4$ classes, categorized according to
the multiplicity of the eigenvalue $1$.

At first glance, therefore, there might be as many as $2 \times 4=8$ types of $C_2$-symmetry, but $2$ possibilities
can be eliminated.

If the eigenvalue $1$ appears with
multiplicity $3$ in $\alpha(\tau) \in \Og(4)$, then the fixed locus $\Fix(\tau; S^3)$ is homeomorphic to $S^2$ and so
Proposition \ref{pr:noS2} applies to say $\Fix(\tau; K)$ is nonempty, and that $K$ cannot be prime. If the eigenvalue
$1$ appears with multiplicity $0$ in $\alpha(g)$, then the action on $S^3$ is the antipodal action, which is free, so
that the action on $S^1$ must also be free.

There are no further restrictions on the representations $\alpha$ and $\alpha|_K$, since all $6$ of the remaining types
are shown to occur in Section \ref{sec:exampl-texorpdfstr-s}. We arrive at the classification of Table
\ref{tab:C2Types}.

\section{Types of \texorpdfstring{$C_n$}{Cₙ}-symmetric knot for \texorpdfstring{$n \ge 3$}{n≥3}}
\label{sec:typesCn}

Let $n\ge 3$. Suppose $(K, \alpha)$ is a nontrivial $C_n$-symmetric knot. We may assume $\alpha$ and $\alpha|_K$ are
actions by isometries. Using the identification of Subsection \ref{sec:simpl-defin-type}, the type of $\alpha$ is the
class of $\alpha$ in
\[ \FORep(C_n ; \Og(4)). \] After making this identification, the preferred generator $\rho \in C_n$ acts
by $1/n$-turn in the positive direction on $K$.

\begin{proposition} \label{pr:faithfulR4RepsCyclic}
Let $n \ge 3$. Then any faithful orthogonal representation of $C_n$ on $\RR^4$ lies in the isomorphism class of exactly
one of the following:
\begin{enumerate}[label=(Cyc\Alph*) ]
\item \label{C1} $w_a \oplus w_b$ where $\{a,b\}$ are an unordered pair in $F(n)$ and $\langle a , b \rangle = \ZZ/(n)$;
\end{enumerate}
or if $n$ is even:
\begin{enumerate}[label=(Cyc\Alph*)]
\setcounter{enumi}{1}
    \item \label{C2} $w_a \oplus w_{\sign} \oplus 1$ where $a \in F(n)$ and $\langle a \rangle = \ZZ/(n)$;
\end{enumerate}
or if $n \equiv 2 \pmod 4$:
\begin{enumerate}[label=(Cyc\Alph*)]
\setcounter{enumi}{2}
\item \label{C3} $w_a \oplus w_{\mathrm{sign}} \oplus 1$ where $a \in F(n)$ and $\langle a \rangle = \langle 2 \rangle \subset \ZZ/(n)$.
\end{enumerate}
All these representations are amphichiral except those in case \ref{C1} where $\{a, b \} \cap \{0, n/2\} = \emptyset$, which are chiral.
\end{proposition}

  The classification follows from the observations of Section \ref{sec:orth-repr-cycl} and an analysis of cases.

\begin{proposition} \label{pr:noCycC}
    Let $n \ge 3$. Suppose $(K, \alpha)$ is a $C_n$-symmetric knot where $\alpha$ is an orthogonal representation. Then $\alpha$ is not in any equivalence class of type \ref{C3}.
\end{proposition}
\begin{proof}
  Suppose for the sake of contradiction that $\alpha$ is in the class \ref{C3}--$(a)$. The action $\alpha|_K$ of $C_n$
  on $K$ is free, whereas $\rho^{n/2}$ acts by reflection on $S^3$, fixing a $2$-sphere. This contradicts Proposition
  \ref{pr:noS2}.
\end{proof}

The action on $S^3$ may be conjugated so that it is by a representation $G \to \Og(4)$ of the form \ref{C1} or \ref{C2}
above.  For consistency with the case of $n=2$, where ``periodic'' and ``freely periodic'' are distinguished, we split
\ref{C1} into two classes. 

\begin{enumerate}
\item \emph{$C_n$-periodic symmetries}, denoted \hyperlink{Per}{Per}--$(a)$. These correspond to some of case \ref{C1}
  above; specifically, where the action of $C_n$ on $S^3$ is
  $\SO(4)$-conjugate to $w_a \oplus w_0$, in which $\rho$ acts on $S^3$ by an $a/n$-turn around the circle $S^1_{zw}$. It must
  be the case that $a \in \ZZ/(n)^\times$. Since $w_a \oplus w_0$ is amphichiral, types \hyperlink{Per}{Per}--$(a)$ and
  \hyperlink{Per}{Per}--$(-a)$ are equivalent. 

  Corollary \ref{cor:AllDih1Cyc1TypesArise} implies that all symmetries of these types arise as symmetries of torus
  knots. The $2$-periodic symmetries, \hyperlink{2P}{2P}, constitute the
  degenerate case of this one when $n=2$.

\item \emph{$C_n$-generalized-freely-periodic symmetries}, denoted \hyperlink{GFPer}{GFPer}--$(a,b)$. These correspond to the rest of case \ref{C1} above, where the action of $C_n$ on $S^3$ is
  $\SO(4)$-conjugate to $w_a \oplus w_b$, where $(a,b)$ together generate the unit ideal in $\ZZ/(n)$ and neither
  is $0$. In $w_a \oplus w_b$, the action
  of $\rho$ on $S^3$ is by a double rotation: an $a/n$-turn around $S^1_{zw}$ and a $b/n$-turn around $S^1_{xy}$. Except when
  $2a=n$ or $2b=n$, the representation $w_a \oplus w_b$ is chiral, so that type \hyperlink{GFPer}{GFPer}--$(a,b)$ is
  equivalent to \hyperlink{GFPer}{GFPer}--$(c,d)$ if and only if one of the following holds
  \[ (a,b) = (c,d), \quad (a,b) = (d,c), \quad (a,b) = (-c,-d), \quad (a,b) = (-d, -c). \]
  The set $T(n)$ is defined in Notation \ref{not:Tn} exactly so that it enumerates the possible type of $C_n$-generalized-freely-periodic
  symmetries.

  Corollary \ref{cor:AllDih1Cyc1TypesArise} implies that all symmetries of these types arise as symmetries of torus
  knots. Freely-$2$-periodic symmetries, \hyperlink{F2P}{F2P}, constitute the
  degenerate case of this one when $n=2$.

\item \emph{$C_n$-rotoreflectional symmetries}, denoted by
  \hyperlink{RRef}{RRef}--$(a)$. These correspond to case \ref{C2} above. These arise only when $n$ is even. Here the action of $C_n$ on $S^3$ is
  $\SO(4)$-conjugate to $w_a \oplus w_\sign \oplus 1$, where $a\in \ZZ/(n)^\times$. In $w_a \oplus w_\sign \oplus 1$, the action
  of $\rho$ on $S^3$ is by a rotoreflection consisting of an $a/n$-turn around $S^1_{zw}$ followed by a reflection
  across the invariant circle $S^1_{xy}$. The representation is amphichiral, so that type \hyperlink{RRef}{RRef}--$(a)$ is
  equivalent to \hyperlink{RRef}{RRef}--$(-a)$.

  Examples of these types of symmetry are constructed in Section \ref{sec:exAmphichiral}.  Strongly-positive amphichiral symmetries,
  \hyperlink{SPAc}{SPAc}, constitute the degenerate case of this one when $n=2$.
\end{enumerate}

These types of symmetry, along with some of the information above, are listed in Table \ref{tab:CnTypes}.

  \benw[inline]{This text is better later, for determining the actions, not describing them:
  In the case of the representations $w_a \oplus w_a$ where $a \in \ZZ/(n)^\times$, the action of $\rho$ on $\RR^4$ has infinitely many pairs of orthogonal eigenspaces of dimension $2$. Take any such pair and consider the restricted action to $S^3$, where the pair of eigenspaces now determines a pair of linked circles on which the action is by rotation by a $a/n$-turns on both. These circles can be oriented so that the angle of rotation is positive and less than $\pi$. The linking number again distinguishes between the types of $C_n$-symmetry \protect\hyperlink{Per}{Per}--$(a,a)$ and \protect\hyperlink{Per}{Per}--$(a,-a)$}

\subsection{Detecting different types}
\label{sec:detectingDifferentCnTypes}

Our aim here is to give criteria that allow one to classify a given $C_n$-symmetric knot $(K, \alpha)$ by observation
from a symmetric knot diagram.

We remind the reader that the implication of our conventions, as in Section \ref{sec:simpl-defin-type} is that if $C_n$
acts on a knot $K \colon S^1 \to S^3$, then the preferred generator $\rho$ of $C_n$ is the group element that acts as a
$1/n$-turn on $S^1$ via $K$.

\begin{proposition} \label{pr:linkingAndWa} Let $n\ge 3$. Suppose $(K,\alpha)$ is a $C_n$-symmetric knot. Suppose the
  action $\alpha$ can be decomposed as $w_a \oplus q \colon C_n \to \Og(4)$, and that $|K|$ is disjoint from $S^1_{zw}$. Then
  $\link(K, S^1_{zw}) \equiv a \pmod{n}$.
\end{proposition}
\begin{proof}
    There exists a $C_n$-deformation retraction of $S^3 \sm S^1_{zw}$ onto $S^1_{xy}$. Write $f\colon S^3 \sm S^1_{zw} \to
    S^1_{xy}$ for the retraction. One interpretation of $\link(|K|, S^1_{zw})$ is as the homology class $K_*[S^1] \in
    \Hoh_1(S^3 \sm S^1_{zw}; \ZZ) = \ZZ$. Since $f_*$ is an isomorphism on homology, the linking number is precisely the
    degree of $f \circ K \colon S^1 \to S^1_{xy}$.

    In our convention, $\rho$ is the generator that acts with rotation number $1/n$ on the knot, and in the
    representation $w_a$, the generator $\rho$ acts with rotation number $a/n$. Since $f \circ K$ commutes with the
    action of $\rho$, Proposition \ref{pr:rotDegreeDConjugacy} applies to tell us that
    \[ \deg(f \circ K) \frac{1}{n} = \frac{a}{n} \]
    in $\RR/\ZZ$. That is, $\link(|K|, S^1_{zw}) \equiv a \pmod n$.
\end{proof}

With this in hand, we can give a short algorithm for determining the type of a $C_n$-symmetric knot $(K, \alpha)$ when
the action on $S^3$ is by isometries. The problem of determining the types when $n=2$ can be solved by counting
the dimensions of fixed loci, which are given in Table \ref{tab:C2Types}. Therefore we concentrate on the cases
$n \ge 3$.
\begin{enumerate}
\item Suppose the action of $C_n$ on $S^3$ is not orientation-preserving. Then the action is by rotoreflections, so the
  type is \hyperlink{RRef}{RRef}--$(a)$ for some $a \in \ZZ/(n)$.

  To determine $a$,  argue as follows. There is a unique
  $C_n$-invariant round circle $X$ that functions as the axis of symmetry for the rotoreflections, so the action of
  $C_n$ is by reflections of $X$ fixing two points. Necessarily, the knot $K$ is disjoint from $X$. In some choice of
  coordinates, the action of $C_n$ on $S^3$ is $w_a \oplus w_{\sign} \oplus 1$ and $X= S^1_{zw}$. Using Proposition
  \ref{pr:linkingAndWa}, we deduce that $\link(K, X) = a \in \ZZ/(n)$.
\item Suppose the action of $C_n$ on $S^3$ is by simple rotations, so there is a $C_n$-fixed axis $X\homeo
  S^1$. Then the type is \hyperlink{Per}{Per}--$(a)$ for some $a \in \ZZ/(n)$.

 To determine $a$,  argue as follows. Necessarily, the knot $K$ is disjoint from $X$. In some choice of
  coordinates, the action of $C_n$ on $S^3$ is $w_a \oplus w_0$ and $X= S^1_{zw}$. Using Proposition
  \ref{pr:linkingAndWa}, we deduce that $\link(K, X) = a \in \ZZ/(n)$.
\item The third possibility is that the action of $C_n$ on $S^3$ is by double rotations, so the type is
  \hyperlink{GFPer}{GFPer}--$(a,b)$ for some $a,b \in \ZZ/(n)$. This is the case precisely when the $C_n$-action on $S^3$
  is by orientation preserving transformations with no global fixed points.

  In some choice of coordinates, the action of $C_n$ is by $w_a \oplus w_b$. It is possible to find two disjoint round
  circles $X$ and $Y$ such that $\link(X,Y) = 1$: by Proposition \ref{pr:whatRoundCircles},  when $a \neq \pm b$, there
  is a unique such pair $S^1_{xy}$ and $S^1_{zw}$; when $a=\pm b$, there are infinitely many such pairs.

  Orient these circles so that they have linking number $1$. Using Proposition \ref{sec:transversality}, we may adjust
  the knot $K$ by some $C_n$-isotopy to be disjoint from $X$ and $Y$. Proposition \ref{pr:whatRoundCircles} tells us
  that $X$ is smoothly $C_n$-isotopic to at least one of $S^1_{zw}$ or $S^1_{xy}$ and $Y$ is smoothly $C_n$-isotopic to
  the other. In fact, when $a \neq \pm b$, one of them must be $S^1_{zw}$ and the other must be $S^1_{xy}$. By use of
  Proposition~\ref{cor:linkSimilarActions} and Proposition \ref{pr:linkingAndWa}, we deduce that (without loss of
  generality)
  \[ \link(X, K) = a, \quad \link(Y, K) = b \quad \in \ZZ/(n). \]
\end{enumerate}

\begin{remark} \label{rem:subdivisionsOfFPer}
  In the case of \hyperlink{GFPer}{GFPer} symmetries, the dimensions of the fixed sets of subgroups of $C_n$ depend on the
  values of $a$ and $b$. We draw attention to two cases:
  \begin{itemize}
  \item If $a$ and $b$ are both relatively prime to $n$, then an \hyperlink{GFPer}{GFPer}--$(a,b)$ action is truly a free
    action of $C_n$ on $S^3$, in that $\Fix(G; S^3) = \emptyset$ for all non-identity subsets $G \subset C_n$. Any other
    case might more correctly be called a ``semi-periodic symmetry'' of the knot, as in \cite[Definition
    1]{Paoluzzi2018}.
  \item At the other extreme, it may be that $ab=0 \in \ZZ/(n)$, or equivalently that $(a,b) \equiv (n_l \lambda_l , n_r \lambda_r) \pmod n$ where $n_ln_r = n$
    and $\gcd(n_l , \lambda_l) = \gcd(n_r, \lambda_r) = 1$. In this case, $\rho^{n_r}$ and $\rho^{n_l}$ generate cyclic
    subgroups of $C_n$ of orders $n_l$ and $n_r$, respectively, and there is a decomposition $C_n = C_{n_l} \times
    C_{n_r}$. The restriction of symmetry to each of the two factor subgroups is periodic. Knots with symmetries of this
    type have been studied in \cite{Guilbault2021}, where they are termed \emph{biperiodic}. 
  \end{itemize}

\end{remark}

\section{Types of \texorpdfstring{$D_{n}$}{Dₙ}-symmetric knot}
\label{sec:typesDn}

\subsection{Action of the automorphisms of \texorpdfstring{$D_{n}$}{Dₙ}}
\label{sec:types-dihedral-knot}

Suppose $(K, \alpha)$ is a nontrivial $D_n$-symmetric knot. Using the identification of Subsection
\ref{sec:simpl-defin-type}, we can assume that the type of $\alpha$ is
the class of $\alpha$ in
\[ \FORep(D_n ; \Og(4))\big/ J_n. \]
Here $\rho \in D_n$ acts by a $1/n$-turn in the positive direction on $K$, and $\sigma \in D_n$ acts by reflection of $K$.

\begin{proposition} \label{pr:ActionDihedral}
Let $n \ge 2$. Up to precomposition by an automorphism in $J_n \subset \Aut(D_{n})$, any faithful representation
$D_{n} \to \Og(4)$ lies in the isomorphism class of exactly one of the following:
\begin{enumerate}[label=(Dih\Alph*) ]
    \item \label{D1} $v_a \oplus v_b$ where $\{a,b\}$ is an unordered pair in $F(n)$ such that $\langle a , b \rangle = \ZZ/(n)$;
    \item \label{D2} $v_a \oplus 1 \oplus 1$ where $a \in F(n)$ is such that $\langle a \rangle = \ZZ/(n)$; 
    \item \label{D3} $v_a \oplus v_\sigma \oplus v_\sigma$ where $a \in F(n)$ is such that $\langle a \rangle = \ZZ/(n)$;
\end{enumerate}
or, if $n$ is even:
\begin{enumerate}[label=(Dih\Alph*) ]
\setcounter{enumi}{3}
    \item \label{D4} $v_a \oplus v_\half \oplus 1$ where $a \in F(n)^\times$;
    \item \label{D5} $v_a \oplus v_\half \oplus v_\sigma$ where $a \in F(n)^\times$;
    \item \label{D6} $v_a \oplus v_\half \oplus v_\half$ where $a \in F(n)^\times$;
\end{enumerate}
or, if $n \equiv 2 \pmod 4$:
\begin{enumerate}[label=(Dih\Alph*) ]
  \setcounter{enumi}{6}
\item \label{D7} $v_a \oplus v_\half \oplus 1$ where $a \in F(n)$ is such that $\langle a \rangle = \langle 2 \rangle$;
\item \label{D8} $v_a \oplus v_\half \oplus v_\sigma$ where $a \in F(n)$ is such that $\langle a \rangle = \langle 2 \rangle$;
\item \label{D9} $v_a \oplus v_\half \oplus v_\half$ where $a \in F(n)$ is such that $\langle a \rangle = \langle 2
  \rangle$;
  \end{enumerate}
  or, if $n=2$:
  \begin{enumerate}[label=(Dih\Alph*) ]
    \setcounter{enumi}{9}
  \item \label{D10} $v_\half \oplus v_\half \oplus v_\half \oplus v_\sigma$;
  \item \label{D11} $v_\half \oplus v_\half \oplus v_\sigma \oplus v_\sigma$;
  \item \label{D12} $v_\half \oplus v_\sigma \oplus v_\sigma \oplus v_\sigma$.
  \end{enumerate}  
All these representations are amphichiral except for those of Type \ref{D1} when $\{a , b \} \cap \{0, n/2\} = \emptyset$, which are chiral.
\end{proposition}

  The proof is by a lengthy, but uncomplicated, exhaustion of cases, based on the description of irreducible orthogonal
  representations of dihedral groups in Section \ref{sec:orth-repr-cycl}. The chirality statement follows immediately from Proposition \ref{pr:chiralRepClassifier}.

  Note that $v_\half$ and $v_\inv$ are exchanged by the action of $J_n$, by Proposition \ref{pr:KnActOnDihedralReps},
  and that $v_\half \oplus v_\inv = v_{n/2}$. We make an arbitrary choice to prefer $v_\half$ to $v_\inv$, and since we
  are listing representations only up to the action of $J_n$, the representation $v_\inv$ does not appear the list
  above. As a consequence of our choice, the inequality $\dim_\RR \Fix(\sigma ; \RR^4) \ge \dim_\RR \Fix(\rho \sigma; \RR^4)$ holds in all cases.

\subsection{Potential types of nontrivial \texorpdfstring{$D_{n}$}{Dₙ}-symmetric knots}
\label{sec:potent-types-nontr}

Let $n\ge 2$. Suppose $(K,\alpha)$ is a nontrivial $D_n$-symmetric knot and the action $\alpha$ is by isometries on $S^3$. By convention, $\rho$
acts on $K$ by a $1/n$-turn and $\sigma$ acts on $K$ by a reflection. Up to $D_n$-equivalence, we may suppose the action
$\alpha$ is drawn from exactly one of the families listed in Proposition \ref{pr:ActionDihedral}. We now explore
ways in which the geometry of the situation imposes further restrictions on $\alpha$.

The amphichiral representations \ref{D7}, \ref{D8} give us $\Fix(\rho; S^3)\homeo S^2$ but $\Fix(\rho, K) = \emptyset$,
which violates Proposition \ref{pr:noS2}. These do not correspond to actions on any knot. The amphichiral
representations \ref{D10}---\ref{D12} give us $\Fix(\rho\sigma; S^3) = \emptyset$. By Proposition
\ref{pr:fixedLociNested}, these representations do not correspond to actions on any knot. Proposition \ref{pr:noS2}
constrains the possibilities for representations of type \ref{D2}, \ref{D4}, \ref{D6} and \ref{D9} as we now explain.

\begin{proposition} \label{pr:subtleAMustBe1Cases}
  Suppose $n \ge 3$. Suppose $(K,\alpha)$ is a nontrivial $D_n$-symmetric knot where $\alpha$ is an orthogonal representation. Suppose
  $\Fix(\sigma; S^3)$ is a $2$-dimensional sphere $X$. Then $\rho^2$ acts by nontrivial rotations on
  $S^3$ and the axis of rotation, $Y$, is $D_n$-invariant. The axis $Y$ is disjoint from $K$ and can be oriented so
  that $\link(Y, K) = 1$.
\end{proposition}
\begin{proof}
  The condition on $\Fix(\sigma; S^3)$ corresponds to $\dim_\RR \Fix(\sigma; \RR^4)=3$. We need only consider the cases
  where the action of $D_n$ is $\SO(4)$-conjugate to \ref{D2}, \ref{D4}, \ref{D6} or \ref{D9}. In each case, the
  $\sigma$-fixed subspace of $\RR^4$ is the subspace spanned by the last $3$ basis vectors. We verify that the action of
  $\rho^2$ is by rotations in each case, and the axis $Y$ is the intersection of $S^3$ with the span of the last $2$ coordinate
  basis vectors in $\RR^4$. This axis is $D_n$-invariant in each case by inspection. 

  Since $Y$ is fixed by an element $\rho^2$ that acts freely on $K$, the axis must be disjoint from the knot. The
  $\sigma$-fixed sphere $X$ is divided into two by the axis $Y$, and either hemisphere is a Seifert surface for $Y$. By
  Proposition \ref{pr:noS2}, the knot $K$ meets the Seifert surface in one point, the situation being symmetric under
  the action of $\sigma$. Therefore the linking number $\link(K, Y) = \pm 1$, and we may choose the orientation on $Y$
  to make this quantity positive.  
\end{proof}

We can determine the parameter $a$ that arises in
the representation by calculating the linking number of $K$ with the axis for the rotation or the rotoreflection given
by $\rho$. This is the algorithm in Subsection \ref{sec:detectingDifferentCnTypes}. Proposition
\ref{pr:subtleAMustBe1Cases} says that $a=1$ if the representation of $D_n$ is of type \ref{D2}, \ref{D4} or \ref{D6}.

We can now eliminate type \ref{D9} altogether. 
\begin{proposition} \label{pr:noD9} Let $m \ge 1$ be an integer. Suppose $(K, \alpha)$ is a $D_{2m}$-symmetric knot
  where $\alpha$ is an orthogonal representation. Then $\alpha$ is not in any equivalence class of type \ref{D9}.
\end{proposition}
\begin{proof}
  Suppose for the sake of contradiction that such a symmetric knot exists. We may assume the homomorphism $\alpha \colon D_{2m} \to \Og(4)$ is exactly
  $v_a \oplus v_\half \oplus v_\half$ where $\langle a \rangle = \langle 2 \rangle$ as ideals in $\ZZ/(2m)$.

  Suppose first that $m \ge 2$.  Proposition \ref{pr:subtleAMustBe1Cases} applies, so that the linking number of $K$ with
  $S^1_{zw}$, the axis of rotation of $\alpha(\rho^2)$ is $1$.  If we restrict the action to $C_{2m} \subset D_{2m}$, then
  Proposition \ref{pr:linkingAndWa} tells us that $a=1$, contradicting the fact that
  $\langle a \rangle = \langle 2 \rangle$.

  If $m=1$, then the fixed points for the action of $\rho$ on $S^3$ form a circle that meets the $2$-sphere
  $\Fix(\sigma; S^2)$ in two points, which are the fixed locus $\Fix(\rho \sigma; S^2)$. Since $\rho\sigma$ acts on $K$
  by a SNAc, it must be the case that these two points lie on $K$, whereas $\rho$ acts by a periodicity, so that the
  superset $\Fix(\rho; S^2)$ is disjoint from $K$, which is a contradiction. Therefore this case also cannot arise.
 \end{proof}

\subsection{The type of a \texorpdfstring{$D_{n}$}{Dₙ}-symmetric knot}
\label{sec:an-enumeration-types}

Here we name and describe the types of $D_n$-symmetric knots, and explain how to determine the type in practice. Much of
the information is summarized in Table \ref{tab:D2nTypes}. We assume throughout that all $D_n$ symmetric knots
$(K, \alpha)$ have been repalced by $D_n$-equivalent knots where the actions $(\alpha, \alpha|_K)$ are by isometries.

Let us assume $n \ge 2$, the case of $D_1$ having been handled in Section \ref{sec:Cyc2}. In all cases $D_n$ is the
group generated by an element $\rho$, that acts as a $1/n$-turn on $K$, and another element, $\sigma$, that acts as an
orientation-reversing symmetry of $K$.

The action on $S^3$ may then be conjugated so that it is by a representation $G \to \Og(4)$ of the form
\ref{D1}--\ref{D8} above.

\begin{enumerate}
\item \textit{Strongly Invertible Generalized Freely Periodic}, denoted \hypertarget{SIFP}{SIFP}--$(a,b)$. These correspond to representations \ref{D1} where $a$ and $b$ are nonzero. Here $\rho$ acts on $S^3$ by double rotations, so that the
  restriction of the action to $C_n$ is of type \hyperlink{GFPer}{GFPer}--$(a,b)$. The element $\sigma$ acts on $K$ by a
  strong inversion, fixing an axis $X_\sigma \homeo S^1$ in $S^3$ that meets $K$ in two points. The order-$2$
  elements $\rho^i \sigma$ are all also strong inversions, and $\rho$ cyclicly permutes the various axes $X_{\rho^i
    \sigma}$. For all $n \in \NN$ and all applicable $(a,b)$ and example of this kind of symmetry may be found among the
  torus knots: see Corollary \ref{cor:AllDih1Cyc1TypesArise}.  See also Figure \ref{fig:10_155} for an example where $n=2$.
\item \textit{Strongly Invertible Periodic}, denoted \hypertarget{SIP}{SIP}--$(a)$. These correspond to representations \ref{D1} where
  one of $a$ or $b$ is $0$, and so they may be understood as the degenerate cases of SIFP where the action of $\rho$ is
  by rotations rather than double rotations. Again, the elements $\rho^i \sigma$ all act by strong inversions and $\rho$
  permutes the axes of these symmetries. It is notable that in this case, the axes of the strong inversions
  $\rho^i \sigma$ all meet at $2$ points, which necessarily lie on the axis of the rotations $\rho^i$. That is, $\alpha$
  has $2$ global fixed points on $S^3$. For all $n \in \NN$ and all applicable $a$, an example of this kind of symmetry
  can be found among the torus knots (Corollary \ref{cor:AllDih1Cyc1TypesArise}) or by restricting the
  $D_{2n}$-symmetric knot in Proposition \ref{pr:snasi} to the dihedral subgroup generated by $\rho^2, \sigma$. See also
  Figure \ref{fig:5_2} for an example where $n=2$.
  \item \textit{Strongly Negative Amphichiral Periodic} \hypertarget{SNAP}{SNAP}--$(a)$. These correspond to representations \ref{D3}. Here $\rho$ acts on $S^3$ by rotation, so that the
  restriction of the action to $C_n$ is of type \hyperlink{Per}{Per}--$(a)$. The order-$2$ elements $\rho^i \sigma$ are strongly negative amphichiral: each one reverses the orientation of both $S^3$ and $K$, and leaves fixed a pair
  of points $Y_{\rho^i \sigma}$ on the knot. The action of $\rho$ cyclically permutes the pairs of points $Y_{\rho^i\sigma}$. For all $n \in \NN$ and all applicable $a$, an example of this kind of symmetry can be found by
  restricting the $D_{2n}$-symmetric knot in Proposition \ref{pr:snasi} to the dihedral subgroup generated by $\rho^2, \rho \sigma$. See also Figure \ref{fig:12a_868} for an example where $n=2$.
  \item Strongly Negative Amphichiral, Strongly Invertible, Periodic: (\hypertarget{SNASI}{SNASI}--$(a)$). These correspond to representations \ref{D5} and can arise only when $n$
    is even. The action of $\rho$ on $S^3$ is by rotoreflections, so that the restriction of the type to $C_n$ is of
    type \hyperlink{RRef}{RRef}--$(a)$. The order-$2$ elements $\rho^i \sigma$ belong to two conjugacy classes,
    depending on whether $i$ is even or odd. In one class, the elements act by strong negative amphichiralities, and in
    the other they act by strong inversions. Up to an automorphism of $D_n$, we may suppose that the action of $\sigma$
    is by strong inversion and of $\rho \sigma$ is by strong negative amphichirality.

    Let $X_{\rho^i \sigma}$ denote the axes of the rotation (when $i$ is even) or rotoreflection (when $i$ is odd) given
    by the action of  $\rho^i \sigma$ on $S^3$. Then $\rho$ permutes the $X_{\rho^i \sigma}$. As in case SNAP, the axes
    have a pair of common fixed points, which are also on the axis of $\rho$. The rotoreflections $\rho$ and
    $\rho^i\sigma$ (when $i$ is odd) interchange the two fixed points, while the rotations $\rho^i \sigma$ (when $i$ is
    even) preserve them.

    For all $n \in 2\NN$ and all applicable $a$, an example of this kind of symmetry is constructed in Proposition
    \ref{pr:snasi}. See Figures \ref{fig:Dih5-1}, \ref{fig:Dih5-3} for constructions indicative of the general
    cases. See also Figure \ref{fig:10_99} for an example where $n=2$.
   
\end{enumerate}

The following three types apply only to composite knots.
\begin{enumerate}
  \setcounter{enumi}{4}
  \item \label{i:dB} \textit{Type \hypertarget{DihB}{DihB} symmetry}. The action of $\rho$ on $S^3$ is by
    rotations, and the actions of the elements $\rho^i\sigma$ are by reflections across $2$-dimensional spheres
    $X_{\rho^i \sigma}$. The spheres $X_{\rho^i \sigma}$ are cyclically permuted by $\rho$.

  \item \label{i:dD} \textit{Type \hypertarget{DihD}{DihD} symmetry}. The action of $\rho$ on $S^3$ is by
    rotoreflections.  The order-$2$ elements $\rho^i \sigma$ belong to two conjugacy classes,
    depending on whether $i$ is even or odd. In one class, the elements act by $S^2$-reflections, and in
    the other they act by strong inversions. Up to an automorphism of $D_n$, we may suppose that the action of $\sigma$
    is by $S^2$-reflection and of $\rho \sigma$ is by strong inversion.
  \item \label{i:dF} \textit{Type \hypertarget{DihF}{DihF} symmetry}. The action of $\rho$ on $S^3$ is by
   double rotations that act by a $1/n$-turn in one direction and a $1/2$-turn in another.  The order-$2$ elements $\rho^i \sigma$ belong to two conjugacy classes,
    depending on whether $i$ is even or odd. In one class, the elements act by $S^2$-reflections, and in
    the other they act by strong negative amphichiralities. Up to an automorphism of $D_n$, we may suppose that the action of $\sigma$
    is by $S^2$-reflection and of $\rho \sigma$ is by strong negative amphichiralities.
  \end{enumerate}

  \begin{construction} \label{cons:nonPrimeExamples}
    Since they do not apply to prime knots, we do not give many examples in this paper of symmetric knots of types
    \hyperlink{DihB}{DihB}, \hyperlink{DihD}{DihD} or \hyperlink{DihF}{DihF}. It is not difficult to produce examples of
    composite knots with these types of symmetry, however. We outline one method here. We begin with a knot $K$ equipped
    with a $C_2$-symmetry of type $S^2$-reflection, strong inversion, or strong negative amphichirality. Each of these symmetries fixes two points on the knot, and we mark one such point.

    If $K$ has $S^2$-reflection symmetry, take the connect-sum of $n$ copies of the knot $K$ at the marked
    point spaced evenly, as in the form of a bracelet, as indicated by the ellipses in Figure \ref{fig:patternDihX}. This
    produces a knot with $D_n$-symmetry of type \hyperlink{DihB}{DihB}.

    If $K$ has a strong inversion, we follow the pattern above with a modification. The integer $n$ must be even. We
    take the connect-sum of $n/2$ copies of $K$ alternating with $n/2$ copies of the reflection of $K$ across the plane
    of the diagram, spaced evenly, as in the form of a bracelet. This produces a knot with $D_n$-symmetry of type \hyperlink{DihD}{DihD}: the diagram is constructed
    so that it is symmetric under rotation by a $1/n$-turn followed by reflection across the plane of the diagram. It also
    has $S^2$-reflection symmetries by virtue of the strong inversion of $K$, and the composite of an $S^2$-reflection
    and the rotoreflection by a $1/n$-turn is the strong inversion on each of the summand knots.

    The case where $K$ has a SNAc is similar to the above. Again, the integer $n$ must be even. We
    take the connect-sum of $n/2$ copies of $K$ alternating with $n/2$ copies of the reverse of $K$\benw{check!}, spaced
    evenly as in the form of a bracelet. This produces a knot with $D_n$-symmetry of type \hyperlink{DihD}{DihD}: the diagram is constructed
    so that it is symmetric under rotation by a $1/n$-turn followed by reflection across the plane of the diagram. It also
    has $S^2$-reflection symmetries by virtue of the strong inversion of $K$. An example is given in Figure
    \ref{fig:DihF} of the result of this construction applied with $n=4$ to the figure-eight knot and a SNAc.
  \end{construction}

   \begin{figure}
      \centering
      \begin{minipage}[t]{0.45\textwidth}
        \centering
      \includegraphics[width=0.9\textwidth]{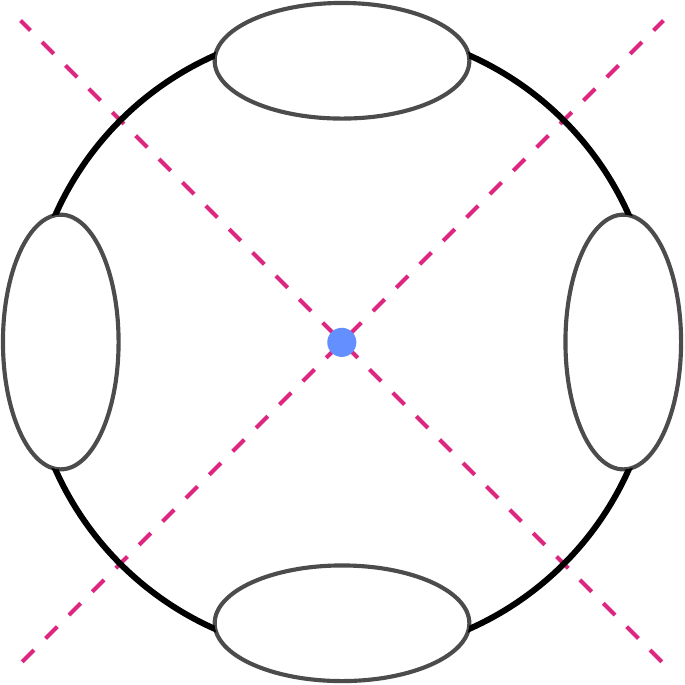}
      \caption[A pattern for producing symmetric knots of types \protect\hyperlink{DihB}{DihB},
      \protect\hyperlink{DihD}{DihD} or \protect\hyperlink{DihF}{DihF}.]{A pattern for producing symmetric knots of
        types \protect\hyperlink{DihB}{DihB}, \protect\hyperlink{DihD}{DihD} or \protect\hyperlink{DihF}{DihF}. Copies of a $C_2$-symmetric knot (perhaps with mirroring) are inserted in the ovals.  The
        conjugates of $\sigma$ act by reflection across planes indicated by dashed violet lines. The action of $\rho$ is by rotation
        around the central blue axis, perhaps followed by a further symmetry depending on the symmetry of $K$. }
      \label{fig:patternDihX}
    \end{minipage}
    \hfil
  \begin{minipage}[t]{0.45\textwidth}
    \centering
    \includegraphics[width=0.9\textwidth]{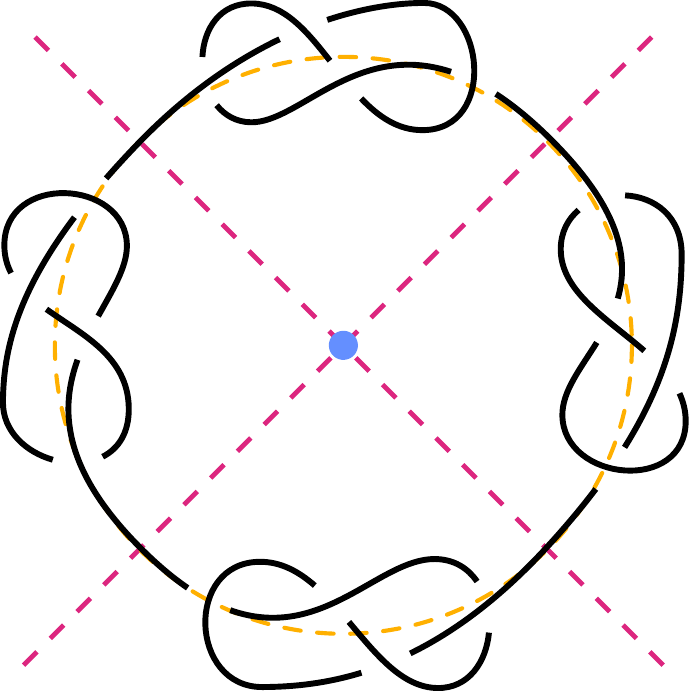}
    \caption[A composite knot with $D_4$-symmetry of type \protect\hyperlink{DihF}{DihF}.]{A composite knot with $D_4$-symmetry of type \protect\hyperlink{DihF}{DihF}. The actions of
      $\sigma$, $\rho\sigma$ are by reflection across the planes indicated by dotted violet lines. The action of $\rho$ is by $\pi/2$-rotation around the central blue axis, followed by a $\pi$-rotation around the dashed orange circle. }
    \label{fig:DihF}
  \end{minipage}  
    \end{figure}
  \vfill \eject

\afterpage{
  \clearpage \thispagestyle{empty} 
  \newgeometry{left=2cm, right=2cm, bottom=1.5cm}
  \begin{landscape}
  \begin{longtable}{|c|c|c|c|c|c|c|c|} \hline Name & Label & $\Fix(\tau; K)$ & $\Fix(\tau; S^3)$ & Prime $K$? & Good Diagram\footnotemark[1] \\
    \hline
    \rowcolor{highlight} Free $2$-periodicity & \hypertarget{F2P}{F2P} & $\emptyset$ & $\emptyset$ & yes & no\\
    \rowcolor{highlight} Strong positive amphichirality & \hypertarget{SPAc}{SPAc} & $\emptyset$ & $S^0$ &  yes & yes \\
    \rowcolor{highlight} $2$-Periodicity & \hypertarget{2P}{2P} & $\emptyset$ & $S^1$ & yes & yes\\
    \rowcolor{highlight} Strong negative amphichirality & \hypertarget{SNAc}{SNAc} & $S^0$ & $S^0$ & yes & yes \\
    \rowcolor{highlight} Strong inversion & \hypertarget{SI}{SI} & $S^0$ & $S^1$ & yes & yes \\
    $2$-sphere reflection & \hypertarget{2R}{2R} & $S^0$ & $S^2$ & no & yes\\
    \hline
    \caption[Types of $C_2$-symmetric knots.]{Types of $C_2$-symmetric knots. The fixed sets are given up to homeomorphism.}
    \label{tab:C2Types}\\
  \end{longtable}

\begin{longtable}{|c|c|c|c|c|c|c|c|c|} \hline
     Name & Repr. & Defined  & $\rho$-action on $S^3$ & Quantification & $\Fix(\rho; S^3)$ &Prime $K$? & Good Diagram\footnotemark[1] &$n=2$ \\ \hline
     \rowcolor{highlight} \hypertarget{GFPer}{GFPer}--$(a,b)$ & \ref{C1}--$(a,b)$,\, $b \neq 0$ & All $n$ & Double rotation by $(a/n, b/n)$& $T(n)$ & $\emptyset$ & yes & no & F2P \\
      \rowcolor{highlight} \hypertarget{Per}{Per}--$(a)$ & \ref{C1}--$(a,0)$ & All $n$ & Rotation by $(a/n)$& $F(n)^\times$ & $S^1$ & yes & yes &  2P \\
 \rowcolor{highlight} \hypertarget{RRef}{RRef}--$(a)$ & \ref{C2}--$(a)$ & Even $n$ & Rotoreflection by $a/n$& $F(n)^\times$ & $S^0$ & yes & yes & SPAc \\
        \hline 
    \caption[Types of $C_n$-symmetric knots for $n \ge 3$.]{Types of $C_n$-symmetric knots for $n \ge 3$. The fixed sets are given up to homeomorphism.}
    \label{tab:CnTypes}\\
\end{longtable}

\begin{longtable}{|c|c|c|c|c|c|c|c|c|c|} \hline
  Name & Repr. & Defined & $\rho$-type & Quantification & $\sigma$-type & $\rho\sigma$-type & Prime $K$? & Good Diagram\footnotemark[1] & $n=1$ \\ \hline
  \rowcolor{highlight} \hyperlink{SIFP}{SIFP}--$(a,b)$ & \ref{D1}--$(a,b)$ & All $n$ & \hyperlink{GFPer}{GFPer}--$(a,b)$ & $T(n)$ & SI & SI  & yes & no & SI \\
  \rowcolor{highlight} \hyperlink{SIP}{SIP}--$(a)$ & \ref{D1}--$(a,0)$ & All $n$ & \hyperlink{Per}{Per}--$(a)$ & $F(n)^\times$ & SI & SI  & yes & yes &  SI \\
  \rowcolor{highlight} \hyperlink{SNAP}{SNAP}--$(a)$ & \ref{D3}--$(a)$ & All $n$ & \hyperlink{Per}{Per}--$(a)$ &  $F(n)^\times$ & SNAc & SNAc & yes & no & SNAc \\
  \rowcolor{highlight} \hyperlink{SNASI}{SNASI}--$(a)$ &\ref{D5}--$(a)$ & Even $n$ &\hyperlink{RRef}{RRef}--$(a)$ & $F(n)^\times$ & SI & SNAc & yes & no & --- \\
  \hyperlink{DihB}{DihB} & \ref{D2} & All $n$ & \hyperlink{Per}{Per}--$(1)$ & unique & 2R & 2R & no & yes & 2R \\
  \hyperlink{DihD}{DihD} &\ref{D4}  & Even $n$ & \hyperlink{RRef}{RRef}--$(1)$ & unique & 2R & SI & no & yes & --- \\
 \hyperlink{DihF}{DihF}  & \ref{D6} & Even $n$ & \hyperlink{GFPer}{GFPer}--$(1,n/2)$ & unique & 2R & SNAc & no & no & --- \\
        \hline 
    \caption[Types of $D_n$-symmetric knots for $n \ge 2$.]{Types of $D_n$-symmetric knots for $n\ge 2$. We have not named the three families of types that are not the type of a $D_n$-symmetric
  prime knot. These families may be referred to by the associated representation type, DihB, DihD, DihF.}
    \label{tab:D2nTypes}\\
\end{longtable}
\footnotetext[1]{A ``good diagram'' exists for a knot symmetry if and only if there is a global fixed point. See Remark \ref{rem:onDiagrams}.}
\end{landscape}
\clearpage
}

\restoregeometry

  
\section{Restrictions to subgroups}
\label{sec:Restrictions}

\subsection{Cyclic Case}
\label{sec:cycRes}

Suppose $(K, \alpha)$ is a $C_n$-symmetric knot of known type from Table \ref{tab:CnTypes}. Let $G$ be a subgroup of $C_n$. We list the possible types of $(K, \alpha|_G)$. The subgroups of $C_n$ are the groups $G_d = \langle \rho^d \rangle$, where $d \mid n$.

The possible restrictions are enumerated in Table \ref{tab:RestrictionOfCyclicTypes}. Note that $G_d$ may be cyclic of order $2$, in which case the dictionary between general $C_n$-symmetry types and $C_2$ types, found in Table \ref{tab:CnTypes}, may be used.\benw{check this table}
\begin{table}[h]
    \centering
 \begin{tabular}{|c|c|c|} \hline 
$C_n$-type & $d$ & $G_d$-type \\* \hline 
\hyperlink{GFPer}{GFPer}--$(a,b)$ & $a \neq 0$ and $b\neq 0$ in $\ZZ/(n/d)$ & \hyperlink{GFPer}{GFPer}--$(a,b)$ \\* 
\hyperlink{GFPer}{GFPer}--$(a,b)$ & otherwise  & \hyperlink{Per}{Per}--$(a+b)$ \\* \hline
\hyperlink{Per}{Per}--$(a)$ & all & \hyperlink{Per}{Per}--$(a)$ \\* \hline
\hyperlink{RRef}{RRef}--$(a)$ & even & \hyperlink{Per}{Per}--$(a)$ \\*
 \hyperlink{RRef}{RRef}--$(a)$ & odd & \hyperlink{RRef}{RRef}--$(a)$ \\* \hline
\end{tabular}
  \caption{Types of restrictions of $C_n$-symmetric knots to subgroups.}
  \label{tab:RestrictionOfCyclicTypes}
\end{table}

\subsection{Dihedral Case}
\label{sec:dihRes}

Suppose $(K,\alpha)$ is a $D_n$-symmetric knot of known type from Table \ref{tab:D2nTypes}. Let $G$ be a subgroup of
$D_n$. We determine the possible types of $(K,\alpha|_G)$. If $G = \langle \rho \rangle$, then the type can be read off Table \ref{tab:D2nTypes}, and therefore if $G$ is generated
by a power of $\rho$, then the type can be determined from Table \ref{tab:D2nTypes} and Table
\ref{tab:RestrictionOfCyclicTypes}.

Conversely, if $G$ is dihedral, then it is one of the dihedral subgroups of $D_n$:
\[ 
G_{d,r} = \langle \rho^d , \rho^r  \sigma \rangle, \quad \text{where } d \mid n, \text{ and } 0 \le r < d. 
\]
Here $G_{d,r}$ is dihedral of order $n/d$, and we allow $d = n$, in which case $G_{n,r}$ is cyclic of order $2$. In these cases, the dictionary between $D_1$-symmetry types and $C_2$-symmetry types, found in Table \ref{tab:D2nTypes}, may be used. The full description of the $G_{d,r}$-symmetric types is given in Table \ref{tab:RestrictionOfDihedralTypes}.
\begin{table}[h]
    \centering
\begin{tabular}{|c|c|c|c|}  \hline
  $D_n$-type & $d$ & $r$ & $G_{d,r}$-type \\ \hline \hline
\hyperlink{SIFP}{SIFP}--$(a,b)$ & $a \neq 0$ and $b\neq 0$ in $\ZZ/(n/d)$ & all & \hyperlink{SIFP}{SIFP}--$(a,b)$ \\
 \hyperlink{SIFP}{SIFP}--$(a,b)$ & otherwise& all & \hyperlink{SIP}{SIP}--$(a+b)$ \\\hline
   \hyperlink{SIP}{SIP}--$(a)$ & all & all & \hyperlink{SIP}{SIP}--$(a)$ \\\hline
  \hyperlink{SNAP}{SNAP}--$(a)$ &  all & all & \hyperlink{SNAP}{SNAP}--$(a)$ \\ \hline
  \hyperlink{SNASI}{SNASI}--$(a)$ & even & even & \hyperlink{SIP}{SIP}--$(a)$ \\ 
  \hyperlink{SNASI}{SNASI}--$(a)$ & even & odd & \hyperlink{SNAP}{SNAP}--$(a)$ \\
  \hyperlink{SNASI}{SNASI}--$(a)$ & odd & all & \hyperlink{SNASI}{SNASI}--$(a)$ \\ \hline
  \ref{D2} & all & all & \ref{D2} \\  \hline
  \ref{D4} & even & even & \ref{D2} \\
  \ref{D4} & even & odd & \hyperlink{SIP}{SIP}--$(1)$ \\
  \ref{D4} & odd & all & \ref{D4} \\ \hline
 \ref{D6} & even & even & \ref{D2} \\
 \ref{D6} & even & odd & \hyperlink{SNAP}{SNAP}--$(1)$ \\
  \ref{D6} & odd & all & \ref{D6} \\ \hline
\end{tabular}
   \caption{Types of restrictions of $D_n$-symmetric knots to dihedral subgroups.}
  \label{tab:RestrictionOfDihedralTypes}  
\end{table}


\section{Examples}
\label{sec:examples}

\subsection{Torus knots}
\label{sec:torus}
Let $p, q$ be relatively prime integers, both greater than $1$. Recall that $S^1$ is a Lie group. We define the $p,q$-torus knot by an explicit formula:
\[ T_{p,q} : S^1 \to S^3, \quad T_{p,q}(s) =\frac{1}{\sqrt{2}} (s^p, s^q),\]
where $s^p$ denotes the $p$-fold power of $s$ in $S^1$.


The knot $T_{p,q}$ admits an $\Og(2)$-symmetric structure. To present this structure, we identify $\Og(2)$ with the group of maps $\C\to \C$ of either of the following two forms:
\begin{enumerate}
\item $\hat \omega : z \mapsto \omega z$, where $\omega \in S^1 \subset \C^\times$, or
\item $\hat \omega^*: z \mapsto \omega \bar z$, where $\omega \in S^1 \subset \C^\times$.
\end{enumerate}

Allow $\SO(2)$ to act on $\C^2 = \RR^4$ by means of the homomorphism 
\[ \alpha_{p,q}(\hat \omega)(z_1, z_2) = (\omega^p z_1, \omega^q z_2) \]
and extend this to an action of $\Og(2)$ by
\[ \alpha_{p,q}(\hat \omega^\ast)(z_1, z_2) = (\omega^p \bar z_1, \omega^q \bar z_2). \]

Restricting to the knot $T_{p,q}$, we calculate
\[ \alpha_{p,q}(\hat \omega)T_{p,q}(s) = (\omega^p s^p, \omega^qs^q) = ((\omega s)^p, (\omega s)^q) = T_{p,q}(\hat \omega s)\]
and
\[ \alpha_{p,q}(\hat \omega^*)T_{p,q}(s) =(\omega^p \bar s^p , \omega^q \bar s^q) = ((\omega \bar s)^p ,(\omega \bar s)^q)= T_{p,q}(\hat \omega^*s).\] Therefore, $T_{p,q}$ is invariant under the $S^1$-action.

It is helpful to view the $\Og(2)$-symmetric structure on $T_{p,q}$ loosely as a limiting case of the \hyperlink{SIFP}{SIFP}-structure as $n \to \infty$. For every full rotation of the knot $T_{p,q}$, there is a $p$-fold full rotation of $S^1_{xy}$ and a $q$-fold full rotation of $S^1_{wz}$. That is, the classification of the action is as of type \hyperlink{SIFP}{SIFP}--$(p,q)$.

By restricting the $\Og(2)$-symmetric structure to the subgroup $D_n \subset \Og(2)$, we arrive at the following result.

\begin{proposition} \label{pr:torusKnotsHaveDih1Type}
    Let $p, q$ be relatively prime integers, both greater than $1$. Let $n$ be any positive integer. Let $\bar p, \bar q$ denote the reductions of $p,q$ modulo $n$. If $\bar p \neq 0$ and $\bar q \neq 0$, then the torus knot $T_{p,q}$ admits the structure of a $D_n$-symmetric knot of type \hyperlink{SIFP}{SIFP}--$(\bar p,\bar q)$. If $\bar q = 0$, then $T_{p,q}$ admits the structure of $D_n$-symmetric knot of type \hyperlink{SIP}{SIP}--$(\bar p)$
\end{proposition}

\begin{corollary} \label{cor:AllDih1Cyc1TypesArise}
    Let $n$ be an integer and suppose $a, b \in \ZZ/(n)$ generate the unit ideal. If $a \neq 0$ and $b\neq 0$, then there exists a torus knot, and in particular a prime knot, with a $D_n$-symmetric structure of type \hyperlink{SIFP}{SIFP}--$(a,b)$ and a $C_n$-symmetric structure of type \hyperlink{GFPer}{GFPer}--$(a,b)$. If $b =0$, then there exists a torus knot with a $D_n$-symmetric structure of type \hyperlink{SIP}{SIP}--$(a)$ and a $C_n$-symmetric structure of type \hyperlink{Per}{Per}-$(a)$.
\end{corollary}
\begin{proof}
    The $C_n$-symmetric structure may be obtained by restricting the $D_n$-symmetric structure.

    For the $D_n$-symmetric structure, we argue as follows. Lift $a,b \in \ZZ/(n)$ to positive
    integers $A, B$.  For some positive integer value of $x$, the integer $B +xn$ is relatively
    prime to $A$\benw{the argument establishing this has been omitted, but it's true}. Set $p=A$ and
    $q = B+xn$, so that $\bar p = a$ and $\bar q = b$. Then the torus knot $T_{p,q}$ has an action
    of type \hyperlink{SIFR}{SIFR}--$(a,b)$ or \hyperlink{SIR}{SIR}--$(a)$.
\end{proof}

\subsection{Amphichiral knots}
\label{sec:exAmphichiral}

\begin{definition}
  Let $K$ be a $G$-symmetric knot of type \hyperlink{SNASI}{SNASI}. Then a \emph{symmetric diagram} $D$ for $K$ is a
  diagram for $K$ on a $G$-invariant sphere, where the projected immersed curve has a rotational symmetry corresponding to the
  rotoreflection of $K$ and a reflection symmetry corresponding to the strong inversions of $K$. See for example Figure
  \ref{fig:Dih5-1}.
\end{definition}

\begin{proposition} \label{pr:snasi}
For each $n \in \NN$ and each $a \in F(n)^\times$, there is a $D_n$-symmetric hyperbolic 
knot of type \hyperlink{SNASI}{SNASI}-$(a)$, and a $C_n$-symmetric hyperbolic knot of type
\hyperlink{RRef}{RRef}--$(a)$.
\end{proposition}
\begin{proof}
 Fix $n \in \ZZ$. If $K$ has a $D_n$-symmetry of type \hyperlink{SNASI}{SNASI}--$(a)$, then by restricting to the
 subgroup $C_n \subset D_n$, we obtain a knot with $C_n$-symmetry of type \hyperlink{RRef}{RRef}--$(a)$. Therefore it
 suffices to address the dihedral case of this proposition.

 We begin by constructing a knot $K_1^n$ with a $D_n$-symmetry of type \hyperlink{SNASI}{SNASI}-$(1)$. We take the
 tangle-sum of $n/2$ copies of the tangle in Figure \ref{fig:tangle}. For the case $n = 6$, see Figure
 \ref{fig:Dih5-1}. Observe that $K_1^n$ has a rotoreflection symmetry and $n/2$ strong inversions, so that it is indeed
 of type \hyperlink{SNASI}{SNASI}. The linking number of $K_1^n$ with the axis of rotoreflection (meeting the diagram at
 the central point) is $1$, so that $K_1^n$ has type \hyperlink{SNASI}{SNASI}-(1).

\begin{figure}
\includegraphics[width=0.3\textwidth]{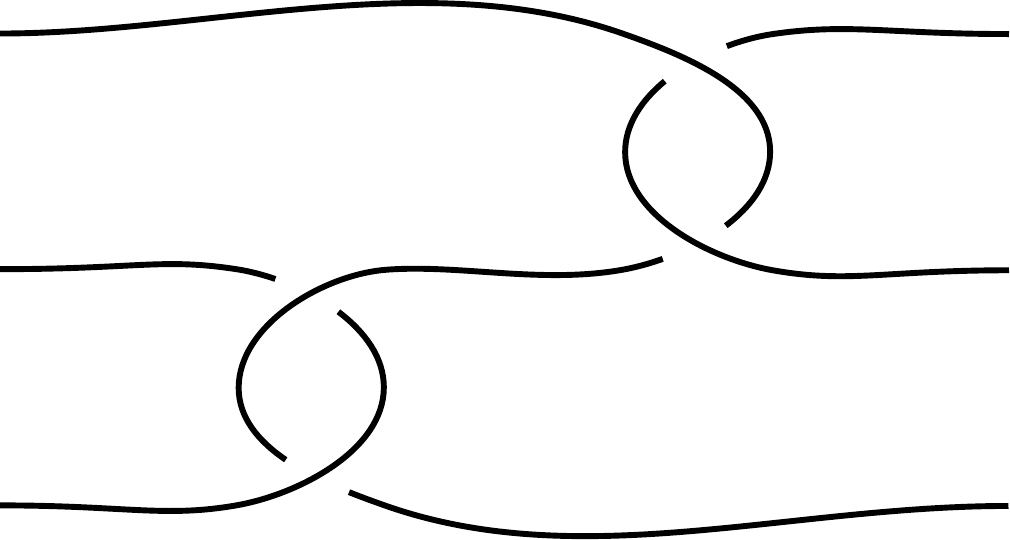}
\caption{The tangle $T$ used in the construction of examples with a \protect\hyperlink{SNASI}{SNASI}-structure.}
\label{fig:tangle} 
\end{figure}

To construct the more general case of a knot $K_a^n$ of type \hyperlink{SNASI}{SNASI}-$(a)$, we will perform the
following modifications to our diagram for $K_1^n$. We must assume that $n$ is even, and since $\gcd(a,n) = 1$, we
deduce that $a$ is odd.
\begin{enumerate}
    \item We replace each strand in our symmetric diagram for $K_1^n$ with $a$ parallel strands, using the blackboard
      framing, so as not to introduce any twists in the diagram. The result is an $a$-component link.
    \item In a neighborhood of each of the $2n$ points in the diagram that correspond to fixed points of a strong
      inversion on $K_1^n$, we replace the $a$ parallel strands with the tangle described in Figure \ref{fig:permutation}, or its mirror, alternating as we cycle through the $2n$ replacements. 
    \item We change the crossings of the diagram to be alternating. 
\end{enumerate}

\begin{figure}
\begin{overpic}[width = 200pt, grid = false]{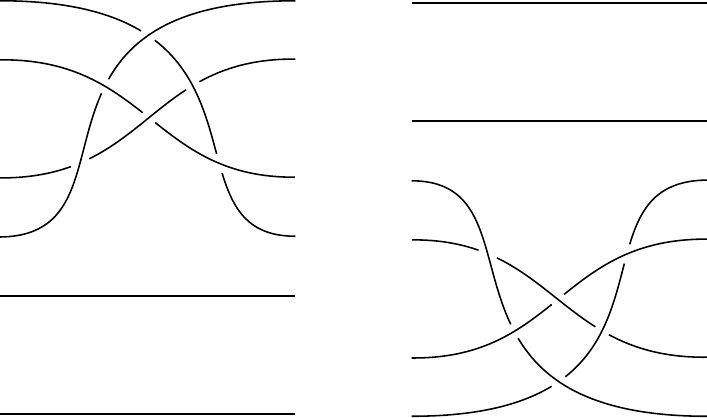}
\put (5,40) {\vdots}
\put (35,40) {\vdots}
\put (95,15) {\vdots}
\put (62,15) {\vdots}
\put (20,7) {\vdots}
\put (78,48) {\vdots}

\put (-5,58) {$1$}
\put (-11,24) {$\frac{a+1}{2}$}
\put (-5,0) {$a$}
\end{overpic}
    \caption[A schematic for an alternating tangle $T'$.]{A schematic for an alternating tangle $T'$ on $a$ strands
      (left), and its mirror $mT'$ (right). Note that each tangle is symmetric under rotation around a vertical axis.}
    \label{fig:permutation}
\end{figure}
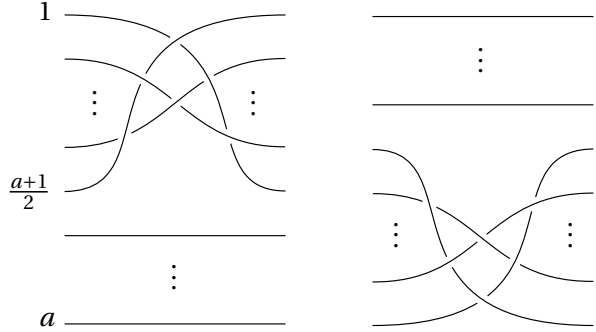

The result is a diagram that retains the strong inversions (since the tangle in Figure \ref{fig:permutation} is
symmetric) and the rotoreflections of $K_1^n$, but that now has linking number $a$ with the axis perpendicular to the
diagram through the rotoreflection point. See Figure \ref{fig:Dih5-3} for an example of an example of a
$D_{8}$-symmetric knot of type \hyperlink{SNASI}{SNASI}-(3) constructed this way.

It remains to check that our procedure always produces a diagram for a hyperbolic knot, rather than some other kind of link. To see that it is a knot, note that the tangle
$T'$ in Figure \ref{fig:permutation} permutes the strands according to the cycle $(1,a-1)(2,a-2)(3,a-3)\dots$, and its mirror
$mT'$ permutes them according to the cycle $(2,a)(3,a-1)\dots$. Since $K_a^n$ is constructed with $n$ copies of $T'$
and $n$ copies of $mT'$ interlaced, we can verify that $K_a^n$ is a knot by composing these permutations in the same
order and verifying that the result has a single orbit. Composing $T'$ with $mT'$ gives us the cyclic permutation
$(1,3,5,\dots,a,2,4,6,\dots,a-1)$ of order $n$. Since gcd$(a,n) = 1$, the $n$-th power of this permutation is also
a cyclic permutation of order $n$. In particular, the strands are joined into a single component by these tangles to
form a knot.

Finally, $K_a^n$ is represented by an alternating diagram without any nugatory crossings, and therefore it is hyperbolic
by \cite[Cor.~2]{Menasco1984}.
\end{proof}

\begin{figure}
    \includegraphics[width=0.35\textwidth]{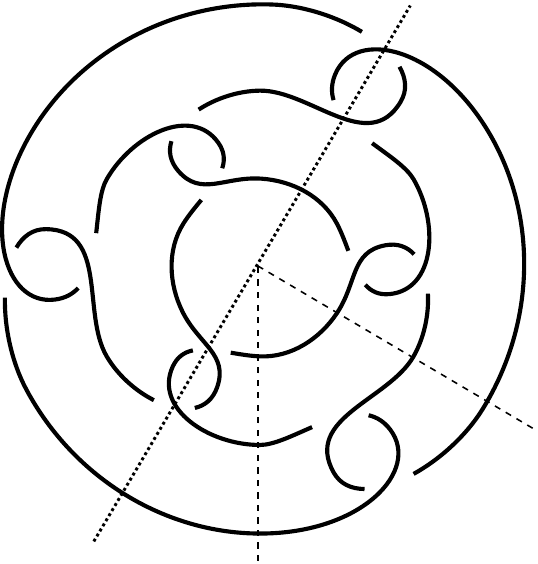}
    \caption[An alternating knot with a $D_6$ \protect\hyperlink{SNASI}{SNASI}-(1) symmetry.]{An alternating knot with a $D_6$ \protect\hyperlink{SNASI}{SNASI}-(1) symmetry. The generator $\rho$ of the order 6 cyclic symmetry can be seen as a $2\pi/3$ rotation followed by a point reflection across the centre point of the diagram. A wedge-shaped fundamental domain for this action is shown between the dashed lines. The order 2 generator $\sigma$ is the strong inversion given by $\pi$-rotation around the dotted line.
    This appears as entry $12a_{1202}$ in the \cite{knotinfo} tables.}
    \label{fig:Dih5-1}
\end{figure}

\begin{figure}
\includegraphics[width=0.45\textwidth]{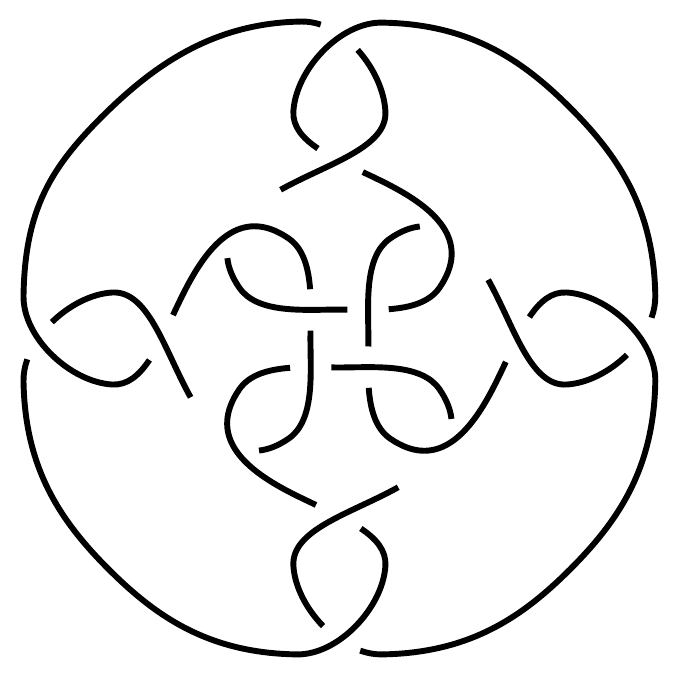}\hfill
\includegraphics[width=0.45\textwidth]{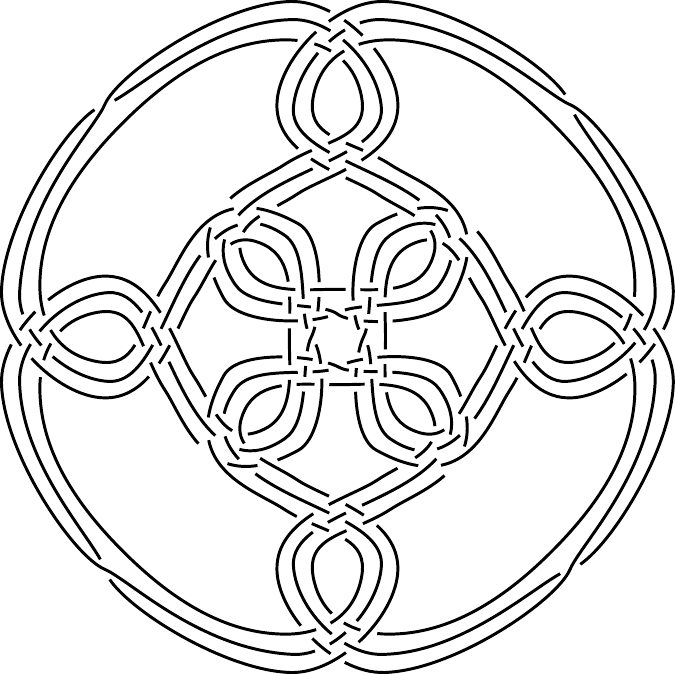}
    \caption[Two $D_8$-symmetric knots of types \protect\hyperlink{SNASI}{SNASI}-(1) and \protect\hyperlink{SNASI}{SNASI}-(3).]{A $D_{8}$-symmetric knot of type \protect\hyperlink{SNASI}{SNASI}-(1) (left), and of type \protect\hyperlink{SNASI}{SNASI}-(3) (right), generated as a modification of the knot on the left. In each case the generator $\rho$ acts by a rotoreflection, and the generator $\sigma$ acts by a strong inversion. The element $\rho \sigma$ is strongly negative amphichiral.}
    \label{fig:Dih5-3}
\end{figure}

\subsection{Maximal Symmetries}
For the examples in Section \ref{sec:exAmphichiral}, we can verify that the constructed SNASI symmetries are the full symmetry group since the knots are all alternating. Specifically, any symmetry of an alternating knot is visible, modulo a sequence of flypes, in any alternating diagram. We omit a complete description of these visibility results, but see \cite{EQW23} for rotoreflections, \cite{CQ23} for periodic and freely periodic symmetries of order larger than 2, and \cite{Boyle21} for order-2 symmetries. Since our diagrams are alternating and do not contain any nontrivial flypes, any symmetry of the underlying knot must be visible in the diagram. It is then straightforward to verify that the diagrams have no additional symmetry. 

Using the same idea, we can construct hyperbolic knots with maximal SIP, SNAP, Per, and RRef symmetries of each possible order; see Figures \ref{fig:maximal_symmetries1} and \ref{fig:maximal_symmetries2} for examples modified from Figure \ref{fig:Dih5-3}.

The remaining symmetry types, SIFP and GFPer, cannot be seen as symmetries of alternating knots when the free period has even order; see \cite[Corollary 4.5]{Boyle19}. Hence it will require understanding the symmetries of some family of non-alternating knots to resolve Conjecture \ref{conj:allsymshyperbolic}.

\begin{figure}
\includegraphics[width=0.45\textwidth]{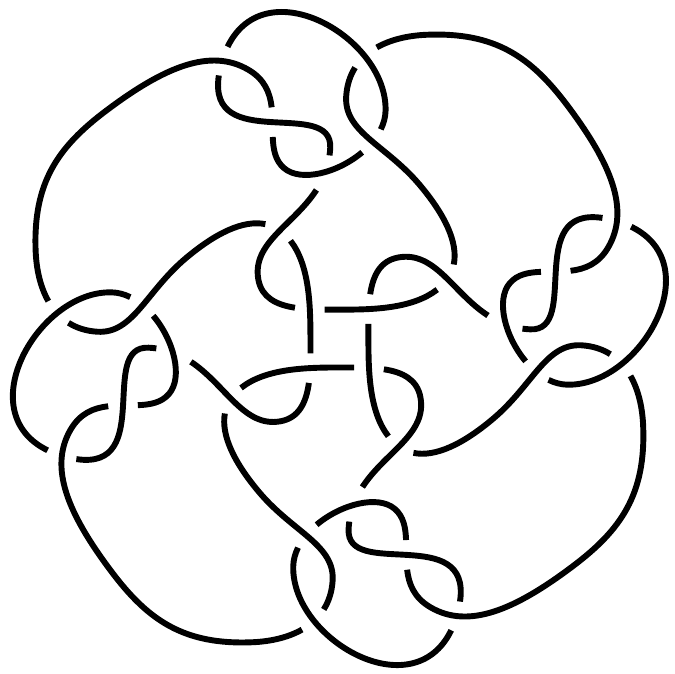}\hfill
\includegraphics[width=0.45\textwidth]{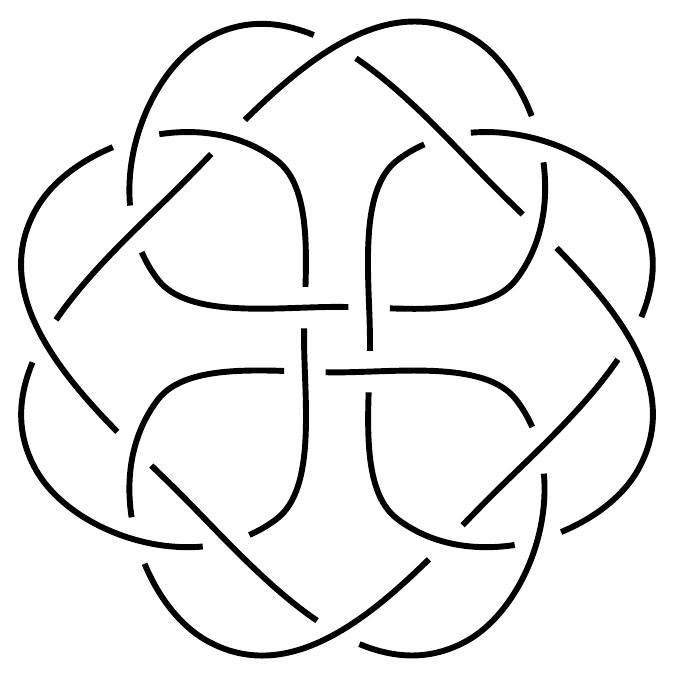}
\caption{A hyperbolic knot with full symmetry group $C_4$ of type \protect\hyperlink{Per}{{Per}} (left) and a hyperbolic knot with full symmetry group $D_4$ of type \protect\hyperlink{SIP}{SIP} (right).}
\label{fig:maximal_symmetries1}
\end{figure}

\begin{figure}
\includegraphics[width=0.45\textwidth]{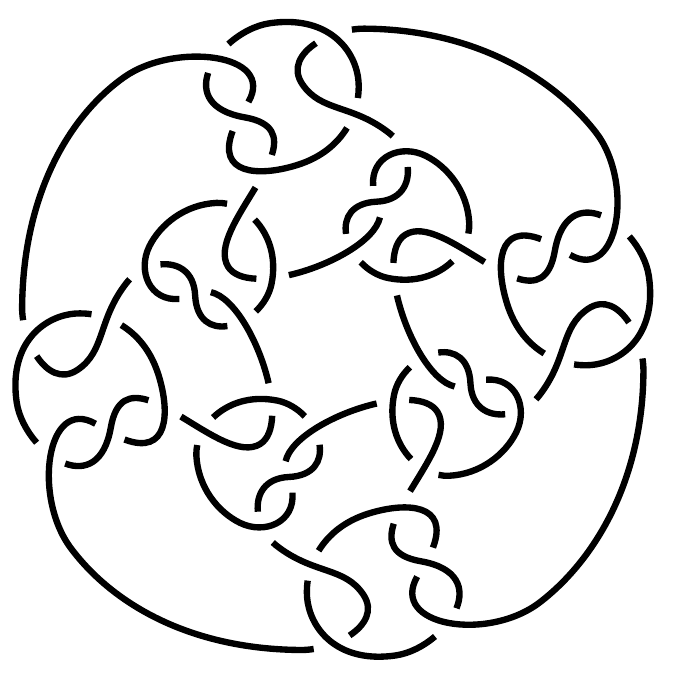}\hfill
\includegraphics[width=0.45\textwidth]{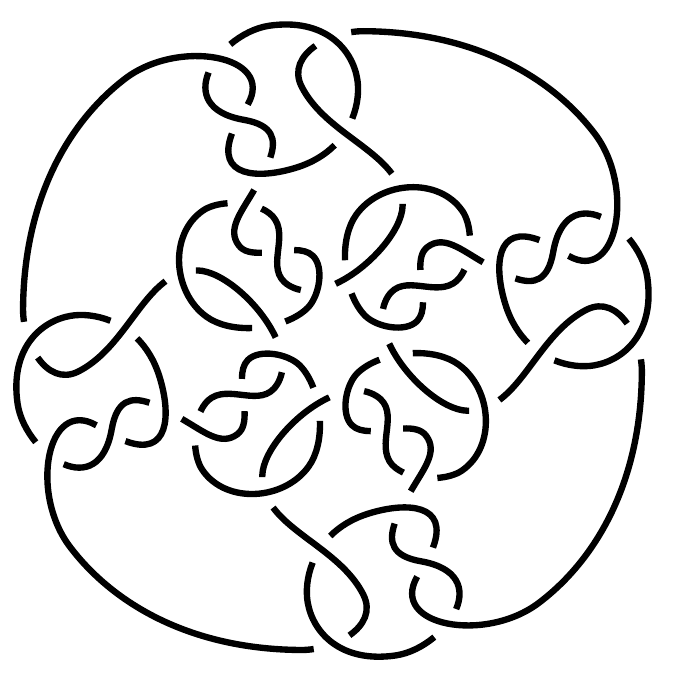}
\caption{A hyperbolic knot with full symmetry group $C_8$ of type \protect\hyperlink{RRef}{{RRef}} (left) and a hyperbolic knot with full symmetry group $D_4$ of type \protect\hyperlink{SNAP}{SNAP} (right).}
\label{fig:maximal_symmetries2}
\end{figure}

\appendix

\section{Conjugacy of diffeomorphisms of \texorpdfstring{$S^1$}{S¹}}

The results of this appendix are presumably well known to experts. One can find statements for subgroups of
$\Diff^+(S^1)$ in various places in the literature, but those for the general case of subgroups of $\Diff(S^1)$ are
harder to locate.

\subsection{Conjugacy of homeomorphisms}

\begin{proposition} \label{pr:orientedConjugateDihedral}
    Suppose $G$ is a finite subgroup of $\Homeo(S^1)$. Then there exists some $h \in \Homeo^+(S^1)$ so that $hGh^{-1} \subset \Og(2)$.
\end{proposition}
The proof is that of Proposition \cite[Prop.~4.1]{Ghys2001}.\benw{This argument requires ideas from outside my comfort zone.}
\begin{proof}
    Let $dx$ denote the Lebesgue measure on $S^1 = \RR/(2 \pi \ZZ)$. As in \cite[Prop.~4.1]{Ghys2001}, there exists a measure 
    \[ \mu = \frac{1}{|G|} \sum_{g \in G} g_* dx \]
    that is $G$-invariant, has no atom, is nonzero on nonempty open sets, and satisfies $\mu(S^1) = 2\pi$. Again as in \cite[Prop.~4.1]{Ghys2001}, there exists an orientation-preserving homeomorphism $h \colon S^1 \to S^1$ so that $h_* \mu = dx$. Then the subgroup $hGh^{-1}$ acts on $S^1$ by homeomorphisms that preserve the Lebesgue measure, i.e., by isometries. This shows that $hGh^{-1}$ is a subgroup of $\Og(2)$. 
\end{proof}

\begin{corollary}
    If $G$ is a finite subgroup of $\Homeo(S^1)$, then $G$ is cyclic or dihedral.
\end{corollary}

\subsection{Lifting to \texorpdfstring{$\RR$}{ℝ}}
\label{sec:liftingtoR}

The following is largely adapted from \cite[Exp.~I]{Herman1979}. Let $\pi \colon \R \to S^1$ be the map $x \mapsto (\cos(2\pi x) , \sin(2 \pi x))$. We will frequently refer to $\pi(x)$ as ``$x$''.

Suppose $\phi\colon S^1 \to S^1$ is a map. By covering-space theory, there exists a map $f\colon \RR \to \RR$ making the diagram
\begin{equation}
\begin{tikzcd} \RR \dar{\pi} \rar{f} & \RR \dar{\pi} \\ S^1 \rar{\phi} &  S^1
\end{tikzcd}  \label{eq:commsquare}  
\end{equation}
commute. In fact, $f$ is uniquely defined by specifying $f(0) \in \pi^{-1}(\phi(0))$. We call any such $f$ a \emph{lift} of $\phi$.

A map $f\colon \RR \to \RR$ is a lift of some map $S^1 \to S^1$ if and only if $\pi\circ f (x)$ depends only on $x \pmod{\ZZ}$. This implies that $f$ is a lift of some map $\phi\colon S ^1 \to S^1$ if and only if there exists some integer $d$ such that
\[ f(x+1) - f(x) = d , \quad \forall x \in \RR. \]
The integer $d$ is the degree of $\phi$, and we will loosely refer to it as the ``degree'' of $f$.

Let us write $D_{\pm}(S^1)$ for the subgroup of the group of diffeomorphisms $\RR \to \RR$ (under composition, endowed
with the topology of uniform convergence) consisting of maps satisfying
\[ f(x+1)-f(x)= \pm 1, \quad \forall x \in \RR, \]
and $D(S^1)$ for the subgroup of $D_{\pm}(S^1)$ where the degree is $1$. 

The elements of $D_{\pm}(S^1)$, being diffeomorphisms, are strictly monotone and it follows that the elements of $D(S^1)$ are strictly increasing while the elements of $D_{\pm}(S^1)  \sm D(S^1)$ are strictly decreasing.

If $f \in D_{\pm}(S^1)$, then $f$ defines a diffeomorphism $\bar f$ of $S^1$, as follows: for any $y \in S^1$, choose a
lift $\tilde y \in \RR$. Define $\bar f(y) = f(\tilde y)$, which is independent of the choice of $\tilde y$.
There exists a commutative square of homomorphisms
\[ \begin{tikzcd}
    D(S^1) \arrow[r, hook] \arrow[d, two heads] & D_{\pm}(S^1) \arrow[d, two heads] \\ \Diff^+(S^1) \arrow[r, hook] & \Diff(S^1).
\end{tikzcd}\]

\subsection{Rotation number}

\begin{definition} \label{def:rtn}
    If $f \in D(S^1)$, then the \emph{rotation number} is defined to be the real number
    \[ \rtn(f) = \lim_{k \to \infty} \frac{ f^k(x) - x }{k}. \]
    This definition is due to Poincar\'e. It is shown in \cite[II Prop.~2.3]{Herman1979} that this quantity does not depend on $x$.
    
    If $\phi \in \Diff^+(S^1)$, then $\rtn(\phi)$ may be defined in $\RR/\ZZ = S^1$, since two choices of lift $\tilde \phi_0$ and $\tilde \phi_1$ of $\phi$ lead to values of $\rtn$ that differ by an integer.    
\end{definition}

In either case, the rotation number satisfies $\rtn(f_1 f_2) = \rtn(f_1) \rtn(f_2)$ if $f_1$ and $f_2$ commute. The
integer-rotations $\id+m$, where $m\in\ZZ$, commute with all $f \in D(S^1)$, and we see that $\rtn(f + m) = \rtn(f) + m$
for $f \in D(S^1)$.

If $\phi \in \Diff^+(S^1)$ is periodic of finite order $n$, then $0=\rtn(\phi^n) = n \rtn(\phi)$, so that $\rtn(\phi) =
a/n$ for some integer $a$. If $a=0$, then it is possible to lift $\phi$ to $f \in D(S^1)$ for which $\rtn(f) = 0$ and
$f^n = \id_\R$ and then applying \cite[IV Cor.~5.3.2]{Herman1979} we see that $f = \id_\R$ so that $\phi=\id_{S^1}$. As a consequence, $\rtn$ sets up an injective homomorphism $\rtn \colon \langle \phi \rangle \to \RR/\ZZ$.

The rotation number is a conjugation invariant in $\Diff^+(S^1)$, as is proved in \cite[II~Prop.~2.10]{Herman1979}. Here
we prove a slight extension of this fact.
\begin{proposition} \label{pr:rotDegreeDConjugacy}
  Suppose $f, g, h \colon S^1 \to S^1$ are maps such that $h \circ f = g \circ h$ and $f,g \in \Diff^+(S^1)$. Then 
\[ \deg(h) \rtn(f)  = \rtn(g). \] 
\end{proposition}
\begin{proof}
  Choose lifts of $f,g,h$ to smooth maps $\tilde f, \tilde g, \tilde h\colon \RR \to \RR$ so that $\tilde f$ and $\tilde g$ are
  diffeomorphisms for which $\tilde f - \id_\RR$ and $\tilde g - \id_\RR$ are periodic with period $1$. There exists some constant integer $N$ so that
  \[ \tilde g \circ \tilde h - \tilde h \circ \tilde f  = N \]
  and therefore, by an induction argument, we calculate
  \begin{equation*}
    \tilde g^k \circ \tilde h - \tilde h \circ \tilde f^k = k N \label{eq:2}
  \end{equation*}
  for all integers $k$. Rearrange, subtract $\tilde h$ from both sides and divide  by $k$ to give
  \begin{equation*}
    \frac{  \tilde g^k \circ \tilde h - \tilde h }{k} = N + \frac{ \tilde h \circ \tilde f^k - \tilde h}{k}.
  \end{equation*}
  Now express $\tilde h = \deg(h)\id_\RR + \psi$ where $\psi$ is periodic of period $1$, and in particular
  bounded. Applying this formula for $\tilde h$ on the right hand side gives us
 \begin{equation*}
    \frac{ \tilde g^k \circ \tilde h - \tilde h }{k} = N + \frac{ \deg(h) \tilde f^k + \psi\circ \tilde f^k- \deg(h)
      \id_\RR - \psi }{k} = N + \frac{ \deg(h) (\tilde f^k - \id_\RR)}{k} + \frac{\psi\circ \tilde f^k - \psi }{k}.
  \end{equation*}
  Applied to any $x$, the limit as $k \to \infty$ on the left hand side is $\rtn(\tilde g)$, whereas the limit on the
  right hand side is $N + \deg(h) \rtn(\tilde f) + 0$.  Reducing modulo $\ZZ$ gives
  \[ \rtn(g) = \deg(h) \rtn(f)\]
  as required. 
\end{proof}

\subsection{Conjugacy in \texorpdfstring{$\Diff(S^1)$}{Diff(S¹)}}

  \begin{proposition} \label{pr:conjDiffCyclic} Suppose $g : [0,1] \to \Diff^+(S^1)$ is a continuous path where each
    $g_t = g(t)$ generates a cyclic subgroup $G_t$ of order $n$ of
    $\Diff^+(S^1)$. Then there exists a continuous $h: [0,1] \to \Diff^+(S^1)$ so that $h_t g_t h_t^{-1}$ is rotation by a
    $1/n$-turn.
  \end{proposition}
\begin{proof}
        It is possible to lift $g_t$ continuously to $\tilde g_t \in D(S^1)$, since $D(S^1)$ is a covering space of
     $\Diff^+(S^1)$. Observe that $\rtn(\tilde g_t) \equiv 1/n \pmod{\ZZ}$, and $\rtn(\tilde g_t)$ is constant as a
     function of $t$. Now \cite[IV Cor.~5.3.2]{Herman1979} applies to say that each $\tilde g_t$ is conjugate by a diffeomorphism 
     \begin{equation}
        \tilde h_t = \sum_{i=0}^{n-1}\frac{\tilde g^i_t}{n}\label{eq:1}
      \end{equation}
      to the addition-of-$\rtn(\tilde g_t)$ map (the conjugating element of \cite[IV Cor.~5.3.2]{Herman1979}
    differs from ours by a rotation, introduced for purposes unrelated to ours). By direct calculation, $\tilde h_t \in
    D(S^1)$. Therefore $g_t$ is conjugate by the reduction $h_t$ of $\tilde h_t$ to a rotation by a $1/n$-turn, and so
    $h_t$ is the desired element.
  \end{proof}

   \begin{proposition} \label{pr:conjDiffDihedral} Suppose $(g, s) : [0,1] \to \Diff^+(S^1) \times \Diff(S^1)$ is a continuous path where each
    $(g_t, s_t) = (g,s)(t)$ generates a dihedral subgroup $G_t$ of order $2n$ in $\Diff(S^1)$. Then there exists a
    continuous $h: [0,1] \to \Diff^+(S^1)$ so that $h_t g_t h_t^{-1}$ is
    rotation by a $1/n$-turn for all $t$ and $h_t s_t h_t^{-1}$ is a reflection of $S^1$ fixing $(1,0)$ and $(-1,0)$.
  \end{proposition}
  
  \begin{proof}
    We lift $g_t$ as in the proof of \ref{pr:conjDiffCyclic}, and let let $\tilde s$ be a lift of $s$ to an element of
    $D_{\pm}(S^1)$. Note that necessarily $\deg(s) = -1$ so that $\tilde s_t$ is strictly decreasing for all $t$. For
    all $t$, there exists a unique $p_t \in \RR$ so that $\tilde s(p_t) = p_t$.

    The group theory of $\langle \tilde g_t, \tilde s_t \rangle $ implies that $\tilde s_t^2 \equiv \id \pmod \ZZ$ and $\tilde s_t \tilde g_t \equiv \tilde g_t^{-1} \tilde s_t \pmod \ZZ$. In both cases, consideration of $p_t$ allows us to be precise about the integers involved:
    \begin{align*}
        \tilde s_t^2 & = \id_\R \\
        \tilde s_t \tilde g_t & = \tilde g_t^{-1} \tilde s_t.
    \end{align*}
    Rotation numbers ensure that $\tilde g_t^n(p_t) = p_t+ n\rtn(\tilde g_t)$. Applying $\tilde g_t^i$ to this, we see that
    $\tilde g_t^{i+n} = \tilde g_t^i + n\rtn(\tilde g_t)$ for all $i \in \ZZ$.
    
    Consider 
    \begin{equation}
        \tilde h_t = \sum_{i=0}^{n-1} \frac{ \tilde g_t^i -  \tilde g_t^i \circ \tilde s_t}{2n}.
        \label{eq:defH}
    \end{equation}
    This is a strictly increasing map $\tilde h_t \colon \RR \to \RR$ with everywhere positive derivative and $\tilde
    h_t(x+1) = \tilde h_t(x) + 1$ by direct calculation. This implies $\tilde h_t \in D(S^1)$. Certainly $\tilde h_t$
    depends continuously on $t$, since $\tilde g_t$ and $\tilde s_t$ do.
    
    By using the relations we developed above, we calculate
    \[ \tilde h_t  \tilde g_t = \rtn(\tilde g_t) + \tilde h_t  \]
    and 
    \[\tilde h_t \tilde s_t = - \tilde h_t. \]
    Reducing from $D(S^1)$ to $\Diff^+(S^1)$, we obtain an orientation-preserving diffeomorphism $h_t$ so that $h_tg_th_t^{-1}$
    is rotation by a $1/n$-turn and $h_ts_th_t^{-1}$ is the desired reflection.    
\end{proof}

%

\printbibliography
\end{document}
